\documentclass[11pt,oneside,reqno]{amsart}

\usepackage{xparse}
\usepackage{environ}
\usepackage{amsmath}
\usepackage[nobysame,initials]{amsrefs}
\usepackage{amssymb}
\usepackage[usesnames,svgnames]{xcolor}
\usepackage{lipsum}
\usepackage[mathscr]{euscript}
\DeclareMathAlphabet{\mathpzc}{OT1}{pzc}{m}{it}
\usepackage{comment}
\usepackage{rotating}
\usepackage{stackrel}
\usepackage{color,soul}
\usepackage{authblk}

\usepackage{pifont}
\usepackage{caption}
\usepackage{subcaption}
\usepackage{tikz,pgf}
\usetikzlibrary{decorations.pathreplacing}
\usetikzlibrary{plotmarks}
\usetikzlibrary{fadings,shapes.arrows,shadows}   
\usetikzlibrary{patterns}

\usetikzlibrary{calc}
  \pgfdeclarelayer{background}
  \pgfsetlayers{background,main}
\usepackage[stable]{footmisc}
\usepackage{hyperref}
\usepackage{relsize}
\usepackage{enumerate}
\usepackage{textcomp}

\usepackage[margin=0.82in]{geometry}
\usepackage{wrapfig}

\newtheorem{theorem}{Theorem}[section]
\newtheorem{proposition}[theorem]{Proposition}

\newtheorem{definition}[theorem]{Definition}
\newtheorem{problem}[theorem]{Problem}

\newtheorem{remark}[theorem]{Remark}

\newtheorem{lemma}[theorem]{Lemma}
\newtheorem{example}[theorem]{Example}

\newcommand{\aA}{{\mathcal A}}

\newcommand{\shi}[1]{{\sf Shi} (#1)}
\newcommand{\cT}{{\mathcal T}}
\newcommand{\ccT}[1]{{\mathcal T}_{#1}}
\newcommand{\x}{{\sf x}}

\newcommand{\av}[2]{ { Av_{#1}}(#2)}

\newcommand{\s}[1]{\scalebox{0.6}{#1}}

\newcommand{\RR}{{\mathbb R}}

\newcommand{\NN}{{\mathbb N}}

\newcommand{\fs}{\footnotesize}

\newcommand{\two}[1]{{\color{cyan} #1}}

\newcommand{\shiT}[1]{{\mathcal T}_{#1}}

\newcommand{\pa}{\pi}
\newcommand{\Dp}[1]{{\mathcal D}_{#1}}

\newcommand{\de}[2]{{d_{#1,#2}}}

\newcommand{\tnk}[2]{\mathcal{H}(#1,#2)}

\newcommand{\U}{{ U}}
\newcommand{\D}{{ D}}

\newcommand{\ndown}{\#\D}
\newcommand{\nup}{\#\U}

\newcommand{\col}[1]{C_{#1}}
\newcommand{\row}[1]{R_{#1}}
\newcommand{\ucov}[2]{{\mathcal UC}_{#1}(#2)}
\newcommand{\ucovv}[2]{\overline{\mathcal UC}_{#1}(#2)}
\newcommand{\lcov}[2]{{\mathcal L\mathcal C}_{#1}(#2)}

\usetikzlibrary{fadings,shapes.arrows,shadows}   

\newcommand{\reqnomode}{\tagsleft@true}
\newcommand{\leqnomode}{\tagsleft@false}

\newcommand{\dP}[1]{\mathscr{D}_{#1}}
\definecolor{azu}{rgb}{0.0, 0.5, 1.0}

\newcommand{\tf}[1]{T^{#1}_{
    \begin{tikzpicture}[scale=0.09]
    \draw[fill=black] (0,0) circle (12pt);
    \draw[fill=black] (0,1) circle (12pt);
    \draw[fill=black] (1,1) circle (12pt);
    \end{tikzpicture}}}
\newcommand{\te}[1]{ T^{#1}_{
    \begin{tikzpicture}[scale=0.09]
    \draw (0,0) circle (12pt);
    \draw (0,1) circle (12pt);
    \draw (1,1) circle (12pt);
    \end{tikzpicture}}}

\newcommand{\tv}[1]{ T^{#1}_{
    \begin{tikzpicture}[scale=0.09]
    \draw[fill=black] (0,0) circle (12pt);
    \draw[fill=black] (0,1) circle (12pt);
    \draw (1,1) circle (12pt);
    \end{tikzpicture}}}

\newcommand{\tor}[1]{ T^{#1}_{
    \begin{tikzpicture}[scale=0.09]
    \draw (0,0) circle (12pt);
    \draw[fill=black] (0,1) circle (12pt);
    \draw[fill=black] (1,1) circle (12pt);
    \end{tikzpicture}}}

\newcommand{\tg}[1]{ T^{#1}_{
    \begin{tikzpicture}[scale=0.09]
    \draw (0,0) circle (12pt);
    \draw (1,1) circle (12pt);
    \draw[fill=black] (0,1) circle (12pt);
    \end{tikzpicture}}}

\NewDocumentCommand\refi{m}{\eqref{Item:#1}}
\NewEnviron{eq*}{%
\begin{equation*}\begin{split}
  \BODY
\end{split}\end{equation*}
}
\NewEnviron{eq}[1]{%
\begin{equation}\label{Equation:#1}\begin{split}
  \BODY
\end{split}\end{equation}
}
\newcommand{\oeis}[1]{\href{http://oeis.org/#1}{#1}}

\newcommand*{\tikzarrow}[2]{%
  \tikz[
  baseline=(A.base),             
  font=\footnotesize\sffamily    
  ]
  \node[
  single arrow,                  
  single arrow head extend=5pt,  
  inner sep=1.5pt,                 
  top color=white,               
  bottom color=#1,               
  drop shadow                    
  ] (A) {#2};%
}

\title{Patterns in Shi tableaux and Dyck paths}

\author{Myrto Kallipoliti}
\address{Faculty of Mathematics, University of Vienna, Vienna, Austria}
\email{myrto.kallipoliti@univie.ac.at}

\author{Robin Sulzgruber}
\address{Deptartment of Mathematics and Statistics, York University, Toronto, Canada}
\email{rsulzg@yorku.ca}

\author{Eleni Tzanaki}
\address{Department of Mathematics and Applied Mathematics, University of Crete, Heraklion, Greece}
\email{etzanaki@uoc.gr}

\allowdisplaybreaks

\author{Myrto Kallipoliti}
\author{Robin Sulzgruber}
\author{Eleni Tzanaki}

\begin{document}
 \maketitle 
\begin{center}
{Myrto Kallipoliti}, {Robin Sulzgruber} and {Eleni Tzanaki}
\end{center}

\begin{abstract}
  Shi tableaux are special binary fillings of certain  Young diagrams
  which arise in the study of Shi hyperplane arrangements related to 
   classical root systems. For type $A$, the set $\shiT{}$
   of Shi tableaux naturally coincides  with the set of Dyck paths, 
    for which various notions of patterns have been introduced and studied over  the years.      In this paper we define  a notion of pattern occurrence in $\shiT{}$
   which, although it can be regarded as
   a pattern on Dyck paths, it  is motivated by the 
    underlying  geometric structure of the tableaux. 
  Our main goal in this work is to  study the poset of Shi tableaux defined by 
   pattern-containment. More precisely,  
  we determine explicit formulas for upper and lower covers for each 
  $T\in\shiT{}$, we consider pattern avoidance for 
  the  smallest non-trivial tableaux (size 2) and generalize these results to     certain tableau of larger size.  
  We conclude with open problems and possible future directions. 
  \end{abstract}

\section{Introduction}

The investigation of patterns in families of discrete objects is an active
topic in Combinatorics, with connections to various areas in Mathematics, 
as well as  other fields such as Physics, Biology,
Sociology and Computer science
\cite{pp-phylogeny,kitaev-book-11,lackner-17}.
Generally speaking, the notion of pattern occurrence or pattern avoidance can be 
described as the presence or, respectively, absence of a substructure inside a larger 
structure.
Patterns were first considered for permutations: an occurrence of a pattern $\sigma$
in a permutation $\pi$ is a subword of $\pi$ whose letters appear in the same relative
order  as those in $\sigma$.
For instance, the permutation $132$ occurs as a pattern in $32514$ since the 
subsequence $254$ (among others) is ordered in the same way as $132$ (see 
Figure~\ref{fig:perm} for an illustration using  matrix representations).
\begin{figure}[h!]
	\begin{center}
		\begin{tikzpicture}[scale=0.4]
		\begin{scope}[xshift=-10cm]
		\draw[step=1.0,black,thin] (1,-3) grid (6,2);
		\node at(3.5,3){$32514$};
		\node at(0.5,1.5){\fs$\tt5$};
		\node at(3.5,1.5){$\bullet$};
		\node at(0.5,0.5){\fs$\tt4$};
		\node at(4.5,-2.5){$\bullet$};
		\node at(0.5,-0.5){\fs$\tt3$};
		\node at(1.5,-0.5){$\bullet$};
		\node at(0.5,-1.5){\fs$\tt2$};
		\node at(2.5,-1.5){$\bullet$};
		\node at(0.5,-2.5){\fs$\tt1$};
		\node at(1.5,-3.5){\fs$\tt 1$};
		\node at(2.5,-3.5){\fs$\tt 2$};
		\node at(3.5,-3.5){\fs$\tt 3$};
		\node at(4.5,-3.5){\fs$\tt 4$};
		\node at(5.5,-3.5){\fs$\tt 5$};
		\node at(5.5,0.5){$\bullet$};
		\node at(-1.2,0){$\pi:$};
		\draw[->,line width=1pt,color=black,opacity=0.5] (7,0) to[bend left,out=20](9,0); 
		
		\end{scope}
		
		\begin{scope}[xshift=0cm]
		\draw[lightgray](1,-3) grid (2,2);
		\draw[lightgray](2.01,-3) grid (6,-2);
		\draw[lightgray](2.01,-1) grid (6,0);
		\draw[lightgray](4,1) -- (5,1);
		\draw[lightgray](4,2) -- (5,2);
		\draw[thick](2,-2) grid (4,-1);
		\draw[thick](2,0) grid (4,2);
		\draw[thick](5,0) grid (6,2);
		\draw[thick](5,-2) grid (6,-1);
		
		\node at(0.5,1.5){\fs$\tt5$};
		\node at(0.5,0.5){\fs$\tt4$};
		\node at(0.5,-0.5){\fs$\tt3$};
		\node at(0.5,-1.5){\fs$\tt2$};
		\node at(0.5,-2.5){\fs$\tt1$};
		\node at(1.5,-3.5){\fs$\tt 1$};
		\node at(2.5,-3.5){\fs$\tt 2$};
		\node at(3.5,-3.5){\fs$\tt 3$};
		\node at(4.5,-3.5){\fs$\tt 4$};
		\node at(5.5,-3.5){\fs$\tt 5$};
		
		\node[lightgray] at(1.5,-0.5){$\bullet$};
		\node[lightgray] at(4.5,-2.5){$\bullet$};
		\node at(3.5,1.5){$\bullet$};
		\node at(3.5,3){$254$};
		\node at(2.5,-1.5){$\bullet$};
		\node at(5.5,0.5){$\bullet$};
		\draw[->,line width=1pt,color=black,opacity=0.5] (7,0) to[bend left,out=20](9,0); 
		\end{scope}
		
		\begin{scope}[xshift=12cm, yshift=2cm]
		\draw[step=1.0,black,thin] (1,-3) grid (4,0);
		\node at(2.5,1){$132$};
		\node at(0.5,-0.5){\fs$\tt 3$};
		\node at(1.5,-2.5){$\bullet$};
		\node at(0.5,-1.5){\fs $\tt 2$};
		\node at(2.5,-0.5){$\bullet$};
		\node at(0.5,-2.5){\fs $\tt 1$};
		\node at(3.5,-1.5){$\bullet$};
		\node at(1.5,-3.5){\fs $\tt 1$};
		\node at(2.5,-3.5){\fs $\tt 2$};
		\node at(3.5,-3.5){\fs $\tt3$};
		\node at(-1.2,-2){$\sigma:$};
		\end{scope}
		\end{tikzpicture}
		\caption{An occurrence of the pattern $\sigma=132$ in $\pi=32514$.  $\sigma$ is
			obtained by deleting the first and forth column, as well as the third and fifth row of
			$\pi$.}
		\label{fig:perm}
	\end{center}
\end{figure}
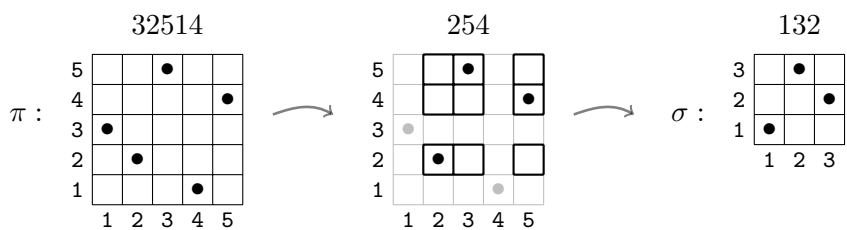
%
The systematic investigation of patterns in permutations and more generally in words began in the 70s with the work of Knuth on sorting permutations using data structures \cite{knuth-art-I}, and later with the work of Simion and Schmidt \cite{restrpermut-85}.
In the last decade the study of patterns in permutations and words has grown 
explosively  (see \cite{combperm-bona-04, kitaev-book-11})
and has been extended  in the context of various 
other structures  such as set partitions, trees, lattice paths,
fillings of Young diagrams, just to name a few. 
For instance, in \cite{spiridonov-fpsac} Spiridonov considered a notion of pattern
occurrence for binary fillings of grid shapes, which  naturally generalizes
permutation pattern occurrence as follows: an  occurrence of a pattern in a filling of
a grid shape $T$ is a filling of a sub-shape (or {\em minor shape}) $S$ of $T$, 
obtained by removing some rows and columns of $T$ (see Figure~\ref{fig:grid_pattern} 
for two such examples).
\begin{figure}[h!]
  \begin{tikzpicture}[scale=0.45]
	\begin{scope}[xshift=0cm,scale=0.8]
		\begin{scope}[xshift=-9cm,yshift=-1cm]
			\draw(2,0) grid (4,1);
			\draw(1,1) grid (5,2);
			\draw(0,2) grid (1,3);
			\draw(3,2) grid (4,3);
			\node at(2.5,0.5){$\circ$};
			\node at(3.5,0.5){$\circ$};
			\node at(1.5,1.5){$\bullet$};
			\node at(2.5,1.5){$\circ$};
			\node at(3.5,1.5){$\bullet$};
			\node at(4.5,1.5){$\circ$};
			\node at(0.5,2.5){$\circ$};
			\node at(3.5,2.5){$\bullet$};
			\node at(0,1){$T$};
			\draw[->,line width=1pt,color=black,opacity=0.5] (5.5,2) to[bend left,out=20]   
			(7.5,2);     
		\end{scope}
		
		\begin{scope}[xshift=-1cm,yshift=-1cm]
			\draw[lightgray] (2,0) grid (4,1);
			\draw[lightgray] (1,1) grid (2,2);
			\draw[lightgray] (4,1) grid (5,2);
			\draw[thick] (2,1) grid (4,2);
			\draw[thick] (0,2) grid (1,3);
			\draw[thick] (3,2) grid (4,3);
			\node[lightgray] at(2.5,0.5){$\circ$};
			\node[lightgray] at(3.5,0.5){$\circ$};
			\node[lightgray] at(1.5,1.5){$\bullet$};
			\node at(2.5,1.5){$\circ$};
			\node at(3.5,1.5){$\bullet$};
			\node[lightgray] at(4.5,1.5){$\circ$};
			\node at(0.5,2.5){$\circ$};
			\node at(3.5,2.5){$\bullet$};
			\draw[->,line width=1pt,color=black,opacity=0.5] (5.5,2) to[bend left,out=20](7.5,2); 
		\end{scope}
		
		\begin{scope}[xshift=6cm,yshift=0cm]
			\draw[step=1.0,black,thin] (2,0) grid (4,1);
			\draw[step=1.0,black,thin] (1,1) grid (2,2);
			\draw[step=1.0,black,thin] (3,1) grid (4,2);
			\node at(2.5,0.5){$\circ$};
			\node at(3.5,0.5){$\bullet$};
			\node at(1.5,1.5){$\circ$};
			\node at(3.5,1.5){$\bullet$};
			\node at(2,-1){$S$};
		\end{scope}
	\end{scope}
	\begin{scope}[xshift=13cm,scale=0.8]
		\begin{scope}[xshift=0cm,yshift=2cm]
			\draw (0,0) rectangle (4,1);
			\draw (0,-1) rectangle (3,1);
			\draw (0,-2) rectangle (2,1);
			\draw (0,-3) rectangle (1,1);
			
			\fill(0.5,-2.5)circle (6pt);  
			\fill(0.5,-1.5)circle (6pt);  
			\draw(1.5,-1.5)circle (6pt);  
			\fill(0.5,-.5)circle (6pt);  
			\fill(1.5,-.5)circle (6pt);  
			\draw(2.5,-.5)circle (6pt);  
			\fill(0.5,.5)circle (6pt);  
			\fill(1.5,.5)circle (6pt);  
			\fill(2.5,.5)circle (6pt);  
			\draw(3.5,.5)circle (6pt);  
			
			\node at(-1.2,-1){$T:$};
			\draw[->,line width=1pt,color=black,opacity=0.5] (4,-1) to[bend left,out=20] 
			(5.5,-1);  
		\end{scope}
		
		\begin{scope}[xshift=7cm,yshift=0]
			\draw[lightgray,thin] (-1,-1) grid (0,3);
			\draw[lightgray,thin] (0,1) grid (2,2);
			\draw[lightgray,thin] (2,2) grid (3,3);
			\draw[thick](0,0) grid (1,1);
			\draw[thick](0,2) grid (2,3);
			\fill[color=lightgray] (-0.5,-0.5)circle (6pt);
			\fill[color=lightgray] (-0.5,.5)circle (6pt);
			\fill[color=lightgray] (-0.5,1.5)circle (6pt);
			\fill[color=lightgray] (0.5,1.5)circle (6pt);
			\draw[color=lightgray] (1.5,1.5)circle (6pt);
			\draw[color=lightgray] (2.5,2.5)circle (6pt);
			\fill[color=lightgray] (-0.5,2.5)circle (6pt);
			\fill(0.5,2.5)circle (6pt);
			\draw(0.5,.5)circle (6pt);
			\fill(1.5,2.5)circle (6pt);
			\draw[->,line width=1pt,color=black,opacity=0.5] (3,1) to[bend left,out=20] (4.5,1); 
		\end{scope}
		
		\begin{scope}[xshift=12cm,yshift=0cm]
			\draw[step=1.0,black,thin] (0,0) grid (1,1);
			\draw[step=1.0,black,thin] (0,1) grid (2,2);
			\draw[color=black] (0.5,0.5)circle (6pt);
			\fill[color=black] (0.5,1.5)circle (6pt);
			\fill[color=black] (1.5,1.5)circle (6pt);
			\node at(3,0.8){$:S$};
		\end{scope}
	\end{scope}
\end{tikzpicture}
	\caption{Two examples of pattern occurrence in binary fillings of grid shapes
		in the sense of \cite{spiridonov-fpsac}. 
		In both cases, $S$ is obtained by deleting  two columns and a row of  $T$.}
	\label{fig:grid_pattern}
\end{figure}
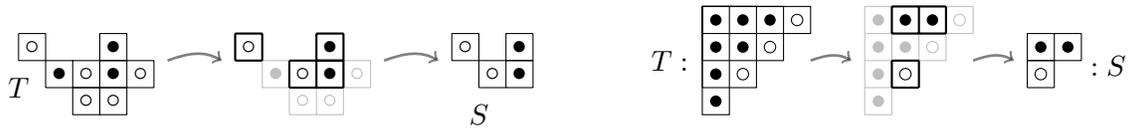
In a totally different context,  the authors in    \cite{dyckpp-14} and \cite{dmtcs:2334}
 considered a notion of pattern occurrence for Dyck paths,\, i.e.,  paths on the discrete plane from
$(0,0)$ to $(2n,0)$, $n\in\mathbb{N}$, consisting of up-steps $\U\!\!:\!(1,1)$ 
and down-steps $\D\!\!:\!(1,-1)$ never going below the $x$-axis. The pattern 
occurrence is defined by 
deleting   $U$ and $D$-steps so that the resulting path is again a Dyck path 
(with fewer steps). See  Figure~\ref{fig:dyckpatt} for an example of a Dyck 
path and a 
pattern occurrence.

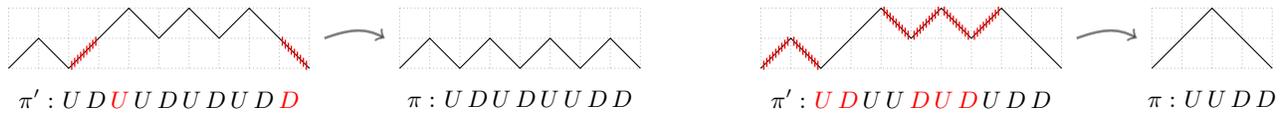
\begin{figure}
\begin{tikzpicture}[scale=0.5]
	\begin{scope}[scale=0.8, xshift=0cm,yshift=0cm]
		\draw[step=1,lightgray,thin,densely dotted] (0,0) grid (10,2);
		\draw (0,0)--(1,1)-- (2,0)--(3,1)--(4,2)--(5,1)--(6,2)--(7,1)--(8,2)--(10,0);
		\node at (5,-1){\fs $\pa':$ 
			$\U\,\D\,{\color{red}\U}\,\U\,\D\,\U\,\D\,\U\,\D\,{\color{red}\D}$};
		
		\foreach \i in {0.1,0.2,0.3,0.4,0.5,0.6,0.7,0.8,0.9}
		{\draw[red,line width=0.5pt] ($(0,-0.15)+(\i+2,\i)$)--($(0,0.15)+(\i+2,\i)$);}
		\foreach \i in {0.1,0.2,0.3,0.4,0.5,0.6,0.7,0.8,0.9}
		{\draw[red,line width=0.5pt] ($(0,-0.15)+(\i+9,1-\i)$)--($(0,0.15)+(\i+9,1-\i)$);}
		\draw[->,line width=1pt,color=black,opacity=0.5] (10.5,1) to[bend left,out=20] 
		(12.5,1);
	\end{scope}
	\begin{scope}[scale=0.8, xshift=13cm,yshift=0cm]
		\draw[step=1,lightgray,thin,densely dotted] (0,0) grid (8,2);
		\draw (0,0)--(1,1)--(2,0)--(3,1)--(4,0)--(5,1)--(6,0)--(7,1)--(8,0);
		\node at (4,-1){\fs $\pa:$  $\U\,\D\,\U\,\D\,\U\,\U\,\D\,\D$};
		
	\end{scope}
	\begin{scope}[scale=0.8, xshift=25cm,yshift=0cm]
		\draw[step=1,lightgray,thin,densely dotted] (0,0) grid (10,2);
		\draw (0,0)--(1,1)-- (2,0)--(3,1)--(4,2)--(5,1)--(6,2)--(7,1)--(8,2)--(10,0);
		\node at (5,-1){\fs $\pa':$  ${\color{red} 
				\U\,\D}\,\U\,\U\,{\color{red}\D\,\U\,\D}\,\U\,\D\,\D$};
		\foreach \i in {0.1,0.2,0.3,0.4,0.5,0.6,0.7,0.8,0.9}
		{\draw[red,line width=0.5pt] ($(0,-0.15)+(\i+5,\i+1)$)--($(0,0.15)+(\i+5,\i+1)$);}
		\foreach \i in {0.1,0.2,0.3,0.4,0.5,0.6,0.7,0.8,0.9}
		{\draw[red,line width=0.5pt] ($(0,-0.15)+(\i+6,2-\i)$)--($(0,0.15)+(\i+6,2-\i)$);}
		\draw[->,line width=1pt,color=black,opacity=0.5](10.5,1)to[bend left,out=20]
		(12.5,1);
		
		\foreach \i in {0.1,0.2,0.3,0.4,0.5,0.6,0.7,0.8,0.9}
		{\draw[red,line width=0.5pt] ($(-4,-0.15)+(\i+4,\i)$)--($(-4,0.15)+(\i+4,\i)$);}
		\foreach \i in {0.1,0.2,0.3,0.4,0.5,0.6,0.7,0.8,0.9}
		{\draw[red,line width=0.5pt] ($(7,-0.15)+(\i,\i+1)$)--($(7,0.15)+(\i,\i+1)$);}
		
		\foreach \i in {0.1,0.2,0.3,0.4,0.5,0.6,0.7,0.8,0.9}
		{\draw[red,line width=0.5pt] ($(1,-0.15)+(\i,1-\i)$)--($(1,0.15)+(\i,1-\i)$);}
		\foreach \i in {0.1,0.2,0.3,0.4,0.5,0.6,0.7,0.8,0.9}
		{\draw[red,line width=0.5pt] ($(4,-0.15)+(\i,2-\i)$)--($(4,0.15)+(\i,2-\i)$);}

	\end{scope}
	\begin{scope}[scale=0.8, xshift=38cm,yshift=0cm]
		\draw[step=1,lightgray,thin,densely dotted] (0,0) grid (4,2);
		\draw (0,0)--(2,2)-- (4,0);
		\node at (2,-1){\fs $\pa:$ $\U\,\U\,\D\,\D$};
	\end{scope}
	%
\end{tikzpicture}
	\caption{Two instances of pattern occurrence in the context of  \cite{dyckpp-14}. 
		In both examples, $\pa$ is a shorter  Dyck path obtained from the original 
		path 
		$\pa'$ by  deleting  the  steps  in red. }
	\label{fig:dyckpatt}
\end{figure}

%

Inspired by \cite{dyckpp-14} and \cite{spiridonov-fpsac}, 
we introduce a notion of patterns for the so called {\em Shi tableaux}, which 
are  structures   defined in \cite{FVT13}  to encode  dominant regions of the 
($m$-extended) Shi arrangement. 
For $m=1$  Shi  tableaux can easily  be  described as binary fillings 
of certain Young diagrams. 
More precisely, a {\em Shi tableau} of size $n$ is a 
binary filling with $\bullet$ or $\circ$ 
of a staircase Young diagram of shape $(n,n-1,\ldots,1)$, satisfying  the 
property 
that if a cell contains $\bullet$  then all cells above and to its left 
contain $\bullet$. It is straightforward to see
(cf. Figure~\ref{fig:dyckshi}) 
that Shi tableaux of size $n$ biject   to Dyck paths of semilength $n+1$.

\begin{figure}[h!]
	\begin{center}
  \begin{tikzpicture}[scale=0.34]
	\begin{scope}[xshift=0cm,yshift=4cm]
		\draw (0,0)rectangle (4,1);
		\draw (0,-1)rectangle (3,1);
		\draw(0,-2)rectangle(2,1);
		\draw(0,-4)--(0,-3);
		\draw(4,1)--(5,1);
		\draw(0,-3)  rectangle (1,1);
		\draw [line width=0.9pt, rounded corners=0.2] 
		(0,-4)--(0,-3)--(1,-3)--(1,-1)--(2,-1)--(2,0)--(3,0)--(3,1)--(5,1);
		\fill[color=black] (0.5,0.5) circle (7pt);
		\fill[color=black](1.5,0.5) circle (7pt);
		\fill[color=black](2.5,0.5) circle (7pt);
		\draw[color=black](3.5,0.5) circle (7pt);
		\draw[color=black](2.5,-0.5) circle (7pt);
		\draw[color=black](1.5,-1.5) circle (7pt);
		\fill[color=black](0.5,-0.5) circle (7pt);
		\fill[color=black](1.5,-0.5) circle (7pt);
		\fill[color=black](0.5,-1.5) circle (7pt);
		\fill[color=black](0.5,-2.5) circle (7pt);
	\end{scope}
	
	\begin{scope}[scale=0.8, xshift=14cm,yshift=1.2cm]
		\draw[step=1,lightgray,thin,densely dotted] (0,0) grid (10,5);
		\draw[line width=0.7pt, rounded corners=0.2] (0,0) -- (1,1) -- (2,0) -- (3,1) -- 
		(4,2) -- (5,1) -- (6,2) -- (7,1) -- (8,2)-- (10,0);
		{\foreach \i in {0,1,2,3} \fill[color=black, opacity=1] (2+\i,1+\i)circle 
			(10pt);}
		\fill[color=black, opacity=1] (5,2)circle (10pt);
		\fill[color=black, opacity=1] (6,3)circle (10pt);
		\fill[color=black, opacity=1] (7,2)circle (10pt);
	\end{scope}
\end{tikzpicture}
		\caption{
			A Shi tableau $T\in \shiT{4}$ and the corresponding Dyck path $\pa\in\dP{5}$.}
		\label{fig:dyckshi}
	\end{center}
\end{figure}
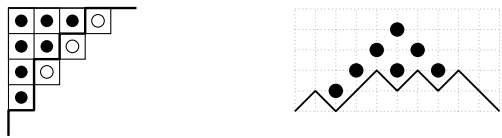

In accordance with  Spiridonov's definition \cite{spiridonov-fpsac}, the notion of 
pattern occurrence in Shi tableaux can be defined by deleting columns and rows.
However, in our pattern we impose a stronger condition; we allow deletions 
of rows and columns after which the underlying Young diagram is again a 
staircase Young diagram (see Section~\ref{sec:shi} for the precise definition). 
Part of the motivation of this work
lies in \cite[\textsection6]{dyckpp-14}, where Bacher et al. ask whether it is possible to
transport the pattern order on Dyck paths along some of the bijections between Dyck
paths and other members of the Catalan family in order to obtain interesting order
structures on different combinatorial objects. 
So far, our intuition to work on this subject has been affirmed by the enumerative
results presented here, which hint at interesting connections to the theory of pattern
avoidance for permutations.
This is perhaps not surprising in view of the fact that one of the original 
problems of this area, the enumeration of $312$-avoiding permutations, is also 
related to Dyck paths.
Furthermore, there are ties to algebraic objects that arise in connection to 
crystallographic root systems.
For example, $ad$-nilpotent ideals of a Borel subalgebra of the complex simple Lie
algebra of type $A$ with a bounded class of nilpotence,  studied in
\cite{akop-adnil-02}, can be described using pattern avoidance for Shi tableaux.

This paper is organized as follows.
In Section~\ref{sec:prel} the basic definitions and notation, including 
our definition of pattern occurrence in Shi tableaux, 
are provided. 
In Section~\ref{sec:ulbounds} we give explicit formulas for the number of 
upper and lower covers for each Shi tableau, in the poset defined by 
pattern-containment. 
In Section~\ref{sec:avoid} we give precise characterizations of pattern 
avoidance for each  Shi tableaux of size $2$ and generalize these results 
for certain tableaux of larger size. 
In the process we encounter $ad$-nilpotent ideals and the bijection on Dyck paths known as the zeta map. 
We complete the paper with a short discussion of open problems in Section~\ref{sec:problems} including possible connections to permutations avoiding a pair of patterns.

\section{Preliminaries}
\label{sec:prel}

\subsection{Dyck paths}
\label{ssec:Dpaths}
A \emph{Dyck path} of \emph{semilength} $n$
is a path on the plane, from $(0,0)$ to $(2n,0)$, consisting of up-steps $(1,1)$ and
down-steps $(1,-1)$ that never go below the $x$-axis.
It is  well-known that the cardinality of the set  $\dP{n}$  of 
Dyck paths of semilength $n$,  is given by the Catalan number 
$\frac{1}{n+1}\binom{2n}{n}$.
Replacing each up-step of a Dyck path $\pa\in\dP{n}$ with the letter $\U$ and 
each down-step with the letter $\D$, $\pa$ can also be written as a word with 
$2n$ letters on the alphabet $\{\U,\D\}$. Clearly, a word $w\in \{\U,\D\}^n$ 
corresponds to a Dyck  path if and only if  each  initial subword of $w$ 
contains at least as many letters $\U$ as letters $\D$. 
Another convenient way to encode paths in $\dP{n}$ is in terms of  standard 
Young tableaux of size $2\times{n}$, i.e., arrangements of the numbers 
$1,2,\ldots,2n$  in a $2\times{n}$ rectangle so that each row and each column 
is increasing.  The correspondence to Dyck paths  is as follows: in the top (resp. bottom) row of the $2\times{n}$  tableau we  register in increasing order the positions of the $\U$-steps  (resp. $\D$-steps)  of the Dyck  path  (see Figure 
\ref{fig:deletions}). Notice that since each Dyck path  begins with an up-step, the top-left cell of the $2\times n$ standard Young tableau is always 1.

The \emph{height} of a Dyck path is the highest $y$-coordinate attained in the path. A
\emph{return step} is a downstep that returns the path to the ground level.
Let $\tnk{n}{k}$ denote the set of Dyck paths of semilength $n$ and height at most $k$.
It is not hard to see that $|\tnk{n+1}{k}|=\sum_{i=0}^n |\tnk{i}{k}||\tnk{n-i}{k-1}|$,
with initial conditions $|\tnk{0}{k}|=1$ for all $k\geq 0$, and $|\tnk{n}{0}|=0$ for all
$n>0$ (see \cite{knuth-plantedtrees-72}).

The \emph{bounce path} $b(\pa)$ of a Dyck path $\pa$ is  described by
the following  algorithm:  starting at $(0,0)$ we travel along the up-steps of
$\pa$ until we encounter the  beginning of a down-step. Then, we turn and travel down until we hit the $x$-axis. Then, we travel up until
we again encounter the beginning of a down-step of $\pa$, we then turn down and travel
to the $x$-axis, etc. We continue in this way until we arrive at $(2n,0)$. 
For instance, if $\pa=\U\D\U\U\D\U\D\U\D\D$, then $b(\pa)=\U\D\U\U\D\D\U\U\D\D$.
A {\em peak at height $k$} of a Dyck path $\pa$ is a point $(x_0,k)$ of $\pa$ 
which is
immediately preceded by an up-step and immediately succeeded by a down-step.
Similarly a {\em valley at height $k$} of $\pa$ is a point $(x_0,k)\in\pa$ that is immediately preceded by a down-step and immediately followed by an up-step.
Viewing the path $\pa$ as a word, a peak  is an occurrence of a $\U\D$ and
its height  is the number of $\U$'s minus the number of 
$\D$'s that precede the  peak. A valley 
is an occurrence of $\D\U$ and its height is defined analogously. 
\subsection{Shi tableaux and our poset structure}
\label{sec:shi}
In order to switch  from  Dyck paths 
of semilength $n+1$ to Shi tableaux of size $n$ and vice versa, 
we make the convention that each Young diagram of shape $(n,\ldots,1)$ 
contains an additional empty row and empty column and we label rows  
from bottom to top and columns from left to right (so that, 
for $i=1,\ldots,n+1$,  the 
$i$-th row has $i-1$ boxes   and the $i$-th column has $n+1-i$ 
boxes, as shown in Figure \ref{fig:deletions}).
In this way, the  $i$-th $\D$-step ($\U$-step) of the Dyck path $\pa$ is the 
horizontal  (vertical) unit step on the $i$-th row (column).  
We denote by $\shiT{n}$ the set of all  Shi tableaux of size $n$ and by 
$\shiT{}$ the set of all Shi tableaux of all sizes.

Our main goal in this paper is to introduce and study a new notion of 
\emph{patterns} on the set $\shiT{}$ of  Shi tableaux.
The pattern containment is described in terms of two types of 
deletions on Shi tableaux which we call  {\em bounce deletions}. 
For $2\leq{i}\leq{n+1}$, we denote by $\de{i}{i-1}$  the action of deleting the 
$i$-th row and the  $(i-1)$-st column of $T$,  and
for $1\leq{i}\leq{n+1}$ we denote by
$\de{i}{i}$  the action of deleting the $i$-th  row and
the $i$-th column of $T$ (see Figure \ref{fig:deletions}).
The special cases $\de{i}{i}$ with $i=1$ or $n+1$,
can be thought of as deleting the first column and top row respectively.
Viewing the tableau $T$ as a Dyck path $\pa$ and indexing its $\U$ and 
$\D$-steps from $1$ to $n+1$, it is immediate to see that  
$\de{i}{i-1}$ deletes $\U_i,\D_{i-1}$,  
whereas  $\de{i}{i}$  deletes $\U_i,\D_i$ from $\pi$.  
Another  way to define the actions $\de{i}{i-1}, \de{i}{i}$ is through
the $2\times(n+1)$ standard Young tableaux $S(T)$ of $\pa$. 
The standard Young tableau of $\de{i}{i}(T)$ is obtained by deleting the 
$i$-th column of the standard $2\times (n+1)$ tableau $S(T)$ of $\pa$ and adjusting the larger entries.
The standard Young tableau of 
$\de{i}{i-1}(T)$ is obtained by deleting the $i$-th entry of the first row 
and $(i-1)$-st entry of the second row  of $S(T)$
and adjusting the larger entries (see Figure \ref{fig:deletions}). 
\begin{figure}[h!]
	\begin{center}
		\begin{tikzpicture}[scale=0.5]
			\draw[color=black,line width=1pt]
			(0,0)--(0,2)--(1,2)--(1,4)--(2,4)--(2,5)--(4,5)--(4,6)--(6,6);
			\draw[dotted] (0, 0) grid (0, 1);
			\draw[dotted] (0, 1) grid (1, 2);
			\draw[dotted] (0, 2) grid (2, 3);
			\draw[dotted] (0, 3) grid (3, 4);
			\draw[dotted] (0, 4) grid (4, 5);
			\draw[dotted] (0, 5) grid (5, 6);
			\draw[dotted] (5,6) -- (6,6);
			\draw[line width=12pt,opacity=0.15,rounded corners=0]
			(0,2.5)--(1.5,2.5)--(1.5,6);
			\fill[opacity=0.7] (0,0) circle (2pt);
			\fill[opacity=0.7] (6,6) circle (2pt);
			\draw[->,line width=1pt,color=black,opacity=0.5] (4,2) to[bend left,out=20] (7,2);
			\node at (5.5,2.9){\color{black}$\de{3}{2}$};
			\node at (-1,0.5){{\s{ \em row 1}}};
			\node at (-1,1.5){{\s{ \em row 2}}};
			\node at (-1,2.5){{\s{ \em row 3}}};
			\node at (-1.2,5.5){{\s{ \em row n+1}}};
			\node at (.7,7.2){\rotatebox{65}{{\s{ \em column 1}}}};
			\node at (1.7,7.2){\rotatebox{65}{{\s{ \em column 2}}}};
			\node at (4,7.2){$\cdots$};
			\node at (6,7.5){\rotatebox{65}{{\s{ \em column n+1}}}};
			\fill[color=black,opacity=0.4] (3.5,5.5) circle (5pt);
			\fill[color=black,opacity=0.4] (2.5,5.5) circle (5pt);
			\fill[color=black,opacity=0.4] (1.5,4.5) circle (5pt);
			\fill[color=black,opacity=0.4] (1.5,5.5) circle (5pt);
			\fill[color=black,opacity=0.4] (0.5,2.5) circle (5pt);
			\fill[color=black,opacity=0.4] (0.5,3.5) circle (5pt);
			\fill[color=black,opacity=0.4] (0.5,4.5) circle (5pt);
			\fill[color=black,opacity=0.4] (0.5,5.5) circle (5pt);
			\node[lightgray] at(0.5,1.5){$\circ$};
			\node[lightgray] at(1.5,2.5){$\circ$};
			\node[lightgray] at(1.5,3.5){$\circ$};
			\node[lightgray] at(2.5,3.5){$\circ$};
			\node[lightgray] at(2.5,4.5){$\circ$};
			\node[lightgray] at(3.5,4.5){$\circ$};
			\node[lightgray] at(4.5,5.5){$\circ$};
			
			
			\begin{scope}[xshift=8cm,yshift=0.5cm]
				\draw[color=black,line width=1pt]
				(0,0)--(0,2)--(1,2)--(1,4)--(3,4)--(3,5)--(5,5);
				\draw[dotted] (0, 0) grid (0, 1);
				\draw[dotted] (0, 1) grid (1, 2);
				\draw[dotted] (0, 2) grid (2, 3);
				\draw[dotted] (0, 3) grid (3, 4);
				\draw[dotted] (0, 4) grid (4, 5);
				\draw[dotted] (4,5) -- (5,5);
				\fill[opacity=0.7] (0,0) circle (2pt);
				\fill[opacity=0.7] (5,5) circle (2pt);
				\fill[color=black,opacity=0.4] (1.5,4.5) circle (5pt);
				\fill[color=black,opacity=0.4] (2.5,4.5) circle (5pt);
				\fill[color=black,opacity=0.4] (0.5,2.5) circle (5pt);
				\fill[color=black,opacity=0.4] (0.5,3.5) circle (5pt);
				\fill[color=black,opacity=0.4] (0.5,4.5) circle (5pt);
				\node[lightgray] at(0.5,1.5){$\circ$};
				\node[lightgray] at(1.5,2.5){$\circ$};
				\node[lightgray] at(1.5,3.5){$\circ$};
				\node[lightgray] at(2.5,3.5){$\circ$};
				\node[lightgray] at(3.5,4.5){$\circ$};

			\end{scope}
			
			\begin{scope}[xshift=17cm]
				\draw[->,line width=1pt,color=black,opacity=0.5] (4,2) to[bend left,out=20] (7,2);
				\node at (5.5,3){\color{black}$\de{3}{3}$};
				\draw[color=black,line width=1pt]
				(0,0)--(0,2)--(1,2)--(1,4)--(2,4)--(2,5)--(4,5)--(4,6)--(6,6);
				\draw[dotted] (0, 0) grid (0, 1);
				\draw[dotted] (0, 1) grid (1, 2);
				\draw[dotted] (0, 2) grid (2, 3);
				\draw[dotted] (0, 3) grid (3, 4);
				\draw[dotted] (0, 4) grid (4, 5);
				\draw[dotted] (0, 5) grid (5, 6);
				\draw[dotted] (5,6) -- (6,6);
				\draw[line width=12pt,opacity=0.15,rounded corners=0]
				(0,2.5)--(2.5,2.5)--(2.5,6);
				\fill[opacity=0.7] (0,0) circle (2pt);
				\fill[opacity=0.7] (6,6) circle (2pt);
				
				\node at (-1,0.5){{\s{ \em row 1}}};
				\node at (-1,1.5){{\s{ \em row 2}}};
				\node at (-1,2.5){{\s{ \em row 3}}};
				\node at (-1.2,5.5){{\s{ \em row n+1}}};
				\node at (.7,7.2){\rotatebox{65}{{\s{ \em column 1}}}};
				\node at (1.7,7.2){\rotatebox{65}{{\s{ \em column 2}}}};
				\node at (4,7.2){$\cdots$};
				\node at (6,7.5){\rotatebox{65}{{\s{ \em column n+1}}}};
				\fill[color=black,opacity=0.4] (3.5,5.5) circle (5pt);
				\fill[color=black,opacity=0.4] (2.5,5.5) circle (5pt);
				\fill[color=black,opacity=0.4] (1.5,4.5) circle (5pt);
				\fill[color=black,opacity=0.4] (1.5,5.5) circle (5pt);
				\fill[color=black,opacity=0.4] (0.5,2.5) circle (5pt);
				\fill[color=black,opacity=0.4] (0.5,3.5) circle (5pt);
				\fill[color=black,opacity=0.4] (0.5,4.5) circle (5pt);
				\fill[color=black,opacity=0.4] (0.5,5.5) circle (5pt);
				\node[lightgray] at(0.5,1.5){$\circ$};
				\node[lightgray] at(1.5,2.5){$\circ$};
				\node[lightgray] at(1.5,3.5){$\circ$};
				\node[lightgray] at(2.5,3.5){$\circ$};
				\node[lightgray] at(2.5,4.5){$\circ$};
				\node[lightgray] at(3.5,4.5){$\circ$};
				\node[lightgray] at(4.5,5.5){$\circ$};
			\end{scope}
			
			\begin{scope}[xshift=25cm,yshift=0.5cm]
				\draw[color=black,line width=1pt]
				(0,0)--(0,2)--(1,2)--(1,3)--(2,3)--(2,4)--(3,4)--(3,5)--(5,5);
				\draw[dotted] (0, 0) grid (0, 1);
				\draw[dotted] (0, 1) grid (1, 2);
				\draw[dotted] (0, 2) grid (2, 3);
				\draw[dotted] (0, 3) grid (3, 4);
				\draw[dotted] (0, 4) grid (4, 5);
				\draw[dotted] (4,5) -- (5,5);
				\fill[opacity=0.7] (0,0) circle (2pt);
				\fill[opacity=0.7] (5,5) circle (2pt);
				\fill[color=black,opacity=0.4] (1.5,4.5) circle (5pt);
				\fill[color=black,opacity=0.4] (2.5,4.5) circle (5pt);
				\fill[color=black,opacity=0.4] (0.5,2.5) circle (5pt);
				\fill[color=black,opacity=0.4] (0.5,3.5) circle (5pt);
				\fill[color=black,opacity=0.4] (1.5,3.5) circle (5pt);
				\fill[color=black,opacity=0.4] (0.5,4.5) circle (5pt);
				\node[lightgray] at(0.5,1.5){$\circ$};
				\node[lightgray] at(1.5,2.5){$\circ$};
				\node[lightgray] at(2.5,3.5){$\circ$};
				\node[lightgray] at(3.5,4.5){$\circ$};
			\end{scope}
			
			\begin{scope}[yshift=-2.2cm,scale=0.7]
				\draw (0,0) grid (6,2);
				\foreach \x/\y in {0.5/1,1.5/2,2.5/4,3.5/5,4.5/7,5.5/10}
				\node at (\x,1.5) {\tiny \y};
				
				\foreach \x/\y in {0.5/3,1.5/6,2.5/8,3.5/9,4.5/11,5.5/12}
				\node at (\x,.5) {\tiny \y};
				\draw[red,line width=2pt,opacity=0.6] (3,2)--(1,0); 
			\end{scope} 
			\begin{scope}[xshift=8cm,yshift=-2.2cm,scale=0.7]
				\draw (0,0) grid (5,2);
				\foreach \x/\y in {0.5/1,1.5/2,2.5/4,3.5/5,4.5/8}
				\node at (\x,1.5) {\tiny \y};
				
				\foreach \x/\y in {0.5/3,1.5/6,2.5/7,3.5/9,4.5/10}
				\node at (\x,.5) {\tiny \y};
			\end{scope}

			\begin{scope}[xshift=17cm,yshift=-2.2cm,scale=0.7]
				\draw (0,0) grid (6,2);
				\foreach \x/\y in {0.5/1,1.5/2,2.5/4,3.5/5,4.5/7,5.5/10}
				\node at (\x,1.5) {\tiny \y};
				
				\foreach \x/\y in {0.5/3,1.5/6,2.5/8,3.5/9,4.5/11,5.5/12}
				\node at (\x,.5) {\tiny \y};
				\draw[red,line width=2pt,opacity=0.6] (2.5,2)--(2.5,0); 
			\end{scope} 
			\begin{scope}[xshift=25cm,yshift=-2.2cm,scale=0.7]
				\draw (0,0) grid (5,2);
				\foreach \x/\y in {0.5/1,1.5/2,2.5/4,3.5/6,4.5/8}
				\node at (\x,1.5) {\tiny \y};
				
				\foreach \x/\y in {0.5/3,1.5/5,2.5/7,3.5/9,4.5/10}
				\node at (\x,.5) {\tiny \y};
			\end{scope}    
		\end{tikzpicture}
		\caption{The bounce deletions $\de{3}{2}$ and $\de{3}{3}$ on the Dyck path 
			$\U\U\D\U\U\D\U\D\D\U\D\D$
			results in $\U\U\D\U\U\D\D\U\D\D$ and $\U\U\D\U\D\U\D\U\D\D$ respectively.
		}
		\label{fig:deletions}
	\end{center}
\end{figure}
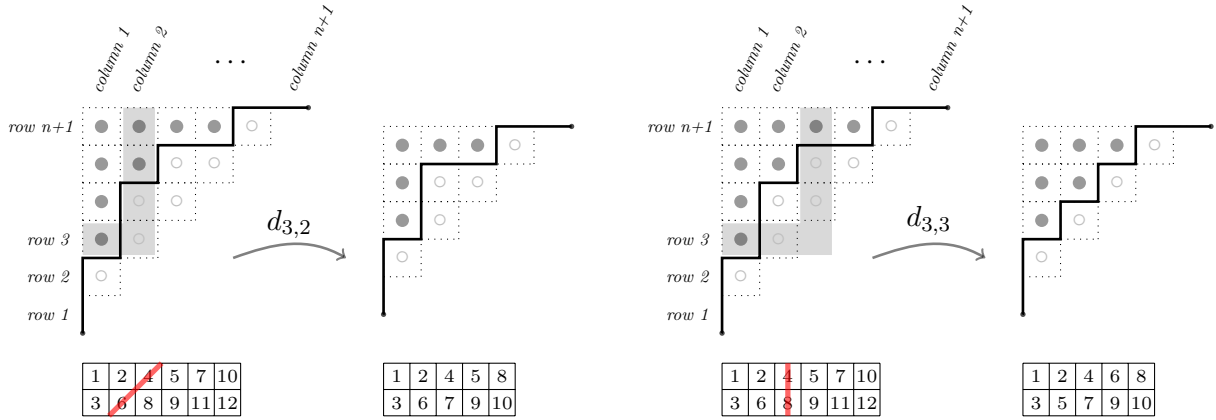
\begin{figure}
	
	\begin{tikzpicture}[scale=0.4]
		
		\begin{scope}[xshift=5cm]
			\draw[color=black,line width=0.8pt]
			(0,0)--(0,2)--(1,2)--(1,4)--(2,4)--(2,5)--(4,5)--(4,6)--(6,6)--(6,7)--(7,7);
			\draw[densely dotted] (0, 0) grid (0, 1);
			\draw[densely dotted] (0, 1) grid (1, 2);
			\draw[densely dotted] (0, 2) grid (2, 3);
			\draw[densely dotted] (0, 3) grid (3, 4);
			\draw[densely dotted] (0, 4) grid (4, 5);
			\draw[densely dotted] (0, 5) grid (5, 6);
			\draw[densely dotted] (0, 6) grid (6,7);
			\draw[densely dotted] (6,7) -- (7,7);
			\node at (2,0){ $T$};
			\draw[line width=10pt,opacity=0.15,rounded corners=0] (0,2.5)--(1.5,2.5)--(1.5,7);
			\foreach \i in {0.1,0.2,0.3,0.4,0.5,0.6,0.7,0.8,0.9}
			{\draw[red,line width=0.3pt,xshift=1cm,yshift=\i cm] (-0.1,1.9)--(0.1,2.1);}
			\foreach \i in {1.1,1.2,1.3,1.4,1.5,1.6,1.7,1.8,1.9}
			{\draw[red,line width=0.3pt,xshift=\i cm,yshift=4 cm] (-0.1,-0.1)--(0.1,0.1);}
			\node at (6.5,3){$\de{3}{2}$};
		\end{scope}
		
		\begin{scope}[xshift=14cm]
			
			\draw[->,line width=0.7pt,color=black,opacity=0.5](-4,2) to[bend left,out=20](-1.5,2);
			\draw[color=black,line width=0.8pt]
			(0,0)--(0,2)--(1,2)--(1,4)--(3,4)--(3,5)--(5,5)--(5,6)--(6,6);
			\draw[densely dotted] (0, 0) grid (0, 1);
			\draw[densely dotted] (0, 1) grid (1, 2);
			\draw[densely dotted] (0, 2) grid (2, 3);
			\draw[densely dotted] (0, 3) grid (3, 4);
			\draw[densely dotted] (0, 4) grid (4, 5);
			\draw[densely dotted] (0, 5) grid (5, 6);
			\draw[densely dotted] (5,6) -- (6,6);

			\draw[line width=10pt,opacity=0.15,rounded corners=0](0,4.5)--(4.5,4.5)--(4.5,6);
			\fill[opacity=0.7] (0,0) circle (2pt);
			\fill[opacity=0.7] (6,6) circle (2pt);
			\foreach \i in {0.1,0.2,0.3,0.4,0.5,0.6,0.7,0.8,0.9}
			{\draw[red,line width=0.3pt,xshift=3cm,yshift=\i cm] 
				($(0,2)+(-0.1,1.9)$)--($(0,2)+(0.1,2.1)$);}
			\foreach \i in {1.1,1.2,1.3,1.4,1.5,1.6,1.7,1.8,1.9}
			{\draw[red,line width=0.3pt,xshift=\i cm,yshift=5 cm] 
				($(3,0)+(-0.1,-0.1)$)--($(3,0)+(0.1,0.1)$);}
			\node at (6.5,3){$\de{5}{5}$};
		\end{scope}
		\begin{scope}[xshift=23cm]
			
			\draw[->,line width=0.7pt,color=black,opacity=0.5](-4,2) to[bend left,out=20](-1.5,2);
			\draw[color=black,line width=0.8pt]
			(0,0)--(0,2)--(1,2)--(1,4)--(4,4)--(4,5)--(5,5);
			\draw[densely dotted] (0, 0) grid (0, 1);
			\draw[densely dotted] (0, 1) grid (1, 2);
			\draw[densely dotted] (0, 2) grid (2, 3);
			\draw[densely dotted] (0, 3) grid (3, 4);
			\draw[densely dotted] (0, 4) grid (4, 5);
			\draw[densely dotted] (4,5) -- (5,5);
			\draw[line width=10pt,opacity=0.15,rounded corners=0](0.5,0)--(0.5,5);
			\foreach \i in {0.1,0.2,0.3,0.4,0.5,0.6,0.7,0.8,0.9}
			{\draw[red,line width=0.3pt,xshift=0cm,yshift=\i cm] 
				($(0,-2)+(-0.1,1.9)$)--($(0,-2)+(0.1,2.1)$);}
			\foreach \i in {1.1,1.2,1.3,1.4,1.5,1.6,1.7,1.8,1.9}
			{\draw[red,line width=0.3pt,xshift=\i cm,yshift=2 cm] 
				($(-1,0)+(-0.1,-0.1)$)--($(-1,0)+(0.1,0.1)$);}
			\node at (5.5,3){$\de{1}{1}$};
		\end{scope}
		\begin{scope}[xshift=31cm]
			
			\draw[->,line width=1pt,color=black,opacity=0.5](-4,2) to[bend 
			left,out=20](-1.5,2);
			\draw[color=black,line width=0.8pt]
			(0,0)--(0,3)--(3,3)--(3,4)--(4,4);
			\draw[densely dotted] (0, 0) grid (0, 1);
			\draw[densely dotted] (0, 1) grid (1, 2);
			\draw[densely dotted] (0, 2) grid (2, 3);
			\draw[densely dotted] (0, 3) grid (3, 4);
			\draw[densely dotted] (3,4) -- (4,4);
			
			\node at (2,0){ $T'$};
			
		\end{scope}
	\end{tikzpicture}
	
	\caption{}   
	\label{fig:patern_ex}
\end{figure}    

Using bounce deletions we can endow $\shiT{}$ with a poset structure.  
For any undefined terminology on 
posets we refer the reader to \cite[Section 3]{StanV1}. 
\begin{definition}
	\label{def:ourpattern}
	The set $\shiT{}$ of Shi tableaux becomes a poset by declaring that $T$ covers $T'$
	if $T'$ is obtained from $T$ after a bounce deletion. 
	We say that $T'$ \emph{occurs as a pattern} in $T$ if  $T'\preceq_{\shiT{}} T$, i.e., if $T'$ can be obtained 
	from $T$ after an iteration of  bounce deletions.
\end{definition}

See Figure~\ref{fig:patern_ex} for an example of pattern occurrence in Shi tableaux. 
Notice  that if  the  pattern $T'$ 
occurs  in $T$ then, in terms of Dyck paths,  $T'$ occurs as a pattern in $T$ in the
sense of \cite{spiridonov-fpsac}.  
Indeed, the actions $\de{i}{i-1},\de{i}{i}$
delete pairs of $\U,\D$  
so that the resulting path is a again a Dyck path
(which is precisely the requirement of  pattern occurrence in 
\cite{spiridonov-fpsac}). The reverse is not always true, since
in our pattern definition the deleted pair $\U,\D$
should obey stronger restrictions. For example, the path 
$\pa'=\U\U\D\D$  occurs as a pattern in   $\pa=\U\D\U\D\U\D$
in the sense of \cite{spiridonov-fpsac}  but not in the sense of Definition
\ref{def:ourpattern}.


\begin{figure}
	\begin{tikzpicture}[scale=0.35]
		
		\node at (0,16){$\vdots$};
		\node at (-4,16){$\vdots$};
		\node at (4,16){$\vdots$};
		\node at (9,16){\rotatebox{-90}{$\ddots$}};
		\node at (-9,16){$\ddots$};
		\node at (15,16){\rotatebox{-90}{$\ddots$}};
		\node at (-15,16){$\ddots$};
		
		\begin{scope}[xshift=0cm,yshift=1cm,scale=1]
			\draw[rounded corners=1, color=azu, line width=0.81,scale=0.8]  
			(0,-7)--(0,-6)--(1,-6);
			\draw[densely dotted,scale=0.8]  (0,-7)--(0,-6)--(1,-6);
			
			\draw[line width=0.5pt,color=black,opacity=0.5](0.5,-4.5)to[bend 
			right,out=3](2,-2.3);  
			\draw[line width=0.5pt,color=black,opacity=0.5](0.5,-4.5)to[bend 
			right,out=3](-2,-2.3);  
		\end{scope}
		\begin{scope}[xshift=-2.2cm,yshift=-1cm,scale=0.67]
			\draw[rounded corners=1, color=azu, line width=0.81]  (0,0)--(0,2)--(2,2);
			\draw[densely dotted] (0, 0) grid (0, 1);
			\draw[densely dotted] (0, 1) grid (1, 2);
			\draw[densely dotted] (1,2)--(2,2);

		\end{scope}
		\begin{scope}[xshift=2cm,yshift=-1cm,scale=0.67]
			\draw[rounded corners=1, color=azu, line width=0.81]  
			(0,0)--(0,1)--(1,1)--(1,2)--(2,2);
			\draw[densely dotted] (0, 0) grid (0, 1);
			\draw[densely dotted] (0, 1) grid (1, 2);
			\draw[densely dotted] (1,2)--(2,2);
		\end{scope}

		\begin{scope}[xshift=-12.5cm,yshift=4cm,scale=0.67]
			\draw[rounded corners=1, color=azu, line width=0.81](0,0)--(0,3)--(3,3);
			\draw[densely dotted] (0, 2) grid (2,3);
			\draw[densely dotted] (0, 1) grid (1, 2);
			\draw[densely dotted] (2,3)--(3,3);
			\draw[densely dotted] (0,0)--(0,1);
			
			\draw[line width=0.5pt,color=black,opacity=0.5](0.5,-0.5)to[bend
			left,out=3](15.5,-5);  
			
		\end{scope}
		\begin{scope}[xshift=-6cm,yshift=4cm,scale=0.67]
			\draw[rounded corners=1,color=azu,line width=0.81](0,0)--(0,2)--(1,2)--(1,3)--(3,3);
			\draw[densely dotted] (0, 2) grid (2,3);
			\draw[densely dotted] (0, 1) grid (1, 2);
			\draw[densely dotted] (2,3)--(3,3);
			\draw[densely dotted] (0,0)--(0,1);
			\draw[line width=0.5pt,color=black,opacity=0.5](0.5,-0.5)to[bend left,out=3](6,-5);  
			\draw[line width=0.5pt,color=black,opacity=0.5](0.5,-0.5)to[bend left,out=3](12,-5);  
		\end{scope}
		\begin{scope}[xshift=0cm,yshift=4cm,scale=0.67]
			\draw[rounded corners=1,color=azu,line width=0.81](0,0)--(0,2)--(2,2)--(2,3)--(3,3);
			\draw[densely dotted] (0, 2) grid (2,3);
			\draw[densely dotted] (0, 1) grid (1, 2);
			\draw[densely dotted] (2,3)--(3,3);
			\draw[densely dotted] (0,0)--(0,1);
			\draw[line width=0.5pt,color=black,opacity=0.5](0.5,-0.5)to[bend 
			right,out=3](-2.7,-5);  
			\draw[line width=0.5pt,color=black,opacity=0.5](0.5,-0.5)to[bend left,out=3](3.2,-5);  
		\end{scope}
		\begin{scope}[xshift=6cm,yshift=4cm,scale=0.67]
			\draw[rounded corners=1,color=azu,line width=0.81](0,0)--(0,1)--(1,1)--(1,3)--(3,3);
			\draw[densely dotted] (0, 2) grid (2,3);
			\draw[densely dotted] (0, 1) grid (1, 2);
			\draw[densely dotted] (2,3)--(3,3);
			\draw[densely dotted] (0,0)--(0,1);
			\draw[line width=0.5pt,color=black,opacity=0.5](0.5,-0.5)to[bend 
			right,out=3](-5.5,-5);  
			\draw[line width=0.5pt,color=black,opacity=0.5](0.5,-0.5)to[bend 
			right,out=3](-11.5,-5);  
		\end{scope}
		\begin{scope}[xshift=12cm,yshift=4cm,scale=0.67]
			\draw[rounded corners=1,color=azu,line 
			width=0.81](0,0)--(0,1)--(1,1)--(1,2)--(2,2)--(2,3)--(3,3);
			\draw[densely dotted] (0, 2) grid (2,3);
			\draw[densely dotted] (0, 1) grid (1, 2);
			\draw[densely dotted] (2,3)--(3,3);
			\draw[densely dotted] (0,0)--(0,1);
			\draw[line width=0.5pt,color=black,opacity=0.5](0.5,-0.5)to[bend 
			right,out=3](-14.3,-5); 
		\end{scope}

		\begin{scope}[xshift=-16cm,yshift=13cm,scale=0.45]
			\draw[rounded corners=1,color=azu,line width=0.81](0,0)--(0,4)--(4,4);
			\draw[densely dotted] (0, 3) grid (3,4);
			\draw[densely dotted] (0, 2) grid (2,3);
			\draw[densely dotted] (0, 1) grid (1, 2);
			\draw[densely dotted] (3,4)--(4,4);
			\draw[densely dotted] (0,0)--(0,1);
			\draw[line width=0.5pt,color=black,opacity=0.5](0,-0.5)to[bend left,out=2](8.3,-14); 
		\end{scope}
		\begin{scope}[xshift=-13.5cm,yshift=13cm,scale=0.45]
			\draw[rounded corners=1, color=azu,line width=0.81] 
			(0,0)--(0,3)--(1,3)--(1,4)--(4,4);
			\draw[densely dotted] (0, 3) grid (3,4);
			\draw[densely dotted] (0, 2) grid (2,3);
			\draw[densely dotted] (0, 1) grid (1, 2);
			\draw[densely dotted] (3,4)--(4,4);
			\draw[densely dotted] (0,0)--(0,1);
			\draw[line width=0.5pt,color=black,opacity=0.5](0,-0.5)to[bend left,out=2](3,-14); 
			\draw[line width=0.5pt,color=black,opacity=0.5](0,-0.5)to[bend left,out=2](17,-14); 
		\end{scope}
		\begin{scope}[xshift=-11cm,yshift=13cm,scale=0.45]
			\draw[rounded corners=1, color=azu,line width=0.81] 
			(0,0)--(0,3)--(2,3)--(2,4)--(4,4);
			\draw[densely dotted] (0, 3) grid (3,4);
			\draw[densely dotted] (0, 2) grid (2,3);
			\draw[densely dotted] (0, 1) grid (1, 2);
			\draw[densely dotted] (3,4)--(4,4);
			\draw[densely dotted] (0,0)--(0,1);
			\draw[line width=0.5pt,color=black,opacity=0.5](0,-0.5)to[bend left,out=2](-1.8,-14); 
			\draw[line width=0.5pt,color=black,opacity=0.5](0,-0.5)to[bend left,out=2](11.8,-14); 
			\draw[line width=0.5pt,color=black,opacity=0.5](0,-0.5)to[bend left,out=-2](24.8,-14); 
		\end{scope}
		\begin{scope}[xshift=-8.5cm,yshift=13cm,scale=0.45]
			\draw[rounded corners=1, color=azu,line width=0.81] 
			(0,0)--(0,2)--(1,2)--(1,4)--(4,4);
			\draw[densely dotted] (0, 3) grid (3,4);
			\draw[densely dotted] (0, 2) grid (2,3);
			\draw[densely dotted] (0, 1) grid (1, 2);
			\draw[densely dotted] (3,4)--(4,4);
			\draw[densely dotted] (0,0)--(0,1);
			\draw[line width=0.5pt,color=black,opacity=0.5](0.5,-0.5)to[bend 
			left,out=3](-6.8,-14);  
			\draw[line width=0.5pt,color=black,opacity=0.5](0.5,-0.5)to[bend 
			left,out=3](6.5,-14);  
			\draw[line width=0.5pt,color=black,opacity=0.5](0.5,-0.5)to[bend 
			left,out=0.015](32,-14);  
		\end{scope}
		\begin{scope}[xshift=-6cm,yshift=13cm,scale=0.45]
			\draw[rounded corners=1, color=azu,line width=0.81] 
			(0,0)--(0,2)--(1,2)--(1,3)--(2,3)--(2,4)--(4,4);
			\draw[densely dotted] (0, 3) grid (3,4);
			\draw[densely dotted] (0, 2) grid (2,3);
			\draw[densely dotted] (0, 1) grid (1, 2);
			\draw[densely dotted] (3,4)--(4,4);
			\draw[densely dotted] (0,0)--(0,1);
			\draw[line width=0.5pt,color=black,opacity=0.5](0.5,-0.5)to[bend 
			left,out=3](1,-14);
			\draw[line width=0.5pt,color=black,opacity=0.5](0.5,-0.5)to[bend 
			left,out=3](14,-14);    
			\draw[line width=0.5pt,color=black,opacity=0.5](0.5,-0.5)to[bend 
			left,out=3](27,-14);  
		\end{scope}
		\begin{scope}[xshift=-3.5cm,yshift=13cm,scale=0.45]
			\draw[rounded corners=1, color=azu,line width=0.81] 
			(0,0)--(0,2)--(2,2)--(2,4)--(4,4);
			\draw[densely dotted] (0, 3) grid (3,4);
			\draw[densely dotted] (0, 2) grid (2,3);
			\draw[densely dotted] (0, 1) grid (1, 2);
			\draw[densely dotted] (3,4)--(4,4);
			\draw[densely dotted] (0,0)--(0,1);
			\draw[line width=0.5pt,color=black,opacity=0.5](0.5,-0.5)to[bend 
			left,out=3](9,-14);
			\draw[line width=0.5pt,color=black,opacity=0.5](0.5,-0.5)to[bend 
			left,out=3](22,-14);    
			\draw[line width=0.5pt,color=black,opacity=0.5](0.5,-0.5)to[bend 
			left,out=3](-4.3,-14);    
		\end{scope}
		
		\begin{scope}[xshift=-1cm,yshift=13cm,scale=0.45]
			\draw[rounded corners=1, color=azu,line width=0.81] 
			(0,0)--(0,1)--(1,1)--(1,4)--(4,4);
			\draw[densely dotted] (0, 3) grid (3,4);
			\draw[densely dotted] (0, 2) grid (2,3);
			\draw[densely dotted] (0, 1) grid (1, 2);
			\draw[densely dotted] (3,4)--(4,4);
			\draw[densely dotted] (0,0)--(0,1);
			\draw[line width=0.5pt,color=black,opacity=0.5](0.5,-0.5)to[bend 
			left,out=3](-23,-14);  
			\draw[line width=0.5pt,color=black,opacity=0.5](0.5,-0.5)to[bend 
			left,out=3](17,-14);  
		\end{scope}
		\begin{scope}[xshift=1cm,yshift=13cm,scale=0.45]
			\draw[rounded corners=1, color=azu,line width=0.81] 
			(0,0)--(0,1)--(1,1)--(1,3)--(2,3)--(2,4)--(4,4);
			\draw[densely dotted] (0, 3) grid (3,4);
			\draw[densely dotted] (0, 2) grid (2,3);
			\draw[densely dotted] (0, 1) grid (1, 2);
			\draw[densely dotted] (3,4)--(4,4);
			\draw[densely dotted] (0,0)--(0,1);
			\draw[line width=0.5pt,color=black,opacity=0.5](0.5,-0.5)to[bend 
			left,out=3](-14,-14);  
			\draw[line width=0.5pt,color=black,opacity=0.5](0.5,-0.5)to[bend 
			left,out=3](13,-14);  
			\draw[line width=0.5pt,color=black,opacity=0.5](0.5,-0.5)to[bend 
			left,out=3](25,-14);  
		\end{scope}
		
		\begin{scope}[xshift=3.5cm,yshift=13cm,scale=0.45]
			\draw[rounded corners=1, color=azu,line width=0.81] 
			(0,0)--(0,1)--(1,1)--(1,3)--(3,3)--(3,4)--(4,4);
			\draw[densely dotted] (0, 3) grid (3,4);
			\draw[densely dotted] (0, 2) grid (2,3);
			\draw[densely dotted] (0, 1) grid (1, 2);
			\draw[densely dotted] (3,4)--(4,4);
			\draw[densely dotted] (0,0)--(0,1);
			\draw[line width=0.5pt,color=black,opacity=0.5](0.5,-0.5)to[bend 
			left,out=3](20,-14); 
			\draw[line width=0.5pt,color=black,opacity=0.5](0.5,-0.5)to[bend 
			left,out=3](7.8,-14);   
			\draw[line width=0.5pt,color=black,opacity=0.5](0.5,-0.5)to[bend 
			left,out=3](-6,-14);  
		\end{scope}
		\begin{scope}[xshift=6cm,yshift=13cm,scale=0.45]
			\draw[rounded corners=1, color=azu,line width=0.81] 
			(0,0)--(0,2)--(1,2)--(1,3)--(3,3)--(3,4)--(4,4);
			\draw[densely dotted] (0, 3) grid (3,4);
			\draw[densely dotted] (0, 2) grid (2,3);
			\draw[densely dotted] (0, 1) grid (1, 2);
			\draw[densely dotted] (3,4)--(4,4);
			\draw[densely dotted] (0,0)--(0,1);
			\draw[line width=0.5pt,color=black,opacity=0.5](0.5,-0.5)to[bend 
			left,out=3](15,-14);
			\draw[line width=0.5pt,color=black,opacity=0.5](0.5,-0.5)to[bend 
			left,out=3](-24.5,-14);    
			\draw[line width=0.5pt,color=black,opacity=0.5](0.5,-0.5)to[bend 
			left,out=3](-11,-14);  
		\end{scope}
		\begin{scope}[xshift=8.5cm,yshift=13cm,scale=0.45]
			\draw[rounded corners=1, color=azu,line width=0.81] 
			(0,0)--(0,3)--(3,3)--(3,4)--(4,4);
			\draw[densely dotted] (0, 3) grid (3,4);
			\draw[densely dotted] (0, 2) grid (2,3);
			\draw[densely dotted] (0, 1) grid (1, 2);
			\draw[densely dotted] (3,4)--(4,4);
			\draw[densely dotted] (0,0)--(0,1);
			\draw[line width=0.5pt,color=black,opacity=0.5](0,-0.5)to[out=-120,in=0](-20,-10);  
			\draw[line width=0.5pt,color=black,opacity=0.5](-20,-10)to[out=180,in=10](-41,-14);  
			\draw[line width=0.5pt,color=black,opacity=0.5](0,-0.5)to[bend 
			left,out=30](-15,-14);
		\end{scope}
		
		\begin{scope}[xshift=11cm,yshift=13cm,scale=0.45]
			\draw[rounded corners=1, color=azu,line width=0.81] 
			(0,0)--(0,1)--(1,1)--(1,2)--(2,2)--(2,4)--(4,4);
			\draw[densely dotted] (0, 3) grid (3,4);
			\draw[densely dotted] (0, 2) grid (2,3);
			\draw[densely dotted] (0, 1) grid (1, 2);
			\draw[densely dotted] (3,4)--(4,4);
			\draw[densely dotted] (0,0)--(0,1);
			\draw[line width=0.5pt,color=black,opacity=0.5](0,-0.5)to[bend 
			left,out=20](-8,-14);
			\draw[line width=0.5pt,color=black,opacity=0.5](0,-0.5)to[bend 
			left,out=20](4.5,-14);
		\end{scope}
		
		\begin{scope}[xshift=13.5cm,yshift=13cm,scale=0.45]
			\draw[rounded corners=1, color=azu,line width=0.81] 
			(0,0)--(0,2)--(2,2)--(2,3)--(3,3)--(3,4)--(4,4);
			\draw[densely dotted] (0, 3) grid (3,4);
			\draw[densely dotted] (0, 2) grid (2,3);
			\draw[densely dotted] (0, 1) grid (1, 2);
			\draw[densely dotted] (3,4)--(4,4);
			\draw[densely dotted] (0,0)--(0,1);
			\draw[line width=0.5pt,color=black,opacity=0.5](0,-0.5)to[bend 
			left,out=20](-0.8,-14);
			\draw[line width=0.5pt,color=black,opacity=0.5](0,-0.5)to[bend 
			left,out=10](-26.5,-14);
		\end{scope}

		\begin{scope}[xshift=16cm,yshift=13cm,scale=0.45]
			\draw[rounded corners=1, color=azu,line width=0.81] 
			(0,0)--(0,1)--(1,1)--(1,2)--(2,2)--(2,3)--(3,3)--(3,4)--(4,4);
			\draw[densely dotted] (0, 3) grid (3,4);
			\draw[densely dotted] (0, 2) grid (2,3);
			\draw[densely dotted] (0, 1) grid (1, 2);
			\draw[densely dotted] (3,4)--(4,4);
			\draw[densely dotted] (0,0)--(0,1);
			\draw[line width=0.5pt,color=black,opacity=0.5](0,-0.5)to[bend 
			left,out=20](-6,-14);
		\end{scope}
		
		%
		%
		
		\node at (-17.5,14){$\mathcal T_3$};
		\node at (-15,5){$\mathcal T_2$};
		\node at (-5,-0.5){$\mathcal T_1$};
		
	\end{tikzpicture}

	\caption{The first few levels of the  poset $(\shiT{},\preceq_{\shiT{}})$ of Shi tableau}
	\label{fig:all_the_poset}
\end{figure}
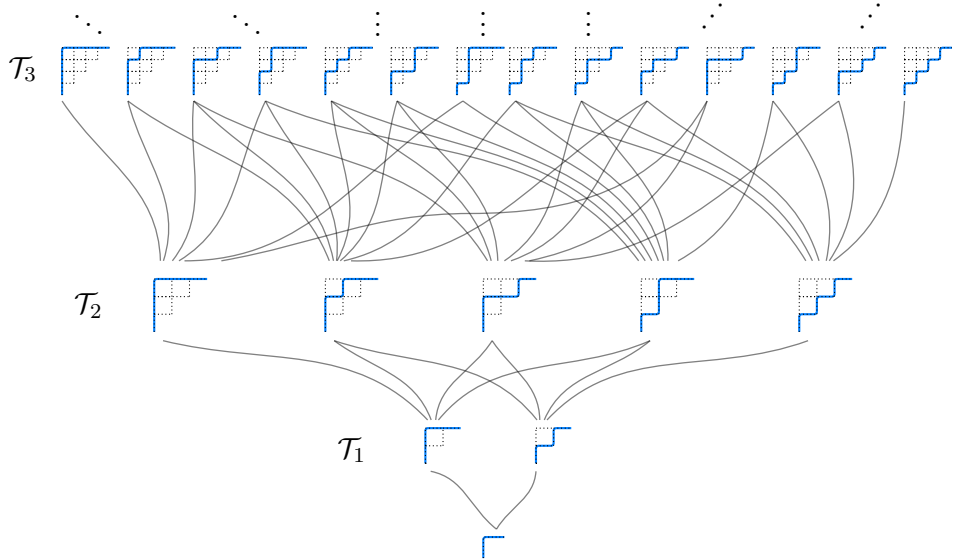

\begin{remark}
	Although we gave the definition of Dyck paths in terms of  unit steps $(1,1)$ 
	and  $(1,-1)$ and endpoints on the $x$-axis, in the remainder of the paper we
	rotate the setting  by 45 degrees (see Figure \ref{fig:dyckshi}). In other 
	words, unless stated otherwise,   we    align all Dyck paths so that their 
	unit steps are $(1,0), (0,1)$ and their 
	endpoints lie on the main diagonal $x=y$. 
	\label{rotate45}
\end{remark}

\subsection{Geometric interpretation of our poset structure} 
In what follows, we discuss the poset structure  $(\shiT{},\preceq_{\shiT{}})$ and more precisely the cover relations from the viewpoint of Shi arrangements.

A  hyperplane arrangement $\aA$ is a finite  set of affine hyperplanes in some
vector space $V\cong\mathbb R^n$.
The {\em regions} of $\aA$ are the connected components defined by the 
complement of the hyperplanes in $\aA$. 
The {\em rank} $r(\aA)$ of a hyperplane arrangement $\aA\subseteq \RR^n$ is the 
dimension of the space $A$ spanned by the normals to the hyperplanes in $\aA$.
The {\em intersection poset} $\mathcal L_{\aA}$ of $\aA$ is the set  of all 
intersections of subcollections of hyperplanes in $\aA$, partially ordered by 
reverse inclusion. The poset $\mathcal L_{\aA}$ is a lattice that 
captures all the combinatorial structure of $\aA$. 
We say that two 
hyperplane arrangements are {\em combinatorially equivalent}
if they have the same  intersection lattice \cite{Stan_ha04}.
If the rank $r(\aA)$ of a hyperplane arrangement $\aA$ is smaller than the 
dimension of the ambient space 
$\RR^n$, we can {\em essentialize} $\aA$, i.e.,  intersect $\aA$ with a subspace $W$ of $\RR^n$
without changing the combinatorial structure. 

To do so,
we chose a subspace $Y$ 
complementary to $A$\footnote{not necessarily  the 
	orthogonal complement} and we define  
\begin{align*}
W:=Y^{\perp}=\{x\in \RR^n: x\cdot y=0 \text{ for all } 
y\in Y \},
\end{align*}
where $x\cdot y$ denotes the standard Euclidean inner product.
The set $\aA_W=\{H\cap W: H\in \aA\}$ is an {\em essentialization of $\aA$}, 
i.e., a hyperplane arrangement in $W$ 
combinatorially equivalent to $\aA$ with $r(\aA_W)=\dim(\aA_W)$ \cite[Section 1.1]{Stan_ha04}. 
Notice that, if we do the above steps with  $A$ being a  {\em proper} rather 
than the {\em whole} space spanned by the normals to the hyperplanes in 
$\aA$,  we  again 
obtain a hyperplane arrangement $\aA_{W}$ which is
combinatorially equivalent to $\aA$ but with 
$r(\aA_W)<\dim(\aA_W)$. 

The {\em Shi arrangement} of type $A_{n-1}$, denoted by $\shi{n}$,  is the 
hyperplane arrangement in $\mathbb 
R^n$ consisting of the hyperplanes $x_i-x_j=0, 1$ for all $1\leq{i}<j\leq{n}$. 
It is well known that the  regions of $\shi{n}$ are in bijection with 
parking functions on $[n]$, thus counted by $(n+1)^{n-1}$ \cite[Section 
5]{Stan_HAIO96}. 
Here, we focus on the set of {\em dominant} regions of $\shi{n}$ 
which are those regions contained in the {\em dominant cone} 
$\mathcal C_n: x_1>x_2>\cdots>x_n$. The set of dominant regions is enumerated  by the Catalan 
number $\frac{1}{n+1}\tbinom{2n}{n}$  
and  can be encoded using Shi tableaux in 
$\ccT{n-1}$ \cite{Shi_acawg97}.  More precisely, 
each $T\in\ccT{n-1}$ corresponds to  the dominant region $\mathcal R(T)$
with defining inequalities 
\begin{align}
\begin{cases}
0<{x_i}-x_j<1 & \text{ if the cell }  (n-i+1,n-j+1) \text{ is empty } \\
1<{x_i}-x_j & \text{ if the cell } (n-i+1,n-j+1) \text{ is full, }
\end{cases}
\label{TtoR}
\end{align}
where, for the position of each cell,  we keep our earlier numbering on rows 
and columns (see Figure \ref{fig:shi(n)}).
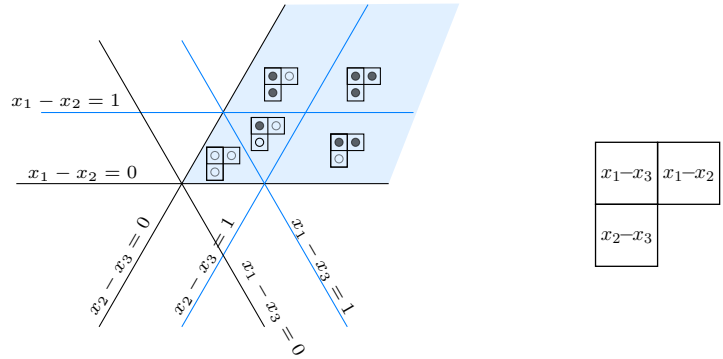
\begin{figure}[h]
\begin{tikzpicture}[scale=1.1]
	
	\fill[azu,opacity=0.12] (0:0)--(0:2.5)--(33:4)--(60:2.5);
	\draw (60:-2)--(60:2.5);
	\draw[xshift=1cm,azu] (60:-2)--(60:2.5);
	\draw (0:-2)--(0:2.5);
	\draw[yshift=0.86cm,xshift=0.3cm,azu] (0:-2)--(0:2.5);
	\draw (120:-2)--(120:2);
	\draw[yshift=0cm,xshift=1cm,azu] (120:-2)--(120:2);
	
	\begin{scope}[scale=0.2,xshift=1.5cm,yshift=1.2cm]
		\node at (1,1){\color{white}.};
		\draw (0,0) rectangle (2,1);
		\draw (0,-1) rectangle (1,1);
		\draw (0,-1) rectangle (1,1);
		\draw[opacity=0.5] (0.5,-0.5)circle (7pt);
		\draw[opacity=0.5] (0.5,0.5)circle (7pt);
		\draw[opacity=0.5] (1.5,0.5)circle (7pt);
	\end{scope}
	\begin{scope}[scale=0.2,xshift=4.2cm,yshift=3cm]
		\draw (0,0) rectangle (2,1);
		\draw (0,-1) rectangle (1,1);
		\draw (0.5,-0.5)circle (7pt);
		\draw[fill=black,opacity=0.6]  (0.5,0.5)circle (7pt);
		\draw[opacity=0.5](1.5,0.5)circle (7pt);
	\end{scope}
	\begin{scope}[scale=0.2,xshift=9cm,yshift=2cm]
		\draw (0,0) rectangle (2,1);
		\draw (0,-1) rectangle (1,1);
		\draw (0,-1) rectangle (1,1);
		\draw[opacity=0.5] (0.5,-0.5)circle (7pt);
		\draw[fill=black,opacity=0.6]  (0.5,0.5)circle (7pt);
		\draw[fill=black,opacity=0.6]  (1.5,0.5)circle (7pt);
	\end{scope}
	\begin{scope}[scale=0.2,xshift=5cm,yshift=6cm]
		\draw (0,0) rectangle (2,1);
		\draw (0,-1) rectangle (1,1);
		\draw[opacity=0.5] (1.5,0.5)circle (7pt);
		\draw[fill=black,opacity=0.6]  (0.5,0.5)circle (7pt);
		\draw[fill=black,opacity=0.6]  (0.5,-0.5)circle (7pt);
	\end{scope}
	\begin{scope}[scale=0.2,xshift=10cm,yshift=6cm]
		\draw (0,0) rectangle (2,1);
		\draw (0,-1) rectangle (1,1);
		\draw[fill=black,opacity=0.6] (0.5,-0.5)circle (7pt);
		\draw[fill=black,opacity=0.6]  (0.5,0.5)circle (7pt);
		\draw[fill=black,opacity=0.6]  (1.5,0.5)circle (7pt);
	\end{scope}

	\node at (-1.4,1){\rotatebox{0}{\tiny $x_1-x_2=1$}};
	\node at (-1.2,0.12){\rotatebox{0}{\tiny $x_1-x_2=0$}};
	\node at (-0.75,-1){\rotatebox{60}{\tiny $x_2-x_3=0$}};
	\node at (0.27,-1){\rotatebox{60}{\tiny $x_2-x_3=1$}};
	
	\node at (1.1,-1.5){\rotatebox{-60}{\tiny $x_1-x_3=0$}};
	\node at (1.7,-1){\rotatebox{-60}{\tiny $x_1-x_3=1$}};
	
	\begin{scope}[xshift=7cm]
		content
		
	\end{scope}
	\begin{scope}[xshift=5cm,yshift=-1cm,scale=0.75]
		\draw[line width= 0.5pt] (0,0) rectangle (1,1);
		\draw[line width= .5pt] (0,1) rectangle (1,2);  
		\draw[line width= .5pt] (1,1) rectangle (2,2);  
		\node at (1.45, 1.5) {\scalebox{0.7}{ $x_1\!\!-\!\!x_2$}};
		\node at (.45, 1.5) {\scalebox{0.7}{ $x_1\!\!-\!\!x_3$}};
		\node at (.45, .5) {\scalebox{0.7}{ $x_2\!\!-\!\!x_3$}};
		
	\end{scope}
	
\end{tikzpicture}
	\caption{ The dominant regions in $\shi{3}$ and their correspondence to Shi 
		tableaux. }
	{\color{white} ....................}
	For example, the tableau 
	\begin{tikzpicture}[scale=0.2,xshift=30cm,yshift=-4cm]
	\draw (0,0) rectangle (2,1);
	\draw (0,-1) rectangle (1,1);
	\draw[opacity=0.5] (1.5,0.5)circle (7pt);
	\draw[fill=black,opacity=0.6]  (0.5,0.5)circle (7pt);
	\draw[fill=black,opacity=0.6]  (0.5,-0.5)circle (7pt);
	\end{tikzpicture}
	corresponds to the dominant region defined by the inequalities  $x_1-x_3>1,$
	$x_2-x_3>1$
	and $0<x_1-x_2<1$. {\color{white} .........} 
	\label{fig:shi(n)}
\end{figure}

Shi arrangements form an
{\em exponential sequence of arrangements (ESA)},  a family of hyperplane arrangements
which posses a strong combinatorial 
symmetry \cite[Section 5.3]{Stan_ha04}. More precisely, $\shi{n}$ has the 
property that, for each $S\subseteq[n]$, its subarrangement $\shi{S}$  consisting of the hyperplanes $x_i-x_j=0,1$ for $i,j\in S$, 
is {\em combinatorially equivalent} to  $\shi{|S|}$. 

Let us apply the  properties of ESA to $\shi{n}$. 
If $S_k:=[n]\setminus \{k\}$ then $\shi{S_k}$ is the 
subarrangement of $\shi{n}$
from which we have deleted all hyperplanes whose equation involves $x_k$. 
It is immediate to see that the deleted hyperplanes are exactly the ones removed by the deletion $\de{i}{i}$ with $i=n-k+1$ (see Figure \ref{fig:Shi_deletions}(a)). 
Therefore, the bounce deletions of type $\de{i}{i}$ capture such instances of  pattern occurrence. 

The other type of bounce deletions do not constitute an instance of ESA, however they behave likewise 
if restricted to a certain hyperplane. To be more precise, let 
$\shi{S'_k}$ be the subarrangement of $\shi{n}$
from which we have deleted all hyperplanes  $x_i-x_{k+1}=0,1$ ($i\leq{k}$) or $x_{k}-x_i=0,1$
($i>k$). 

It is immediate to see that the deleted hyperplanes are exactly the ones removed by the deletion $\de{i}{i-1}$ with $i=n-k+1$ (see Figure \ref{fig:Shi_deletions}(b)). 
Although $\shi{S'_k}$ is not combinatorially equivalent to $\shi{n-1}$, 
it can be shown that it becomes so if we intersect with the hyperplane $x_k-x_{k+1}=0$. 
To see this denote this hyperplane by $W$ and let $Y=\RR(e_k-e_{k+1})$ be the span of its normal vector such that $W=Y^{\perp}$.
Since $Y$ is not contained in the span of the normals to the hyperplanes in $\shi{S_k}$ it follows from the above discussion that $\shi{S_k}_W$ and $\shi{n-1}$ are combinatorially equivalent.
But evidently $\shi{S_k}_W$ is equal to the restriction $\{H\cap W:H\in\shi{S_k'}\}$ of $\shi{S_k'}$ to $W$.

\begin{figure}[h]
	\scalebox{0.85}{\begin{tikzpicture}[scale=0.8]
			\draw[line width= 1pt] (1,0) rectangle (1,5);
			\draw[line width= 1pt] (1,1) rectangle (2,5);  
			\draw[line width= 1pt] (1,2) rectangle (3,5);
			\draw[line width= 1pt] (1,3) rectangle (4,5); 
			\draw[line width= 1pt] (1,4) rectangle (5,5); 
			\draw[line width= 1pt] (5,5)--(6,5);

			\node at (1.45, 4.5) {\scalebox{0.7}{ $x_1\!\!-\!\!x_5$}};
			\node at (2.45, 4.5) {\scalebox{0.7}{ $x_1\!\!-\!\!x_4$}};
			\node at (3.45, 4.5) {\scalebox{0.7}{ $x_1\!\!-\!\!x_3$}};
			\node at (4.45, 4.5) {\scalebox{0.7}{ $x_1\!\!-\!\!x_2$}};
			\node at (1.45, 3.5) {\scalebox{0.7}{ $x_2\!\!-\!\!x_5$}};
			\node at (2.45, 3.5) {\scalebox{0.7}{ $x_2\!\!-\!\!x_4$}};
			\node at (3.45, 3.5) {\scalebox{0.7}{ $x_2\!\!-\!\!x_3$}};
			\node at (3.45, 3.5) {\scalebox{0.7}{ $x_2\!\!-\!\!x_3$}};
			\node at (1.45, 2.5) {\scalebox{0.7}{ $x_3\!\!-\!\!x_5$}};
			\node at (2.45, 2.5) {\scalebox{0.7}{ $x_3\!\!-\!\!x_4$}};
			\node at (1.45, 1.5) {\scalebox{0.7}{ $x_4\!\!-\!\!x_5$}};
			
			\draw[line width=21pt,opacity=0.15,rounded corners=0]
			(1,2.5)--(3.5,2.5)--(3.5,5);
			
			
			\node at (7,3){\tikzarrow{azu!30!gray}{\scalebox{0.9}{$\;\;\;\;\;\de{3}{3}\;\;\;\;\;\;$}}};
			
			\node at (6,0){(a)};
			\begin{scope}[xshift=8cm,yshift=1cm]
				
				\draw[line width= 1pt] (1,1) rectangle (2,4);  
				\draw[line width= 1pt] (1,2) rectangle (3,4);
				\draw[line width= 1pt] (1,3) rectangle (4,4); 
				\draw[line width= 1pt] (1,0)--(1,1);
				\draw[line width= 1pt] (4,4)--(5,4);

				\node at (1.45, 3.5) {\scalebox{0.7}{ $\x_1\!\!-\!\!\x_5$}};
				\node at (2.45, 3.5) {\scalebox{0.7}{ $\x_1\!\!-\!\!\x_4$}};
				\node at (3.45, 3.5) {\scalebox{0.7}{ $\x_1\!\!-\!\!\x_2$}};
				\node at (1.45, 2.5) {\scalebox{0.7}{ $\x_2\!\!-\!\!\x_5$}};
				\node at (2.45, 2.5) {\scalebox{0.7}{ $\x_2\!\!-\!\!\x_4$}};
				\node at (1.45, 1.5) {\scalebox{0.7}{ $\x_4\!\!-\!\!\x_5$}};
				
			\end{scope}
		\end{tikzpicture}\;\;\;\;\;\;\;\;\;\;
		\begin{tikzpicture}[scale=0.8]
			\begin{scope}[xshift=0cm]
				\draw[line width= 1pt] (1,0) rectangle (1,5);
				\draw[line width= 1pt] (1,1) rectangle (2,5);  
				\draw[line width= 1pt] (1,2) rectangle (3,5);
				\draw[line width= 1pt] (1,3) rectangle (4,5); 
				\draw[line width= 1pt] (1,4) rectangle (5,5); 
				\draw[line width= 1pt] (5,5)--(6,5);

				\node at (1.45, 4.5) {\scalebox{0.7}{ $x_1\!\!-\!\!x_5$}};
				\node at (2.45, 4.5) {\scalebox{0.7}{ $x_1\!\!-\!\!x_4$}};
				\node at (3.45, 4.5) {\scalebox{0.7}{ $x_1\!\!-\!\!x_3$}};
				\node at (4.45, 4.5) {\scalebox{0.7}{ $x_1\!\!-\!\!x_2$}};
				\node at (1.45, 3.5) {\scalebox{0.7}{ $x_2\!\!-\!\!x_5$}};
				\node at (2.45, 3.5) {\scalebox{0.7}{ $x_2\!\!-\!\!x_4$}};
				\node at (3.45, 3.5) {\scalebox{0.7}{ $x_2\!\!-\!\!x_3$}};
				\node at (3.45, 3.5) {\scalebox{0.7}{ $x_2\!\!-\!\!x_3$}};
				\node at (1.45, 2.5) {\scalebox{0.7}{ $x_3\!\!-\!\!x_5$}};
				\node at (2.45, 2.5) {\scalebox{0.7}{ $x_3\!\!-\!\!x_4$}};
				\node at (1.45, 1.5) {\scalebox{0.7}{ $x_4\!\!-\!\!x_5$}};
				
				\draw[line width=21pt,opacity=0.15,rounded corners=0]
				(1,2.5)--(2.5,2.5)--(2.5,5);
				\node at (7,3){\tikzarrow{azu!30!gray}{\scalebox{0.9}{$\de{3}{2}\cap\{x_3=x_4\}$}}};
				\node at (6,0){(b)};
			\end{scope}
			\begin{scope}[xshift=9cm,yshift=1cm]
				\draw[line width= 1pt] (1,1) rectangle (2,4);  
				\draw[line width= 1pt] (1,2) rectangle (3,4);
				\draw[line width= 1pt] (1,3) rectangle (4,4); 
				\draw[line width= 1pt] (1,0)--(1,1);
				\draw[line width= 1pt] (4,4)--(5,4);

				\node at (1.45, 3.5) {\scalebox{0.7}{ $\x_1\!\!-\!\!\x_5$}};
				\node at (2.45, 3.5) {\scalebox{0.7}{ $\x_1\!\!-\!\!\x_4$}};
				\node at (3.45, 3.5) {\scalebox{0.7}{ $\x_1\!\!-\!\!\x_2$}};
				\node at (1.45, 2.5) {\scalebox{0.7}{ $\x_2\!\!-\!\!\x_5$}};
				\node at (2.45, 2.5) {\scalebox{0.7}{ $\x_2\!\!-\!\!\x_4$}};
				\node at (1.45, 1.5) {\scalebox{0.7}{ $\x_4\!\!-\!\!\x_5$}};
			\end{scope}
	\end{tikzpicture}}
	\caption{}
	\label{fig:Shi_deletions}
\end{figure}

	\begin{remark}
		As was pointed out by one of the reviewers, the  bounce deletions
		described above	 have a strong resemblance to the deletion and contraction operations on hyperplane arrangements. 
		It would be worth exploring these operations in the context of graphical arrangements (see \cite[Section 2.3]{Stan_ha04})
	\end{remark}

\section{Cover relations}
\label{sec:ulbounds}

In this section we compute the number of upper and lower covers for each Shi tableau $T$
in the poset $(\shiT{},\preceq_{\shiT{}})$. To do so, we define irreducible and strongly irreducible 
Shi tableaux and we find closed/recursive formulas for their number of lower/upper covers.
Then, we explain how each Shi tableau is decomposed into its  irreducible components and how the lower/upper covers are computed in terms of the upper/lower covers of the decomposition. 
\medskip 

We start with notation and definitions that will be used throughout the section. 
In this section, we find it more convenient to use the notation $\pa$ for Dyck paths, rather than $T$ for Shi tableaux, since most of the arguments use the realization of $\pa$ as a word in $\{\U,\D\}^*$. 
However, the underlying tableau of $\pa$ and more precisely its columns and rows is 
present in all the proofs. 

Let $T$ be a Shi  tableau of size $n$ and denote by $\pa\in\dP{n+1}$ the corresponding Dyck path.  
Let us denote by $\col{k}$ the $k$-th column of the underlying tableau of 
$\pa\in\dP{n+1}$ and index the $\U$-steps and $\D$-steps of $\pa$ from $1$ to 
$n+1$. 
The portion  of the path $\pa$ contained in column $\col{k}$ consists of all  
the steps that satisfy $k-1\leq{x}\leq{k}$. In other words 
$\pa\cap\col{k}$ is the subpath between (and not including) $\D_{k-1}$ 
and $\D_{k+1}$. 
An {\em ascent} of a Dyck path $\pa$ is a maximal string of consecutive $\U$-steps  and a 
{\em descent} is a maximal string of consecutive $\D$-steps of $\pa$.
Each Dyck path $\pi$ can be written as a concatenation of ascents and descents, 
i.e.\,,  
\begin{equation}
\pi=\U^{a_1}\D^{b_1}\U^{a_2}\D^{b_2}\cdots\U^{a_\ell}\D^{b_\ell}
\text{ \;\;for some\;\; } 1\leq{\ell}\leq{n+1} \text{ \;\; and\;\;  } a_i,b_i\geq{1}.
\label{asc_des}
\end{equation}
For each ascent $\U^{a_i}$ of $\pa$ we define 
a subpath $\overline \pa_i$  of $\pa$ as follows:
\begin{align}
\overline\pa_1= \pa\cap\bigl(\bigcup_{r=1}^{a_1} 
\col{r}\bigr)  \text{ \;\;\;\; and \;\;\;\;\;\;}
\overline\pa_i=\pa\cap\bigl(\bigcup_{r=0}^{a_i} 
\col{a_1+\cdots+a_{i-1}+r}\bigr) \text{\;\; if \;\;} i \geq{2}
\label{equ:pi}
\end{align}
(see Figure \ref{figci}).
As we will explain subsequently, the subpaths $\overline \pa_i$ are involved 
in the computation of the lower and uppers
covers of $\pa$.
For example, for the lower covers,  $\overline\pa_i$
is the range of possible $\D$-step 
deletions  that correspond to the deletion of a $\U$-step from  the ascent 
$\U^{a_i}$ of $\pa$. 
\medskip

We say that $\pi\in\dP{n+1}$ is {\em irreducible} if it does not touch the 
the line $y=x$  except for the origin and the final point
and  {\em strongly irreducible} if it does not touch the 
the line $y=x+1$  except for the first and last step. 
If $\pa$ is irreducible then $b_1+\cdots+b_i<a_1+\cdots+a_i$ 
and if $\pa$ is  strongly irreducible then $b_1+\cdots+b_i+1<a_1+\cdots+a_i$ 
for all $i$ with $1\leq i<\ell$.

The situation where a Shi tableau {$T'$} can be obtained from a tableau {$T$} by two different bounce deletions is rather restrictive.
The lemma below describes the conditions when this is possible.
Pictorially, the following lemma implies that, if two bounce deletions  produce the same lower cover,  then they both delete a $\U$ and a $\D$-step from the same ascent and descent respectively.
The only exception is the case where $\pa$
is symmetric, where deletion of first column or last row produce the same lower cover  (see  Figure \ref{fig:lem_covers}). 
Given a word $w\in\{\U,\D\}^*$ let $\nup(w)$ denote the number of $\U$-steps in $w$, and let $\ndown(w)$ denote the number of $\D$-steps in $w$.
\begin{lemma}
	\label{Lemma:double_cover}
	Consider  $i,j\in[n]$ with $i<j$ and   let $k_i\in\{i-1,i\}$ and $k_j\in\{j-1,j\}$.
	Let $\pa$ be a Dyck path with $\U$ and $\D$-steps indexed by $\U_1,\ldots,\U_n,\D_1,\ldots, \D_n$, 
	and let $\pa'$ be a Dyck path  obtained from $\pa$ by deleting 
	$\U_i$ and $\D_{k_i}$   and also by deleting 
	$\U_j$ and $\D_{k_j}$.
	That is, $\pi'=\de{i}{k_i}(\pi)=\de{j}{k_j}(\pi)$.
	
	\begin{enumerate}[(i)]
		\item
		\label{Item:double_cover:normal}
		If $\D_{k_i}$ occurs after $\U_j$, then 
		$\U_i$ and $\U_j$  belong to the same ascent of $\pi$.
		Furthermore 
		$\D_{k_i}$ and $\D_{k_j}$  belong to the same descent of $\pi$.
		\item
		\label{Item:double_cover:exception}
		If $\D_{k_i}$ occurs before  $\U_j$, then the segment of $\pi$ connecting 
		$\U_i$ and $\D_{k_j}$
		is of the form $\U^r(\U\D)^{\ell}\D^t$ for some $r,\ell,t\in\NN$ with $\ell>0$.
		Furthermore this segment begins at height $i-k_i\in\{0,1\}$ and ends at height $j-k_j\in\{0,1\}$.
	\end{enumerate}
\end{lemma}

%
%

\begin{proof}
	First assume that we are in case \refi{double_cover:normal}.
	Then there exist unique (possibly empty) words $\alpha,x,y,z,\beta\in\{\U,\D\}^*$ such that
	\begin{eq*}
		\pi'=\alpha xyz\beta
		\qquad\text{and}\qquad
		\pi
		=\alpha \U xy\D z\beta 
		=\alpha x\U yz\D \beta 
	\end{eq*}
	such that $\alpha $ contains $i-1$ $\U$-steps, $\alpha x$ contains $j-1$ $\U$-steps, $\alpha \U xy$ contains $k_i-1$ $\D$-steps, and $\alpha x\U yz$ contains $k_j-1$ $\D$-steps.
	In particular the two words
	\begin{eq*}
		\begin{array}{*{13}c}
			\U & x_1 & \cdots & x_{r-1} & x_r & y_1 & \cdots & y_s & \D & z_1 & \cdots & z_{t-1} & z_t\\
			x_1 & x_2 & \cdots & x_{r} & \U & y_1 & \cdots & y_s & z_1 & z_2 & \cdots & z_t & \D
		\end{array}
	\end{eq*}
	agree.
	Consequently $x=\U^r$ and $z=\D^t$ for some $r,t\in\NN$.
	
	Next assume we are in case \refi{double_cover:exception}.
	Then there exist unique (possibly empty) words $\alpha ,x,y,z,\beta \in\{\U,\D\}^*$ such that
	\begin{eq*}
		\pi'=\alpha xyz\beta 
		\qquad\text{and}\qquad
		\pi
		=\alpha \U x\D yz\beta 
		=\alpha xy\U z\D \beta 
	\end{eq*}
	such that $\alpha $ contains $i-1$ $\U$-steps, $\alpha \U x$ contains $k_i-1$ $\D$-steps, $\alpha xy$ contains $j-1$ $\U$-steps, and $\alpha xy\U z$ contains $k_j-1$ $\D$-steps.
	In particular the two words
	\begin{eq*}
		\begin{array}{*{15}c}
			\U & x_1 & \cdots & x_{r-1} & x_r & \D & y_1 & \cdots & y_{s-2} & y_{s-1} & y_s & z_1 & \cdots & z_{t-1} & z_t\\
			x_1 & x_2 & \cdots & x_{r} & y_1 & y_2 & y_3 & \cdots & y_s & \U & z_1 & z_2 & \cdots & z_t & \D
		\end{array}
	\end{eq*}
	agree. It follows from
	\begin{eq*}
		\D =y_2=y_4=\dots
	\end{eq*}
	and
	\begin{eq*}
		\U=y_{s-1}=y_{s-3}=\dots
	\end{eq*}
	that $s$ is even, say $s=2\ell$ for some $\ell\in\NN$, and $y=(\U\D)^{\ell}$.
	The assumption that the $j$-th $\D$-step in $\pi$ occurs after the $k_i$-th $\U$-step implies $\ell>0$.
	Moreover $x=\U^r$ and $z=\D^t$ for some $r,t\in\NN$.
	We have
	\begin{eq*}
		i-k_i
		=\nup(\alpha )-\ndown(\alpha \U x)
		=\nup(\alpha )-\ndown(\alpha ).
	\end{eq*}
	Thus $\alpha $ ends at height $i-k_i$.
	Similarly
	\begin{eq*}
		j-k_j
		=\nup(\alpha xy)-\ndown(\alpha xy\U z)
		=\nup(\alpha xy\U z\D)-\ndown(\alpha xy\U z\D).
	\end{eq*}
	Thus $\beta $ starts at height $j-k_j$.
\end{proof}

\begin{figure}
	\centering
\begin{tikzpicture}[scale=0.34]
	\begin{scope}[xshift=0cm,yshift=0cm]
		\draw[dotted] (0, 0) grid (0, 1);
		\draw[dotted] (0, 1) grid (1, 2);
		\draw[dotted] (0, 2) grid (2, 3);
		\draw[dotted] (0, 3) grid (3, 4);
		\draw[dotted] (0, 4) grid (4, 5);
		\draw[dotted] (0, 5) grid (5, 6);
		\draw[dotted] (0, 6) grid (6, 7);
		\draw[dotted] (6,7) grid (7, 7);
		\draw[rounded corners=1, color=azu, line width=3,opacity=0.4](2,5.5)--(5.5,5.5)--(5.5,7); 
		\draw[rounded corners=1, color=azu, line width=3,opacity=0.4](2,4.5)--(3.5,4.5)--(3.5,7); 
		\draw[rounded corners=1, color=azu, line width=1] 
		(0,0)--(0,3)--(1,3)--(1,4)--(2,4)--(2,6)--(3,6)--(3,7)--(7,7);
		\node at (4,-0.5){\fs (i)};
	\end{scope}    
	\begin{scope}[xshift=15cm,yshift=0cm]
		\draw[dotted] (0, 0) grid (0, 1);
		\draw[dotted] (0, 1) grid (1, 2);
		\draw[dotted] (0, 2) grid (2, 3);
		\draw[dotted] (0, 3) grid (3, 4);
		\draw[dotted] (0, 4) grid (4, 5);
		\draw[dotted] (0, 5) grid (5, 6);
		\draw[dotted] (0, 6) grid (6, 7);
		\draw[dotted] (6,7) grid (7, 7);
		\draw[rounded corners=1, color=azu, line width=3,opacity=0.4](0,.5)--(.5,.5)--(0.5,3); 
		\draw[rounded corners=1, color=azu, line width=3,opacity=0.4](4,6.5)--(6.5,6.5)--(6.6,7); 
		\draw[rounded corners=1, color=azu, line width=1] 
		(0,0)--(0,3)--(1,3)--(1,4)--(2,4)--(2,5)--(3,5)--(3,6)--(4,6)--(4,7)--(7,7);
		\node at (4,-0.5){\fs (ii)};
	\end{scope}    
	
\end{tikzpicture}
	\caption{The two situations described in Lemma \ref{Lemma:double_cover}}
	\label{fig:lem_covers}
\end{figure}
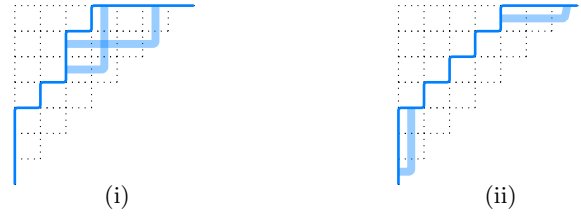

\subsection{Lower covers}
\label{ssec:lb}
Our  goal in this subsection is to describe  and count the elements of the 
set $\lcov{}{\pa}$ of lower covers of $\pa$. To this end, we write $\de{s}{k_s}$
where $k_s\in\{s,s-1\}$ and we group the lower covers of 
$\pa$ according to the ascent from which we delete a $\U$-step. 
More precisely, for each ascent $\U^{a_i}$ we define the set $\lcov{i}{\pa}\subseteq\dP{n}$
as follows: 
\begin{equation}
\begin{aligned}
\lcov{1}{\pa}:=&\{\de{s}{k_s}(\pa) \text{ for }
1\leq{s}\leq{a_1}\}  \hspace{5cm}\text{ if } i=1\\
\lcov{i}{\pa}:=&\{\de{s}{k_s}(\pa) \text{ for }
a_1+\cdots+a_{i-1}+1\leq{s}\leq{a_1+\cdots+a_{i}}\} \;\;\; \text{ if\; } 
2\leq{i}\leq{\ell}. 
\end{aligned} 
\label{setlci}
\end{equation} 
Clearly, the lower  covers of $\pi$ are all paths in the union 
\begin{align}
\label{lcovers}
\lcov{}{\pi}\,=  \bigcup_{i=1}^{\ell}   \lcov{i}{\pi}.  
\end{align}
To describe the paths in each $\lcov{i}{\pa}$
with respect to the original path $\pa$, notice that the 
range of  $s$ for  $\de{s}{k_s}(\pa)\in\lcov{i}{\pa}$
is the index of all  $\U$-steps $\U_s$ in $\U^{a_i}$. 
Since  $\de{s}{k_s}(\pa)$  acts by deleting the pair $\U_s,\D_{k_s}$
with $k_s\in\{s,s-1\}$, we deduce that
all    paths $\pa'\in\lcov{i}{\pa}$ are obtained from $\pa$
by replacing $U^{a_i}$ with $\U^{a_i-1}$ and deleting 
a $\D$-step $\D_{s'}$ where
\begin{equation}
\begin{aligned}
1\leq{s'}\leq{a_1}  \hspace{1.7cm} &\text{ if\; } i=1 \;\;\;\;  \text{ or }  
\\
a_1+\cdots+a_{i-1}\leq{s'}\leq{a_1+\cdots+a_{i}} &\text{ if \;} 
2\leq{i}\leq{\ell}. 
\end{aligned} 
\label{setlci2}
\end{equation} 
This means that the range of possible $\D$-step deletions in 
$\lcov{i}{\pa}$ 
is precisely the set of $\D$-steps of the  subpath $\overline \pa_i$ 
described in       \eqref{equ:pi}. 
In other words, $\lcov{i}{\pa}$
is  the  set of  paths obtained from $\pa$ after deleting a 
$\U$-step from the $i$-th ascent $\U^{a_i}$ and a $\D$-step from $\overline\pa_{i}$. 

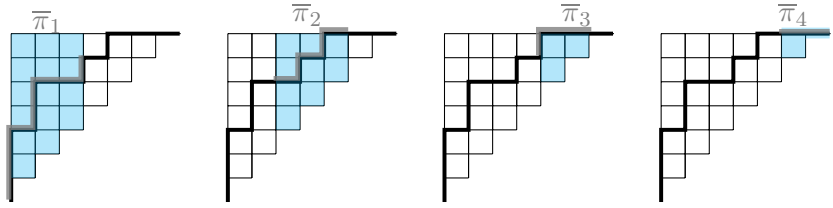
\begin{figure}
 \begin{tikzpicture}[scale=0.32]
	\begin{scope}[xshift=0cm]
		\draw (0,-1)--(0,0);
		\draw (-1,-1) grid (0,5);
		\draw (0,0) grid (1,5);
		\draw (1,1) grid (2,5);
		\draw (2,2) grid (3,5);
		\draw (3,3) grid (4,5);
		\draw (4,4) grid (5,5);
		\draw (5,5) grid (6,5);
		\draw[line width=1.5pt,opacity=1] (-1,-2)--(-1,1)--(0,1) 
		--(0,3)--(2,3)--(2,4)--(3,4)--(3,5)--(6,5);
		
		\fill[line width=1.7pt,opacity=0.3,cyan] 
		(-1,-2)--(-1,-1)--(0,-1)--(0,0)--(1,0)--(1,1)--(2,1)--(2,2)--(2,5)--(-1,5);
		\draw[line          
		width=1.9pt,opacity=0.8,gray,yshift=0.1cm,xshift=-0.1cm] 
		(-1,-2)--(-1,1)--(0,1)--(0,3)--(2,3)--(2,4);
		\node at (.5,5.5){\color{gray}$\overline\pa_1$};
	\end{scope}
	\begin{scope}[xshift=9cm]
		\draw (0,-1)--(0,0);
		\draw (-1,-1) grid (0,5);
		\draw (0,0) grid (1,5);
		\draw (1,1) grid (2,5);
		\draw (2,2) grid (3,5);
		\draw (3,3) grid (4,5);
		\draw (4,4) grid (5,5);
		\draw (5,5) grid (6,5);
		\draw[line width=1.5pt,opacity=1] (-1,-2)--(-1,1)--(0,1) 
		--(0,3)--(2,3)--(2,4)--(3,4)--(3,5)--(6,5);
		
		\fill[line width=1.7pt,opacity=0.3,cyan] 
		(1,1)--(2,1)--(2,2)--(3,2)--(3,3)--(4,3)--(4,5)--(1,5);
		
		\draw[line width=1.8pt,opacity=0.8,gray,yshift=0.12cm,xshift=-0.1cm] 
		(1,3.02)--(2,3.02)--(2,4)--(3,4)--(3,5.05)--(4.1,5.05);
		\node at (2.3,5.8){\color{gray}$\overline\pa_2$};
	\end{scope}
	\begin{scope}[xshift=18cm]
		\draw (0,-1)--(0,0);
		\draw (-1,-1) grid (0,5);
		\draw (0,0) grid (1,5);
		\draw (1,1) grid (2,5);
		\draw (2,2) grid (3,5);
		\draw (3,3) grid (4,5);
		\draw (4,4) grid (5,5);
		\draw (5,5) grid (6,5);
		\draw[line width=1.5pt,opacity=1] (-1,-2)--(-1,1)--(0,1) 
		--(0,3)--(2,3)--(2,4)--(3,4)--(3,5)--(6,5);
		\fill[line width=1.7pt,opacity=0.3,cyan] 
		(3,3)--(4,3)--(4,4)--(5,4)--(5,5)--(3,5);
		\draw[line width=1.8pt,opacity=0.8,gray,yshift=0.15cm,xshift=0cm] 
		(2.9,3.95)--(2.9,5.03)--(5.1,5.03);
		\node at (4.5,5.8){\color{gray}$\overline\pa_3$};
	\end{scope}
	\begin{scope}[xshift=27cm,yshift=0cm]
		\draw (0,-1)--(0,0);
		\draw (-1,-1) grid (0,5);
		\draw (0,0) grid (1,5);
		\draw (1,1) grid (2,5);
		\draw (2,2) grid (3,5);
		\draw (3,3) grid (4,5);
		\draw (4,4) grid (5,5);
		\draw (5,5) grid (6,5);
		\draw[line width=1.5pt,opacity=1] (-1,-2)--(-1,1)--(0,1) 
		--(0,3)--(2,3)--(2,4)--(3,4)--(3,5)--(6,5);
		\fill[line width=1.7pt,opacity=0.3,cyan] 
		(4,4)--(5,4)--(5,4.8)--(6,4.8)--(6,5.25)--(4,5.25);
		\draw[line width=1.7pt,opacity=0.8,gray,yshift=0.1cm,xshift=0cm] 
		(3.9,5)--(6,5);
		\node at (4.5,5.8){\color{gray}$\overline \pi_4$};
	\end{scope}
	
\end{tikzpicture}
	\caption{The path $\pi=\U^3\D\U^2\D^2\U\D\U\D^3$
		and the subpaths $\overline\pa_{i}$ defined in \eqref{equ:pi}. }
	\label{figci}
\end{figure}

\begin{lemma}
	Let  $\pi=\U^{a_1}\D^{b_1}\ldots\U^{a_\ell}\D^{b_\ell}$ be a Dyck path 
	in $\dP{n+1}$. 
	\begin{enumerate}[(i)]
		\item If $\pi$ is strongly irreducible with 
		$\pi\not=\U^a(\D\U)^{n+1-a}\D^{a}$, 
		the sets in the union \eqref{lcovers} are pairwise disjoint.
		\item If  $\pi=\U^a(\D\U)^{n+1-a}\D^{a}$ with  $2<a<n+1$ then 
		the sets in the union \eqref{lcovers} are pairwise disjoint
		except for $\lcov{1}{\pa}$ and $\lcov{\ell}{\pa}$
		which both contain the path $\U^{a}(\U\D)^{n-a}\D^{a}$.
	\end{enumerate}
	\label{lem:ULi_disjoint}  
\end{lemma}

\begin{proof}
	Let $i<j$ and suppose that $\pi'\in\lcov{i}{\pa}\cap\lcov{j}{\pa}$.
	Then $\pi'$ can be obtained from $\pi$ by two bounce deletions $\de{I}{k_I}$ and $\de{J}{k_J}$ such that the $I$-th and $J$-th $\U$-steps of $\pi$ do not belong to the same acsent of $\pi$.
	By Lemma~\ref{Lemma:double_cover} this implies that $\pi$ is of the form $\pi=\alpha U^r(DU)^{\ell}D^t\beta$ for some words $\alpha,\beta\in\{\U,\D\}^*$, and $r,\ell,t\in\NN$ with $\ell>0$, such that $\alpha$ ends at height $I-k_I \in\{0,1\}$, and $\beta$ starts at height $J-k_J \in\{0,1\}$.
	If $\pi$ is strongly irreducible
	then  the only possibilities for the words $\alpha, \beta$ are
	$\alpha=\U$ if $k_I=I-1$  or $\alpha$ is empty if $k_I=I$.
	Similarly $\beta=\D$ if $k_J=J-1$ and $\beta$ is empty if $k_J=J$.
	The claims follow.
\end{proof}

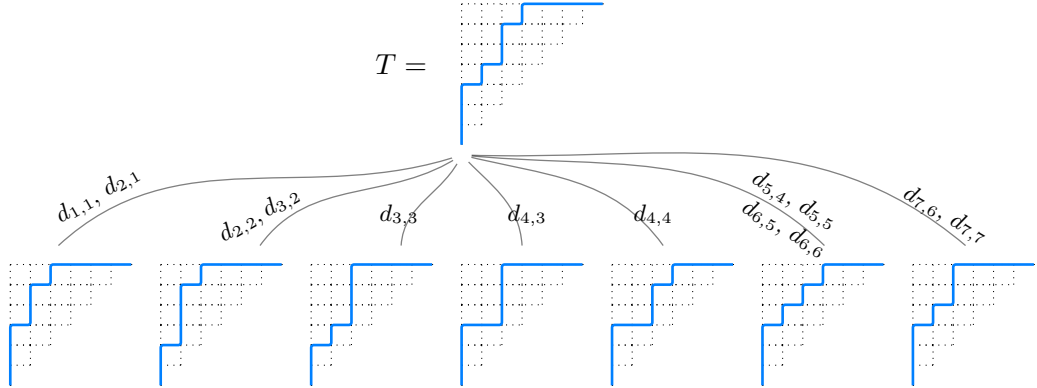
\begin{figure}
	\centering
\begin{tikzpicture}[scale=0.4]
	\begin{scope}[xshift=15cm,yshift=1cm,scale=0.67]
		\draw[dotted] (0, 0) grid (0, 1);
		\draw[dotted] (0, 1) grid (1, 2);
		\draw[dotted] (0, 2) grid (2, 3);
		\draw[dotted] (0, 3) grid (3, 4);
		\draw[dotted] (0, 4) grid (4, 5);
		\draw[dotted] (0, 5) grid (5, 6);
		\draw[dotted] (0, 6) grid (6, 7);
		\draw[dotted] (6,7) grid (7, 7);
		\draw[rounded corners=1, color=azu, line width=1] 
		(0,0)--(0,3)--(1,3)--(1,4)--(2,4)--(2,6)--(3,6)--(3,7)--(7,7);
		
		\draw[line width=0.5pt,color=black,opacity=0.5](0,-0.5)to[bend right,out=6](-20,-5);  
		\draw[line width=0.5pt,color=black,opacity=0.5](0,-0.5)to[bend right,out=10](-10,-5);  
		\draw[line width=0.5pt,color=black,opacity=0.5](0,-0.5)to[bend right,out=10](-3,-5);  
		\draw[line width=0.5pt,color=black,opacity=0.5](0,-0.5)to[bend left,out=10](3,-5);  
		\draw[line width=0.5pt,color=black,opacity=0.5](0,-0.5)to[bend left,out=10](10,-5);  
		\draw[line width=0.5pt,color=black,opacity=0.5](0,-0.5)to[bend left,out=8](18,-5);  
		\draw[line width=0.5pt,color=black,opacity=0.5](0,-0.5)to[bend left,out=8](25,-5); 
		\fill[white](0,-0.5)circle (0.5cm); 
		\node at (-3,4){$T=$};
		\node at (-18,-2.5){\fs \rotatebox{30}{$\de{1}{1},\,\de{2}{1}$}};
		\node at (-10,-3.5){\fs \rotatebox{30}{$\de{2}{2},\de{3}{2}$}};
		\node at (-3.1,-3.5){\fs \rotatebox{0}{$\de{3}{3}$}};
		\node at (3.25,-3.5){\fs \rotatebox{0}{$\de{4}{3}$}};
		\node at (9.5,-3.5){\fs \rotatebox{0}{$\de{4}{4}$}};
		\node at (16.5,-2.8){\fs \rotatebox{-30}{$\de{5}{4},\,\de{5}{5}$}};
		\node at (16,-4.2){\fs \rotatebox{-30}{$\de{6}{5},\,\de{6}{6}$}};
		\node at (24,-3.5){\fs \rotatebox{-30}{$\de{7}{6},\,\de{7}{7}$}};
		
	\end{scope}  
	\begin{scope}[xshift=0cm,yshift=-7cm,scale=0.67]
		\draw[dotted] (0, 0) grid (0, 1);
		\draw[dotted] (0, 1) grid (1, 2);
		\draw[dotted] (0, 2) grid (2, 3);
		\draw[dotted] (0, 3) grid (3, 4);
		\draw[dotted] (0, 4) grid (4, 5);
		\draw[dotted] (0, 5) grid (5, 6);
		\draw[dotted] (5,6) grid (6,6);
		\draw[rounded corners=1, color=azu, line width=1.00000000000000] 
		(0,0)--(0,3)--(1,3)--(1,5)--(2,5)--(2,6)--(6,6);
	\end{scope}
	\begin{scope}[xshift=5cm,yshift=-7cm,scale=0.67]
		\draw[dotted] (0, 0) grid (0, 1);
		\draw[dotted] (0, 1) grid (1, 2);
		\draw[dotted] (0, 2) grid (2, 3);
		\draw[dotted] (0, 3) grid (3, 4);
		\draw[dotted] (0, 4) grid (4, 5);
		\draw[dotted] (0, 5) grid (5, 6);
		\draw[dotted] (5,6) grid (6,6);
		\draw[rounded corners=1, color=azu, line width=1.00000000000000] 
		(0,0)--(0,2)--(1,2)--(1,5)--(2,5)--(2,6)--(6,6);
	\end{scope}
	\begin{scope}[xshift=10cm,yshift=-7cm,scale=0.67]
		\draw[dotted] (0, 0) grid (0, 1);
		\draw[dotted] (0, 1) grid (1, 2);
		\draw[dotted] (0, 2) grid (2, 3);
		\draw[dotted] (0, 3) grid (3, 4);
		\draw[dotted] (0, 4) grid (4, 5);
		\draw[dotted] (0, 5) grid (5, 6);
		\draw[dotted] (5,6) grid (6,6);
		\draw[rounded corners=1, color=azu, line width=1.00] 
		(0,0)--(0,2)--(1,2)--(1,3)--(2,3)--(2,6)--(6,6);
	\end{scope}
	\begin{scope}[xshift=20cm,yshift=-7cm,scale=0.67]
		\draw[dotted] (0, 0) grid (0, 1);
		\draw[dotted] (0, 1) grid (1, 2);
		\draw[dotted] (0, 2) grid (2, 3);
		\draw[dotted] (0, 3) grid (3, 4);
		\draw[dotted] (0, 4) grid (4, 5);
		\draw[dotted] (0, 5) grid (5, 6);
		\draw[dotted] (5,6) grid (6,6);
		\draw[rounded corners=1, color=azu, line width=1] 
		(0,0)--(0,3)--(2,3)--(2,5)--(3,5)--(3,6)--(6,6);
	\end{scope}
	\begin{scope}[xshift=15cm,yshift=-7cm,scale=0.67]
		\draw[dotted] (0, 0) grid (0, 1);
		\draw[dotted] (0, 1) grid (1, 2);
		\draw[dotted] (0, 2) grid (2, 3);
		\draw[dotted] (0, 3) grid (3, 4);
		\draw[dotted] (0, 4) grid (4, 5);
		\draw[dotted] (0, 5) grid (5, 6);
		\draw[dotted] (5,6) grid (6,6);
		\draw[rounded corners=1, color=azu, line width=1] 
		(0,0)--(0,3)--(2,3)--(2,6)--(6,6);
	\end{scope}
	\begin{scope}[xshift=25cm,yshift=-7cm,scale=0.67]
		\draw[dotted] (0, 0) grid (0, 1);
		\draw[dotted] (0, 1) grid (1, 2);
		\draw[dotted] (0, 2) grid (2, 3);
		\draw[dotted] (0, 3) grid (3, 4);
		\draw[dotted] (0, 4) grid (4, 5);
		\draw[dotted] (0, 5) grid (5, 6);
		\draw[dotted] (5,6) grid (6,6);
		\draw[rounded corners=1, color=azu, line width=1] 
		(0,0)--(0,3)--(1,3)--(1,4)--(2,4)--(2,5)--(3,5)--(3,6)--(6,6);
	\end{scope}
	\begin{scope}[xshift=30cm,yshift=-7cm,scale=0.67]
		\draw[dotted] (0, 0) grid (0, 1);
		\draw[dotted] (0, 1) grid (1, 2);
		\draw[dotted] (0, 2) grid (2, 3);
		\draw[dotted] (0, 3) grid (3, 4);
		\draw[dotted] (0, 4) grid (4, 5);
		\draw[dotted] (0, 5) grid (5, 6);
		\draw[dotted] (5,6) grid (6,6);
		\draw[rounded corners=1, color=azu, line width=1.00000000000000] 
		(0,0)--(0,3)--(1,3)--(1,4)--(2,4)--(2,6)--(6,6);
	\end{scope}
	
\end{tikzpicture}    

	\caption{The tableau $T$ corresponding to the path $\pa=\U\U\U\D\U\D\U\U\D\U\D\D\D\D$ and  its lower covers.
		Since $\pa$ is irreducible and  has four ascents, its lower covers are divided into 
		four pairwise disjoint sets $\lcov{i}{\pa}$ as follows. 
		The first tree lower covers belong to $\lcov{1}{\pa}$, the next two in $\lcov{2}{\pa}$, 	
		the next in $\lcov{3}{\pa}$ and the last one in $\lcov{4}{\pa}$. 
	}
	\label{fig:covers}
\end{figure}

Using the above lemma, we can prove the following proposition,
which counts the number of lower covers of all strongly irreducible Dyck paths.
\newpage 

\begin{theorem}
	\label{theor:peak_valley}
	\begin{enumerate}[(i)]
		\item[] 
		\item 
		If $\pi\in\dP{n+1}$ is strongly irreducible, the number of its lower covers 
		is     $$|\lcov{}{\pi}|=|peaks(\pi)| +|valleys(\pi)|$$
		unless $\pi=\U^a(\D\U)^{n+1-a}\D^{a}$ for some $a>2$
		in which case 
		$$|\lcov{}{\pi}|=|peaks(\pi)| +|valleys(\pi)|-1=2(n+1-a).$$
		\item  If $\pa=(\U\D)^{n+1}$ or $\pa=\U^{n+1}\D^{n+1}$ then 
		$|\lcov{}{\pa}|=1$.
		\item 
		If $\pa=\U(\U\D)^{n}\D$ and $n>0$ then $|\lcov{}{\pa}|=n$.
	\end{enumerate}
	\label{theor:UCclosedfor} 
\end{theorem}
\begin{proof}
	To prove (i), recall that all  Dyck paths in $\lcov{i}{\pa}$ are 
	obtained from $\pa$ by replacing $U^{a_i}$ by $\U^{a_i-1}$  and deleting a
	$\D$-step from $\overline \pa_{i}$. The different ways to delete 
	a $\D$-step from $\overline\pa_i$
	are  as many as the ways to 
	delete a $\D$-step from a different descent of $\overline\pa_i$,  which implies that 
	$|\lcov{i}{\pi}|=|{descents}(\overline\pa_{i})|$. 
	Therefore, if  $\pi\in\dP{n+1}$ is strongly irreducible and  $\pi\neq \U^a(\D\U)^{n+1-a}\D^{a}$, in view of Lemma \ref{lem:ULi_disjoint}\two{(i)}, 
	we have that   $|\lcov{}{\pa}|=  \sum_{i=1}^{\ell}|\lcov{i}{\pa}|=
	\sum_{i=1}^{\ell}|descents(\overline\pa_i)|$. 
	Next, notice that for all pairs $i,j$ we have $\overline \pa_i\cap 
	\overline\pa_j=\emptyset$ except for $j=i+1$ where 
	$\overline\pa_i\cap\overline\pa_{i+1}=
	\pa\cap\col{a_1+\cdots+a_i}$. 
	In other words, $\overline \pa_i, \overline \pa_{i+1}$
	share a column and thus a descent. 
	We can therefore rewrite the above sum as
	$|\lcov{}{\pa}|= |descents(\pa)|+ \ell-1=2\ell-1=|peaks(\pi)| +|valleys(\pi)|$.
	In the case where $\pi=\U^a(\D\U)^{n+1-a}\D^{a}$ with $a>2$, 
		  Lemma \ref{lem:ULi_disjoint}(ii) implies that 
		 $|\lcov{}{\pa}|=  \sum_{i=1}^{\ell}|\lcov{i}{\pa}|-1$ which, in turn, 
		 implies the second part of (i). 
	
	For (ii), it is immediate to see that the unique lower cover of $(\U\D)^{n+1}$
	and $\U^{n+1}\D^{n+1}$  is $(\U\D)^n$ and $\U^n\D^n$ respectively. 
	For (iii), we leave it to the reader to check that the lower covers of 
	$\pa=\U(\U\D)^n\D$ are  the paths obtained from 
	$\pa$ 
	by replacing a subword $\U\D\U\D$ by $\D\U$,
	or the path $\U(\U\D)^{n-1}\D$, which are altogether $n$. 
\end{proof}

An {\em irreducible decomposition}  
of a Dyck path $\pi$ is the unique way to write $\pa$ as a concatenation of Dyck paths 
$\pa_i$, where each $\pa_i$ is either an irreducible Dyck path
or a, possibly empty\footnote{
	i.e., we     count the  common endpoint between two consecutive 
	irreducible components as a component itself.
		We refer to it as an {\em trivial } component.}, maximal length sequence of peaks  at height $1$.
	In the first case we 
	say that   $\pi_i$ is an {\em irreducible component} 
	and in the latter,  a {\em connecting component} of $\pa$. 
	Adopting the symbol $\oplus$ for concatenation of paths, we write 
	\begin{align}
		\pa=\pa_1\oplus\pa_2\oplus\cdots\oplus\pa_k. \label{dec}
	\end{align}	
  Notice that, unless $\pa= (\U\D)^{n+1}$, the odd indexed components in 
\eqref{dec} are irreducible while the even indexed ones are connecting. 
If $\pa_1,\pa_3$ are irreducible Dyck paths 
and $\pa_2$ is a sequence of $k$ peaks at height 1, 
we abbreviate $\pa_1\oplus\pa_2\oplus\pa_3$ as 
$\pa_1\oplus_{k}\pa_3$ (see Figure \ref{fig:decomp}). 

Next, we introduce the symbol $\oplus'$ to denote concatenation of paths whose unique
 common point lies on $y=x+1$. More precisely, a {\em strongly irreducible decomposition} of an  {\em irreducible} Dyck path $\pi$
is the unique way to write 
\begin{align}
\pi=\U\pa_1\oplus'\pa_2\oplus'\cdots\oplus'\pa_k\D \label{sdec}	
\end{align}
as a concatenation of paths $\pa_i$, 
where each $\U\pa_i\D$ is either  a strongly irreducible Dyck path  or a
non empty  sequence of  peaks  at height $2$
(see Figure  \ref{fig:decomp2}). As before, we separate 
the components  of a strongly irreducible decomposition  
to {\em strongly irreducible} and 
{\em  connecting} ones and we note that, unless $\pa= \U(\U\D)^{n}\D$, the odd indexed components in \eqref{sdec} are irreducible while the even indexed ones are connecting. 
Finally, we  abbreviate $\U\pa_1\oplus'\pa_2\oplus'\pa_3\D$\; to\; $\U\pa_1\oplus_{k}'\pa_3\D$  
in the case where $\U\pa_2\D$ is a sequence of $k$ peaks at height 2.

\begin{figure}
\begin{tikzpicture}[scale=0.3]
	\begin{scope}[xshift=0cm,yshift=0cm,scale=0.75]
		\draw[dotted] (0, 0) grid (0, 1);
		\draw[dotted] (0, 1) grid (1, 2);
		\draw[dotted] (0, 2) grid (2, 3);
		\draw[dotted] (0, 3) grid (3, 4);
		\draw[dotted] (0, 4) grid (4, 5);
		\draw[dotted] (0, 5) grid (5, 6);
		\draw[dotted] (0, 6) grid (6, 7);
		\draw[dotted] (0, 7) grid (7,8);
		\draw[dotted] (0, 8) grid (8,9);
		\draw[dotted] (0, 9) grid (9,10);
		\draw[dotted] (0, 10) grid (10,11);
		\draw[dotted] (0, 11) grid (11,12);
		\draw[dotted] (0, 12) grid (12,13);
		\draw[dotted] (0, 13) grid (13,14);
		\draw[dotted] (0, 14) grid (14,15);
		\draw[dotted] (0, 15) grid (15,16);
		\draw[rounded corners=1, color=azu, line width=1.00000000000000] 
		(0,0)--(0,3)--(1,3)--(1,4)--(4,4)--(4,5)--(5,5)--(5,6)--(6,6)--(6,9)--(7,9)--
		(7,11)--(10,11)--(10,12)--(11,12)--(11,13)--(12,13)--(13,13)--(13,16)--(16,16);
		\node at (11.5,8.5){=};
		\fill[color=azu,opacity=0.8] (4,4) circle (6pt);
		\fill[color=azu,opacity=0.8] (6,6) circle (6pt);
		\fill[color=azu,opacity=0.8] (13,13) circle (6pt);
	\end{scope}
	\begin{scope}[xshift=12cm,yshift=5cm]
		\draw[dotted] (0, 0) grid (0, 1);
		\draw[dotted] (0, 1) grid (1, 2);
		\draw[dotted] (0, 2) grid (2, 3);
		\draw[dotted] (0, 3) grid (3, 4);
		\draw[rounded corners=1, color=azu, line width=1.00000000000000] 
		(0,0)--(0,3)--(1,3)--(1,4)--(4,4);
		\node at (4,2){$\oplus$};
	\end{scope}
	\begin{scope}[xshift=17.5cm,yshift=6cm,scale=0.9]
		\draw[dotted] (0, 0) grid (0, 1);
		\draw[dotted] (0, 1) grid (1, 2);
		\draw[rounded corners=1, color=azu, line width=1.00000000000000] 
		(0,0)--(0,1)--(1,1)--(1,2)--(2,2);
		\node at (3,1){$\oplus$};
	\end{scope}
	
	\begin{scope}[xshift=22cm,yshift=4cm,scale=0.9]
		\draw[dotted] (0, 0) grid (0, 1);
		\draw[dotted] (0, 1) grid (1, 2);
		\draw[dotted] (0, 2) grid (2, 3);
		\draw[dotted] (0, 3) grid (3, 4);
		\draw[dotted] (0, 4) grid (4, 5);
		\draw[dotted] (0, 5) grid (5, 6);
		\draw[dotted] (0, 6) grid (6, 7);
		\draw[rounded corners=1, color=azu, line width=1.00000000000000] 
		(0,0)--(0,3)--(1,3)--(1,5)--(4,5)--(4,6)--(5,6)--(5,7)--(7,7);
		
		\node at (6,3.2){$\oplus$};
		\node at (9,3.2){$\oplus$};
		\fill[color=azu,opacity=0.8] (7.5,3.2) circle (6pt);
	\end{scope}
	\begin{scope}[xshift=32cm,yshift=5.5cm,scale=0.8]
		\draw[dotted] (0, 0) grid (0, 1);
		\draw[dotted] (0, 1) grid (1, 2);
		\draw[dotted] (0, 2) grid (2, 3);
		\draw[rounded corners=1, color=azu, line width=1.00000000000000] 
		(0,0)--(0,3)--(3,3);
		\node at (4,1.5){$=$};
	\end{scope}

	\begin{scope}[xshift=37cm,yshift=5cm]
		\draw[dotted] (0, 0) grid (0, 1);
		\draw[dotted] (0, 1) grid (1, 2);
		\draw[dotted] (0, 2) grid (2, 3);
		\draw[dotted] (0, 3) grid (3, 4);
		\draw[rounded corners=1, color=azu, line width=1.00000000000000] 
		(0,0)--(0,3)--(1,3)--(1,4)--(4,4);
		\node at (4,2){$\oplus_2$};
	\end{scope}
	%
	\begin{scope}[xshift=43cm,yshift=4cm,scale=0.8]
		\draw[dotted] (0, 0) grid (0, 1);
		\draw[dotted] (0, 1) grid (1, 2);
		\draw[dotted] (0, 2) grid (2, 3);
		\draw[dotted] (0, 3) grid (3, 4);
		\draw[dotted] (0, 4) grid (4, 5);
		\draw[dotted] (0, 5) grid (5, 6);
		\draw[dotted] (0, 6) grid (6, 7);
		\draw[rounded corners=1, color=azu, line width=1.00000000000000] 
		(0,0)--(0,3)--(1,3)--(1,5)--(4,5)--(4,6)--(5,6)--(5,7)--(7,7);
		
		\node at (7,3.5){$\oplus_0$};
	\end{scope}
	\begin{scope}[xshift=50cm,yshift=5.5cm,scale=0.8]
		\draw[dotted] (0, 0) grid (0, 1);
		\draw[dotted] (0, 1) grid (1, 2);
		\draw[dotted] (0, 2) grid (2, 3);
		\draw[rounded corners=1, color=azu, line width=1.00000000000000] 
		(0,0)--(0,3)--(3,3);
	\end{scope}
	%
	%
\end{tikzpicture}
	\caption{The irreducible decomposition of a path. 
		The second and fourth components are connecting components while all others are irreducible components.
	}
	\label{fig:decomp}  
\end{figure}
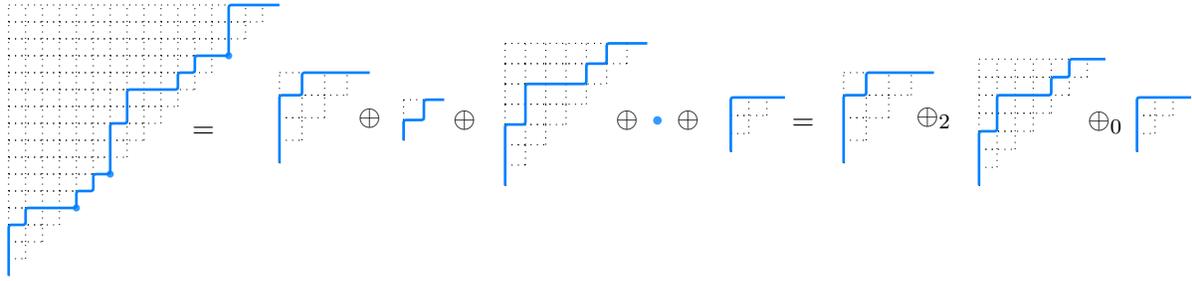

\begin{figure}
\begin{tikzpicture}[scale=0.3]
	\begin{scope}[xshift=0cm,yshift=2cm]
		\draw[dotted] (0, 0) grid (0, 1);
		\draw[dotted] (0, 1) grid (1, 2);
		\draw[dotted] (0, 2) grid (2, 3);
		\draw[dotted] (0, 3) grid (3, 4);
		\draw[dotted] (0, 4) grid (4, 5);
		\draw[dotted] (0, 5) grid (5, 6);
		\draw[dotted] (0, 6) grid (6, 7);
		\draw[dotted] (0, 7) grid (7,8);
		\draw[dotted] (0, 8) grid (8,9);
		\draw[dotted] (0, 9) grid (9,10);
		\draw[rounded corners=1, color=azu, line width=1.00000000000000] 
		(0,0)--(0,3)--(1,3)--(1,5)--(4,5)--(4,6)--(5,6)--(5,7)--(6,7)--(6,8)--(7,8)--
		(7,10)--(10,10);
		\node at (9,5){=};
		\draw[line width=7pt,opacity=0.1,rounded corners=0]
		(0,4.5)--(4.5,4.5)--(4.5,10);
		\fill[color=azu,opacity=0.7] (4,5) circle (5pt);
		\fill[color=azu,opacity=0.7] (7,8) circle (5pt);
	\end{scope}
	
	\begin{scope}[xshift=12cm,yshift=4cm]
		\draw[dotted] (0, 0) grid (0, 1);
		\draw[dotted] (0, 1) grid (1, 2);
		\draw[dotted] (0, 2) grid (2, 3);
		\draw[dotted] (0, 3) grid (3, 4);
		\draw[dotted] (0, 4) grid (4, 5);
		\draw[rounded corners=1, color=azu, line width=1.00000000000000] 
		(0,1)--(0,3)--(1,3)--(1,5)--(4,5);
		\draw[color=azu,dotted,line width=1.4](0,0)--(0,1);
		\draw[color=azu,dotted,line width=1.4](4,5)--(5.5,5);
		\node at (6,3){$\oplus'$};
	\end{scope}
	\begin{scope}[xshift=21cm,yshift=5cm]
		\draw[dotted] (0, 0) grid (0, 1);
		\draw[dotted] (0, 1) grid (1, 2);
		\draw[dotted] (0, 2) grid (2, 3);
		\draw[dotted] (0, 3) grid (3, 4);
		\draw[rounded corners=1, color=azu, line width=1.00000000000000] 
		(0,1)--(0,2)--(1,2)--(1,3)--(2,3)--(2,4)--(3,4);
		\draw[color=azu,dotted,line width=1.4](0,0)--(0,1);
		\draw[color=azu,dotted,line width=1.4](3,4)--(4,4);
		\node at (5,2){$\oplus'$};
	\end{scope}
	
	\begin{scope}[xshift=28cm,yshift=6cm]
		\draw[dotted] (0, 1) grid (1, 2);
		\draw[dotted] (0, 2) grid (2, 3);
		\draw[rounded corners=1, color=azu, line width=1.00000000000000] 
		(0,1)--(0,3)--(2,3);
		\draw[color=azu,dotted,line width=1.4](0,0)--(0,1);
		\draw[color=azu,dotted,line width=1.4](2,3)--(3,3);
		\node at (5,1){=};
	\end{scope}
	
	\begin{scope}[xshift=35cm,yshift=4cm]
		\draw[dotted] (0, 0) grid (0, 1);
		\draw[dotted] (0, 1) grid (1, 2);
		\draw[dotted] (0, 2) grid (2, 3);
		\draw[dotted] (0, 3) grid (3, 4);
		\draw[dotted] (0, 4) grid (4, 5);
		\draw[rounded corners=1, color=azu, line width=1.00000000000000] 
		(0,1)--(0,3)--(1,3)--(1,5)--(4,5);
		\draw[color=azu,dotted,line width=1.4](4,5)--(5.5,5);
		\draw[color=azu,dotted,line width=1.4](0,0)--(0,1);
		\node at (6,3){$\oplus_3'$};
	\end{scope}
	\begin{scope}[xshift=44cm,yshift=6cm]
		\draw[dotted] (0, 1) grid (1, 2);
		\draw[dotted] (0, 2) grid (2, 3);
		\draw[rounded corners=1, color=azu, line width=1.00000000000000] 
		(0,1)--(0,3)--(2,3);
		\draw[color=azu,dotted,line width=1.4](0,0)--(0,1);
		\draw[color=azu,dotted,line width=1.4](2,3)--(3,3);
	\end{scope}
\end{tikzpicture}
	\caption{Decomposition of an irreducible path into its  strongly 
		irreducible components.}
	\label{fig:decomp2}  
\end{figure}
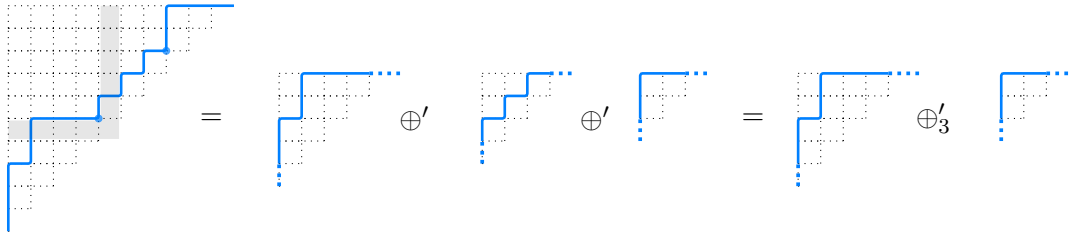

\begin{theorem}
Assume that $\pa\in \dP{n+1}$ with  $\pa\not \in \{(\U\D)^{n+1}, \U(\U\D)^{n}\D$\}. 
	\begin{enumerate}[(i)]
		\item  If $\pa$ is a Dyck path with  irreducible decomposition  
		$\pa=\pa_1\oplus\pa_2\oplus\cdots\oplus\pa_k$ then 
			\begin{align}
			|\lcov{}{\pa}|= \lfloor \tfrac{k}{2} \rfloor+
			\sum_{\substack{\pi_i\text{ irreducible}\\\text{component}     
			}}\!\!\!|\lcov{}{\pa_i}|.   
			\label{equ:ucdecom} 
		\end{align}
		\medskip 
		\item 
		If $\pa$ is an irreducible  Dyck path with strongly  irreducible 
		decomposition   $\pa=\U\pa_1\oplus'\cdots\oplus'\pa_{k}\D$
		then
		\begin{align}
		|\lcov{}{\pa}|&=
		k-1 +\sum_{i=1}^{k} |\lcov{}{\U\pi_i\D}|  \label{equ:str_irr}  \\
		&= 
		k-1 +\sum_{\substack{
				\U\pi_i\D\\
				\text{ strongly}\,\,\,\,\,\\
				\text{irreducible }}}\!\!\!\!\!\!  |\lcov{}{\U\pi_i\D}|+
		\sum_{\substack{\U\pi_i\D=\\\U(\U\D)^{k_i}\D}}  k_i. 
		\label{equ:str_irr2}
		\end{align}
	\end{enumerate}
\end{theorem}
\begin{proof}
	For (i), notice  that the lower covers of $\pa$ are obtained  by applying a bounce deletion according to  one of the  following ways:
	\begin{itemize}
		\item the deletion affects only an irreducible component,
		\item the deletion affects only a connecting component,
		\item the deletion eliminates the last $\D$-step of an irreducible  component and the first $\U$-step of the next (non-trivial connecting) component, 
		\item  the deletion eliminates the  $\D$-step and the  $\U$-step 
		between   a trivial component. 
			\end{itemize}
		In the first case we have a contribution of 
	$$ \sum_{\substack{\pi_i\text{ irreducible}\\\text{component}     
	}}\!\!\!|\lcov{}{\pa_i}|= 
\sum_{i=0}^{\lfloor \tfrac{k}{2}\rfloor} |\lcov{}{\pa_{2i+1}}|
$$
 lower covers. 
 In next two cases, since connecting components have a unique lower cover (Theorem \ref{theor:peak_valley}(ii)),
 the effect of the deletion on the original path $\pa$ is the same for each connecting component.   
 Therefore, if $k'$ is the number of  non-trivial connecting components, we have a contribution of $k'$ in the number of lower covers of $\pa$. 
 Finally, in the last case, we have to add to $|\lcov{}{\pa}|$ the number $k''$ of trivial  connecting components. Since connecting and irreducible components alternate,
    the number $k'+k''$
	of connecting components is $\left\lfloor \tfrac{k}{2}\right\rfloor$. 
	Hence, summing all the cases discussed above, we obtain  \eqref{equ:ucdecom}.

	For (ii), the lower covers of $\pa$ belong to one of the following 
	categories: either  they occur from bounce deletions on a single 
	component $\U\pa_{i}\D$ (the component being either strongly irreducible or of type $\U(\U\D)^{k_i}\D$)
	or they interact between two consecutive components 
	by deleting the last  $\U$-step of   $\U\pa_i\D$ and the first  $\D$-step of $\U\pa_{i+1}\D$
	(as the deletion shown in Figure \ref{fig:decomp2} on the left).
	More precisely, 
	if $(a-1,a)$ is the common endpoint between the two consecutive components 
	$\pa_i\oplus'\pa_{i+1}$ then $\de{a}{a}(\pa)$ is such  a lower cover. 
	The first case contributes $\sum_{i=1}^{k} |\lcov{}{\U\pi_i\D}|$ and 
	the second contributes $k-1$,  in the total number \eqref{equ:str_irr} of 
	lower covers  of $\lcov{}{\pa}$. 
	We can further refine \eqref{equ:str_irr} to  \eqref{equ:str_irr2}
	by distinguishing the type of each component $\pa_i$. 
\end{proof}

\subsection{Upper covers}
\label{ssec:ub}
Our next goal is the  computation of the upper covers of a Dyck path  
$\pa\in\dP{n+1}$. 
To do so,  we have to consider all the ways to  insert a 
$U$ and a $D$-step in $\pi$ so that the resulting path $\pa'$ is a Dyck 
path in $\dP{n+2}$ satisfying $\de{j}{k_j}(\pa')=\pa$ for some $j=1,\ldots,n+2$ and $k_j\in\{j-1,j\}$.
In view of the definition of $\de{j}{k_j}$, if 
the inserted $\U$ is the $j$-th $\U$-step of the new path 
$\pi'$ then the inserted $\D$ should be  either  the  $j$-th or the 
$(j-1)$-st $\D$-step of 
$\pa'$. 
Described more graphically, 
if $p=(x,j-1)$ is  the integer point
(other than  $(0,0)$ or $(n+1,n+1)$) where we insert the new $\U$-step, the range of possible insertions of the new  $\D$-step  is  the  subpath $\pa\cap\col{j-1}$  contained in the $(j-1)$-th column of  $\pa$. 
We can describe the above as a {\em bounce insertion}: we draw a 
horizontal line starting at the point $p$
at which we want to insert $\U$, 
until it hits the main  diagonal  and then we turn vertically up 
selecting the portion of the path contained in the   
column at which the horizontal line terminated. 
The selected subpath is the range of insertion of $\D$, i.e., if we insert a $\U$-step at $p$ then  we can insert a $\D$-step at any integer point of the above subpath  (see Figure  \ref{fig:bounce_insert}).  
If $p=(i,i)$, $1\leq{i}\leq{n}$, is on the main diagonal, then the bounce insertion described 
above reduces to selecting the portion of the path at  column $\col{i}$, i.e., $\pa\cap\col{i}$. 
Trivially,  if $p=(0,0)$ or $(n+1,n+1)$ the only possibilities are to insert $\U\D$ at $p$. 

\begin{figure}
  \begin{tikzpicture}[scale=0.32]
	\begin{scope}[xshift=0cm]
		\draw (0,0) grid (1,10);
		\draw (1,1) grid (2,10);
		\draw (2,2) grid (3,10);
		\draw (3,3) grid (4,10);
		\draw (4,4) grid (5,10);
		\draw (5,5) grid (6,10);
		\draw (6,6) grid (7,10);
		\draw (7,7) grid (8,10);
		\draw (8,8) grid (9,10);
		\draw (9,9) grid (10,10);
		
		\draw[line width=1.5pt,opacity=0.8]
		(0,-1)--(0,3)--(2,3)--(2,5)--(6,5)--(6,7)--(7,7)--(7,9)--(7,10)--(11,10);
		\fill[azu] (6,5)circle(7pt);   
		\draw[line width=2pt,azu,opacity=0.8] (5,5)--(6,5)--(6,7);
		
		\draw[line width=10pt,opacity=0.4,gray](5.5,5)--(5.5,10);  
		
		\draw[<-,line width=1pt,color=azu] (6,4.7) to[bend left,out=-20] (9,3.2);
		\node at (9.5,3){\fs insert  $\U$};
		\node at (9.5,2){\fs at $(6,6)$};
		\draw[->,line width=1pt,color=azu] (5.5,0.3) to[bend left,out=-20] 
		(5.7,6.3);
		\node[text width=3.5cm] at (9.5,0){\fs $\pa\cap\col{6}:$ range of };
		\node[text width=3.5cm] at (9.5,-1){\fs  insertions of $\D$};
	\end{scope}
	\begin{scope}[xshift=16cm]
		\draw (0,0) grid (1,10);
		\draw (1,1) grid (2,10);
		\draw (2,2) grid (3,10);
		\draw (3,3) grid (4,10);
		\draw (4,4) grid (5,10);
		\draw (5,5) grid (6,10);
		\draw (6,6) grid (7,10);
		\draw (7,7) grid (8,10);
		\draw (8,8) grid (9,10);
		\draw (9,9) grid (10,10);
		
		\draw[line width=1.5pt,opacity=0.8]
		(0,-1)--(0,3)--(2,3)--(2,5)--(3,5)--(3,7)--(6,7)--(6,8)--(7,8)--(7,10)--(11,10);
		\fill[azu] (0,2)circle(7pt);   
		\draw[line width=2pt,azu,opacity=0.8] (2,3)--(2,5)--(3,5)--(3,7);
		
		\draw[line width=3pt,opacity=0.4,gray](0,2)--(3.05,2);    
		\draw[line width=10pt,opacity=0.4,gray](2.5,2)--(2.5,10);  
		
		\draw[<-,line width=1pt,color=azu] (0.2,1.9) to[bend left,out=-20] (3.5,0);
		\node at (5.5,0){\fs insert  $\U$};
		\node at (5.5,-1){\fs at $(0,3)$};
		
		\draw[->,line width=1pt,color=azu] (5,2) to[bend left,out=-20] 
		(2.5,4.8);
		\node[text width=3.5cm] at (10.5,3){\fs $\pa\cap\col{3}:$ range of };
		\node[text width=3.5cm] at (10.5,2){\fs  insertions of $\D$};
	\end{scope}
	\begin{scope}[xshift=32cm]
		\draw (0,0) grid (1,10);
		\draw (1,1) grid (2,10);
		\draw (2,2) grid (3,10);
		\draw (3,3) grid (4,10);
		\draw (4,4) grid (5,10);
		\draw (5,5) grid (6,10);
		\draw (6,6) grid (7,10);
		\draw (7,7) grid (8,10);
		\draw (8,8) grid (9,10);
		\draw (9,9) grid (10,10);
		
		\draw[line width=1.5pt,opacity=0.8]
		(0,-1)--(0,3)--(2,3)--(2,5)--(3,5)--(3,7)--(6,7)--(6,8)--(7,8)--(7,10)--(11,10);
		\draw[line width=2pt,azu,opacity=0.8](7,8)--(7,10)--(8,10);
		
		\fill[azu] (5,7)circle(7pt); 
		\draw[->,line width=1pt,color=azu] (4,2) to[bend left,out=-20] 
		(5,6.7);
		\node at (5,1.5){\fs insert  $\U$  };
		\node at (5,.5){\fs  at  $(5,8)$};
		
		\draw[line width=3pt,opacity=0.4,gray](5,7)--(8,7);    
		\draw[line width=10pt,opacity=0.4,gray](7.5,7)--(7.5,10);
		
		\draw[->,line width=1pt,color=azu] (9,5.6) to[bend right,out=20] 
		(7.5,9.8);
		\node[text width=3.5cm] at (12.5,5){\fs $\pa\cap\col{8}:$ range of };
		\node[text width=3.5cm] at (12.5,4){\fs  insertions of $\D$};
	\end{scope}
	
	%
	%
	%
	%
	%
	%
	%
	%
	%

\end{tikzpicture}
	\caption{Bounce insertions:
		if we add a $\U$-step at  $(x,j-1)\in\pa$ and a $\D$-step at any integer
		point of $\pa\cap{\col{j-1}}$, the resulting path $\pa'$
		satisfies $\de{j}{k_j}(\pa')=\pa$
		for some $k_j\in\{j-1,j\}$.   	}
	\label{fig:bounce_insert}  
\end{figure}
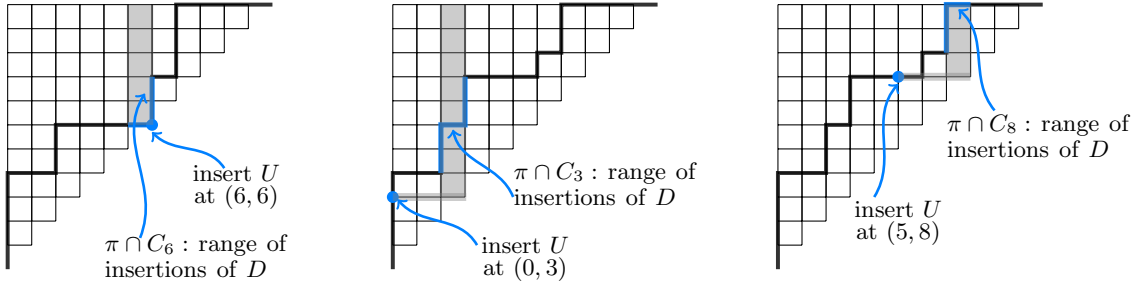
%
\smallskip 
%
%

We next assume that  $\pa$ is strongly irreducible and we group
all its upper covers  according to the $x$-coordinate of the integer  point 
where we insert the new $\U$-step.
This can be further partitioned into the case where 
we insert a new $\U$-step at an ascent $\U^{a_i}$ or the case where we 
insert a 
new $\U$-step between two consecutive $\D$-steps of a descent $\D^{b_i}$ of 
$\pa$. 
Altogether, we have the  following possibilities for the upper covers of 
$\pa$:
\medskip 

\noindent 
(i) 
$\ucov{1}{\pa}$ is the set of upper covers of  $\pa$ obtained by inserting the new $\U$-step at some integer point of    $\U^{a_1}$. 
In this case, the range of possible insertions of  the new  $\D$-step  is the set of all  integer points  of   the   subpath   of $\pa$ contained in the first $a_1$ columns, 
i.e., the subpath  $\pa\cap(\cup_{s=1}^{a_i}\col{s})=\overline \pa_1$.   \\
\noindent 
(ii) $\ucov{i}{\pa}$, $2\leq{i}\leq{\ell}$,  is the set of upper covers of
$\pa$ obtained by inserting the new $\U$-step at an integer point of  $\U^{a_i}$.
Since the endpoints of $\U^{a_i}$ have coordinates 
$(b_1+\cdots+b_{i-1},a_1+\cdots+a_{i-1})$,   $(b_1+\cdots+b_{i-1},a_1+\cdots+a_{i})$ 
the range of possible insertions of the new $\D$-step is the set of all integer points
on the  subpath 
$\pa\cap(\cup_{s=0}^{a_i}\col{a_1+\cdots+a_{i-1}+s})=\overline \pa_i$. 
%

\noindent 
(iii) 
$\ucovv{i}{\pa}$, $1\leq{i}\leq{\ell}$, is the set of upper covers of
$\pa$ obtained by inserting a new $\U$-step   between  a double descent 
$\D\D$ of   $\D^{b_i}$. If $\D^{b_i}$ consists of a single descent, i.e., 
$b_i=1$, then clearly $\ucovv{i}{\pa}=\emptyset$. 
If $b_i>1$, let $(x_{\circ},y_{\circ})$
be the point between  a double descent   of $\D^{b_i}$. Clearly   $y_{\circ}=a_1+\cdots+a_i$ which further implies that, when inserting a new $\U$ at  $(x_{\circ},y_{\circ})$,
the range of all possible   insertions of the new $\D$-step is the set of all integer points on  the subpath $\pa\cap\col{a_1+\cdots+a_i}$ 
of $\pa$ contained in the $(a_1+\cdots+a_i)$-th column  of the underlying tableau. 


\noindent 
(iv)
$\ucovv{\ell+1}{\pa}$ is the set containing the unique upper cover of 
$\pa$ we obtain by inserting  a $\U$-step (and a $\D$-step) at the end of the 
path $\pa$.  In other words,  
$\ucovv{\ell+1}{\pa}=\{\U^{a_1}\D^{b_1}\ldots\U^{a_\ell}\D^{b_\ell}\U\D\}.$
\medskip 

Finally,
since $\pa$ is strongly  irreducible it is not hard to see that all paths 
in the sets $\ucov{i}{\pa}$ and $\ucovv{i}{\pa}$ are again Dyck paths. 
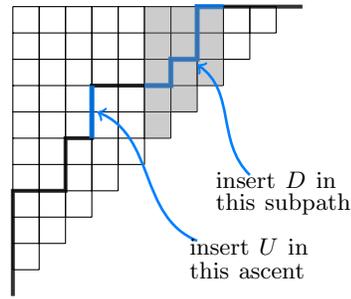
\begin{figure}
  \begin{tikzpicture}[scale=0.35]
	\draw (0,0) grid (1,10);
	\draw (1,1) grid (2,10);
	\draw (2,2) grid (3,10);
	\draw (3,3) grid (4,10);
	\draw (4,4) grid (5,10);
	\draw (5,5) grid (6,10);
	\draw (6,6) grid (7,10);
	\draw (7,7) grid (8,10);
	\draw (8,8) grid (9,10);
	\draw (9,9) grid (10,10);
	
	\draw[line width=1.5pt,opacity=0.8]
	(0,-1)--(0,3)--(2,3)--(2,5)--(3,5)--(3,7)--(6,7)--(6,8)--(7,8)--(7,10)--(11,10);
	\draw[line 
	width=2pt,azu,opacity=0.8](5,7)--(6,7)--(6,8)--(7,8)--(7,10)--(8,10);
	
	\draw[line width=2pt,azu,opacity=0.8](3,5)--(3,7);
	
	
	\draw[line width=10pt,opacity=0.4,gray](5.5,5)--(5.5,10);
	\draw[line width=10pt,opacity=0.4,gray](6.5,6)--(6.5,10);
	\draw[line width=10pt,opacity=0.4,gray](7.5,7)--(7.5,10);
	
	\draw[->,line width=1pt,color=azu] (9,3.6) to[bend right,out=20] 
	(7,7.8);
	\node[text width=3cm] at (12,3.4){\fs insert $\D$ in };
	\node[text width=3cm] at (12,2.5){\fs this  subpath};
	
	\draw[->,line width=1pt,color=azu] (7,1.1) to[bend right,out=30] 
	(3.2,6);
	\node[text width=3cm] at (11,0.9){\fs insert $\U$ in };
	\node[text width=3cm] at (11,0){\fs this  ascent};
	
\end{tikzpicture}
	\caption{ 
		The set $\ucov{3}{\pa}$ contains all paths obtained by increasing the third ascent
		from $\U^2$ to $\U^3$ and inserting a $\D$-step at any integer point on the indicated subpath. 
	}
	\label{fig:bounce_insertUC}  
\end{figure}
We, therefore, conclude that the upper covers of a strongly 
irreducible Dyck  path  $\pi$ are all paths in the  union 
\begin{equation}
\displaystyle
\ucov{}{\pa}=\bigcup_{i=1}^{\ell}  
\ucov{i}{\pa}\cup\bigcup_{i=1}^{\ell+1} 
\ucovv{i}{\pa}.
\label{ucovers}
\end{equation}
\medspace

The following lemma is the analogous of Lemma \ref{lem:ULi_disjoint} in the case of the upper covers.  
\begin{lemma}
	\begin{enumerate}[(i)]
		\item[]
		\item If $\pi'\not=\U^a(\D\U)^{n+1-a}\D^{a}$ is strongly irreducible then 
		the sets $\ucov{i}{\pa'}, \ucovv{j}{\pa'}$ are pairwise disjoint.
		\item If  $\pi'=\U^a(\D\U)^{n+1-a}\D^{a}$ with  $a>{2}$ then 
		the sets $\ucov{i}{\pa'}, \ucovv{j}{\pa'}$ are pairwise disjoint
		except for $\ucov{1}{\pa'}$ and $\ucov{\ell}{\pa'}$
		which both contain the path  $\U^{a}(\D\U)^{n+2-a}\D^{a}$.
	\end{enumerate}
	\label{lem:UCi_disjoint}  
\end{lemma}

\begin{proof}
	Suppose that $\pa$ lies in the intersection of two distinct sets among sets 
	$\ucov{i}{\pa'}$ and $\ucovv{j}{\pa'}$. 
	Then $\pa'$ can be obtained from $\pa$ by two different bounce deletions $\de{I}{k_I}$ and $\de{J}{k_J}$ such that the $I$-th and $J$-th $\U$-steps of $\pa$ do not belong to the same ascent of $\pa$.
	By Lemma~\ref{Lemma:double_cover} this implies that $\pa$ is of the form $\pa=\alpha U^r(DU)^{\ell}D^t\beta$ for some words $\alpha,\beta\in\{\U,\D\}^*$, and $r,\ell,t\in\NN$ with $\ell>0$, such that $\alpha$ ends at height $I-k_I$, and $\beta$ starts at height $J-k_J$.
	Moreover $\pa'=\alpha U^r(DU)^{\ell-1}D^t\beta$.
	If $\pa'$ is strongly irreducible we obtain that $\alpha=\U$ if $k_I=I-1$ and $\alpha$ is empty if $k_I=I$.
	Similarly $\beta=\D$ if $k_J=J-1$ and $\beta$ is empty if $k_J=J$.
	The claims follow.
\end{proof}

Using the fact that the sets in the union \eqref{ucovers} are disjoint if $\pa$ is strongly irreducible, we  arrive at a closed formula for $|\ucov{}{\pa}|$. 
In what follows,  we use the notation $\#\U(\pa)$ for the number of 
$\U$-steps of the path $\pa$.  
\begin{theorem}
	Let $\pa=\U^{a_1}\D^{b_1}\cdots\U^{a_\ell}\D^{b_\ell}$ be a Dyck path in $\dP{n+1}$. 
	\begin{enumerate}[(i)]
		\item If $\pi\in\dP{n+1}$ is strongly  irreducible with 
		$\pi\not=\U^a(\D\U)^{n+1-a}\D^{a}$,  
		the  number of its upper covers is 
		\begin{align}
		\label{up_irred}
		|\ucov{}{\pa}|= 2n+3+\sum_{ i=1 }^{\ell-1} b_i\cdot
		\bigl(\#\U(\pa\cap \col{a_1+\cdots+a_i})\bigr). 
		\end{align}
		\item  If  $\pi=\U^a\D\U\cdots\D\U\D^{a}$ with $a\geq{2}$,  then 
		$|\ucov{}{\pa}|$   
		is given by \eqref{up_irred} reduced by 1. 
		\item  If $\pi=(\U\D)^{n+1}$ or $\pa=\U^{n+1}\D^{n+1}$ then 
		$|\ucov{}{\pa}|=2(n+1)$.
		\item 
		If $\pi=\U(\U\D)^{n}\D$ then $|\ucov{}{\pa}|=4n-5$. 
	\end{enumerate}
	\label{prop:uc}
\end{theorem}

\begin{proof}
	If  $\pi\not=\U^a\D\U\cdots\D\U\D^{a}$ is strongly irreducible then
	the sets in \eqref{ucovers} are disjoint, hence we have 
	\begin{align}
	|\ucov{}{\pa}| & 
	=\sum_{i=1}^{\ell}|\ucov{i}{\pa}|+\sum_{i=1}^{\ell+1}
	|\ucovv{i}{\pa}|
	=1+\sum_{i=1}^{\ell}|\ucov{i}{\pa}|+\sum_{i=1}^{\ell}
	|\ucovv{i}{\pa}|. \label{i} 
	\end{align}
	Each set $\ucov{i}{\pa}$ has as many elements as there are ways to insert 
	a $\D$-step  in $\overline\pa_i$. 
	The different ways to insert a $\D$ in any path 
	are, either to insert a $\D$ in the beginning of the path or 
	to insert a $\D$ immediately after a $\U$-step. 
	This implies that $ |\ucov{i}{\pa}|=\#U(\overline\pa_i)+1$. 
	Likewise, by the definition of $\ucovv{i}{\pa}$, since for each double 
	descent  $\D\D$ in $\D^{b_i}$   the range of possible ways to insert a new $\D$-step is the subpath     $\pa\cap\col{a_1+\cdots+a_i}$, 
	we conclude that the sum in \eqref{i} becomes:    
	\begin{align}
	|\ucov{}{\pa}|   
	&= 1+
	\sum\limits_{i=1}^{\ell}\bigl(\#U(\overline\pa_i)+1\bigr)
	+\sum_{ i=1 }^{\ell} (b_i-1)
	\bigl(\#\U(\pa\cap \col{a_1+\cdots+a_i})+1\bigr) \notag\\
	& = 1+\sum_{i=1}^{\ell}b_i+
	\sum\limits_{i=1}^{\ell}\#U(\overline\pa_i)
	+\sum_{ i=1 }^{\ell} (b_i-1)
	\bigl(\#\U(\pa\cap \col{a_1+\cdots+a_i})\bigr) \notag\\
	& = 1+(n+1)+
	\sum\limits_{i=1}^{\ell}\#U(\overline\pa_i)
	+\sum_{ i=1 }^{\ell} (b_i-1)
	\bigl(\#\U(\pa\cap \col{a_1+\cdots+a_i})\bigr). \label{ii}
	\end{align}
	Counting the $\U$-steps of each $\overline \pa_i$  on each vertical line $x=s$ and taking into account that 
	the $\U$-steps in each $\overline\pa_i\cap\overline\pa_{i+1}=\pa\cap 
	\col{a_1+\cdots+a_i}$ are 
	counted twice, we have
	\begin{align*}
	\sum\limits_{i=1}^{\ell}\#U(\overline\pa_i)=
	\sum_{s=0}^{n+1}
	\Bigl(\#\U(\pa\cap\{x=s\})\Bigr)
	+\sum_{ i=1 }^{\ell-1} 
	\bigl(\#\U(\pa\cap \col{a_1+\cdots+a_i})\bigr).
	\end{align*}
	Thus,  we continue from 
	\eqref{ii} as 
	\begin{align}
	&  n+2+ \sum_{s=0}^{n+1}
	\Bigl(\#\U(\pa\cap\{x=s\})\Bigr)
	+\sum_{ i=1 }^{\ell-1} 
	\bigl(\#\U(\pa\cap \col{a_1+\cdots+a_i})\bigr)
	+\sum_{ i=1 }^{\ell} (b_i-1)
	\bigl(\#\U(\pa\cap \col{a_1+\cdots+a_i})\bigr) \label{b} \\
	& =\,   2 n+3 +\sum_{ i=1 }^{\ell-1} b_i
	\bigl(\#\U(\pa\cap \col{a_1+\cdots+a_i})\bigr), \label{bb}
	\end{align}
	where, to go from \eqref{b} to \eqref{bb}, we used the fact that
	the total number of $\U$-steps in the first sum of \eqref{b} is $n+1$, and that 
	the last column  has no $\U$-steps (i.e., $\#\U(\pa\cap \col{a_1+\cdots+a_\ell})=0$). 
	
	The claim in $(ii)$  is straightforward from Lemma\ref{lem:UCi_disjoint}(ii). 
	For $(iii)$, it is not hard to see that  all the upper covers of 
	$(\U\D)^{n+1}$ are  obtained either  by replacing a $\U\D$ by 
	$\U^2\D^2$, 
	or a $(\U\D)^2$ by $\U(\U\D)^2\D$,  or  $(\U\D)^{n+2}$. 
	In the case where $\pa=\U^{n+1}\D^{n+1}$, the upper covers  are either $\U^k\D\U^{n+2-k}\D^{n+1}$ for $1\leq{k}\leq{n+2}$, or $\U^{n+1}\D^k\U\D^{n+2-k}$ for $2\leq{k}\leq{n+1}$.

	Finally, for $(iv)$,  the upper covers
	$\pa'\in\ucov{}{\pa}$  of $\pa=\U(\U\D)^n\D$ fall in one of the following 
	categories:
	\begin{enumerate}[(i)]
		\item $\pa'=\U(\U\D)^{n+1}\D$,
		\item $\pa'$ is obtained from $\pa$ 
		by increasing a $\D\U$ to $\D^2\U^2$,
		\item $\pa'$ is obtained from $\pa$ 
		by increasing a $\U\D$ to $\U^2\D^2$,
		\item $\pa'$ is obtained from $\pa$ 
		by increasing a $\U\D\U\D$ to $\U^2\D\U\D^2$,
		\item $\pa'$ is obtained from $\pa$ 
		by increasing a $\D\U\D\U$ to $\U^2\U\D\U^2$. 
	\end{enumerate}  
	All the above,  sum up to $4n-5$ possible upper coves.   
\end{proof}
\medspace

\begin{remark}
	The sum in \eqref{up_irred} can be easily  computed as follows. 
	Extend each descent $\D^{b_i}$ of  $\pa$  until it hits 
	the   line $x=y$ and turn vertically up  selecting the subpath
	in the column   at which the extended descent terminated. 
	This is  the subpath 
	$\pa\cap\col{a_1+\cdots+a_i}$ whose number of $\U$-steps is  multiplied by 
	$b_i$ in \eqref{up_irred} (see, for example, \ref{fig:UCformula}(a)).
	\label{remark:UC}
\end{remark}

\begin{figure}
 \begin{tikzpicture}[scale=0.32]
	
	\begin{scope}[xshift=0cm]
		\draw (0,0) grid (1,10);
		\draw (1,1) grid (2,10);
		\draw (2,2) grid (3,10);
		\draw (3,3) grid (4,10);
		\draw (4,4) grid (5,10);
		\draw (5,5) grid (6,10);
		\draw (6,6) grid (7,10);
		\draw (7,7) grid (8,10);
		\draw (8,8) grid (9,10);
		\draw (9,9) grid (10,10);
		
		\draw[line width=1.5pt,opacity=0.8]
		(0,-1)--(0,3)--(2,3)--(2,5)--(3,5)--(3,7)--(6,7)--(6,8)--(7,8)--(7,10)--(11,10);
		
		\draw[line width=3pt,opacity=0.3,gray](0,2.9)--(4.05,2.9);    
		\draw[line width=10pt,opacity=0.3,gray](3.5,3)--(3.5,10);  
		\draw[line width=3pt,opacity=0.45,gray](2,4.9)--(6.05,4.9);    
		\draw[line width=10pt,opacity=0.45,gray](5.5,5)--(5.5,10);  
		\draw[line width=3pt,opacity=0.45,gray](3,6.9)--(8,6.9);    
		\draw[line width=10pt,opacity=0.45,gray](7.5,7)--(7.5,10); 
		
		\draw[line width=3pt,opacity=0.75,gray](6,7.9)--(9,7.9);    
		\draw[line width=10pt,opacity=0.75,gray](8.5,8)--(8.5,10);
		\node at (4,-2){$(a)$};
	\end{scope}
	\begin{scope}[xshift=15cm,yshift=1cm]
		\def\a{9}
		\draw (0,0) grid (1,\a);
		\draw (1,1) grid (2,\a);
		\draw (2,2) grid (3,\a);
		\draw (3,3) grid (4,\a);
		\draw (4,4) grid (5,\a);
		\draw (5,5) grid (6,\a);
		\draw (6,6) grid (7,\a);
		\draw (7,7) grid (8,\a);
		\draw (8,8) grid (9,\a);
		\fill (5,4)circle (6pt);
		
		\draw[line width=1.5pt,opacity=0.8] 
		(0,-1)--(0,3)--(2,3)--(2,4)--(5,4)
		--(5,5)--(5,6)--(5,8)--(6,8)--(6,9)--(10,9);
		\draw[line width=5pt,opacity=0.4,gray](5,4)--(5,8);
		\draw[line width=5pt,opacity=0.4,gray!40!azu](2,4)--(5,4);
		\node at (4,-3){$(b)$\;\;$\pa=\pa_1\oplus\pa_2$};
		\node at (5.6,3.8){$p$};
	\end{scope}
	\begin{scope}[xshift=28cm,yshift=1cm]
		\def\a{9}
		\draw (0,0) grid (1,\a);
		\draw (1,1) grid (2,\a);
		\draw (2,2) grid (3,\a);
		\draw (3,3) grid (4,\a);
		\draw (4,4) grid (5,\a);
		\draw (5,5) grid (6,\a);
		\draw (6,6) grid (7,\a);
		\draw (7,7) grid (8,\a);
		\draw (8,8) grid (9,\a);
		\fill (4,4)circle (6pt);
		\fill (5,5)circle (6pt);
		
		\draw[line width=5pt,opacity=0.4,gray](5,5)--(5,8);
		\draw[line 
		width=5pt,opacity=0.4,gray!40!azu](2,4)--(4,4); 
		\draw[line width=1.5pt,opacity=0.8] (0,-1)--(0,3)--(2,3)--(2,4)--(4,4)
		--(4,5)--(5,5)--(5,6)--(5,8)--(6,8)--(6,9)--(10,9);

		\node at (4,-3){$(c)$\;\;$\pa=\U\pa_1\oplus'_{1}\pa_2\D$};
	\end{scope}
	\begin{scope}[xshift=41cm,yshift=2cm]
		\def\a{8}
		\draw (0,0) grid (1,\a);
		\draw (1,1) grid (2,\a);
		\draw (2,2) grid (3,\a);
		\draw (3,3) grid (4,\a);
		\draw (4,4) grid (5,\a);
		\draw (5,5) grid (6,\a);
		\draw (6,6) grid (7,\a);
		\draw (7,7) grid (8,\a);
		\fill (4,4)circle (6pt);
		
		\draw[line width=1.5pt,opacity=0.8] (0,-1)--(0,3)--(2,3)--(2,4)--(4,4)
		--(4,7)--(5,7)--(5,8)--(9,8); 
		\draw[line width=5pt,opacity=0.4,gray!40!azu]
		(0,3)--(2,3)--(2,4)--(4,4); 
		\draw[line width=5pt,opacity=0.4,gray](4,4)
		--(4,7)--(5,7)--(5,8); 
		
		%
		%
		
		\draw [black,fill=white] plot [only marks, mark=*, mark size=3.9] coordinates 
		{(0,3) (1,3) };

		\draw[black,fill=gray] plot [only marks, mark=*, mark size=3.9]
		coordinates { (2,4)  (3,4)};

		\draw[black,fill=white] plot [only marks, mark=square*, mark size=3.5]
		coordinates {(4,5)  (4,6)  (4,7)};
		\draw[black,fill=gray!70] plot [only marks, mark=square*, mark size=3.5] coordinates 
		{ (5,8)};

		\node at (4,-4){$(d)$\;\;$\pa=\U\pa_1\oplus'\pa_2\D$};
	\end{scope}
\end{tikzpicture}
	\caption{(a)
		In view of Remark \ref{remark:UC}, the number of upper covers of the path on the left
		is $|\ucov{}{\pa}|=(2\cdot10+3)+(2\cdot{2}+1\cdot{1}+3\cdot{2}+1\cdot{0})$.}
	\label{fig:UCformula}
\end{figure}
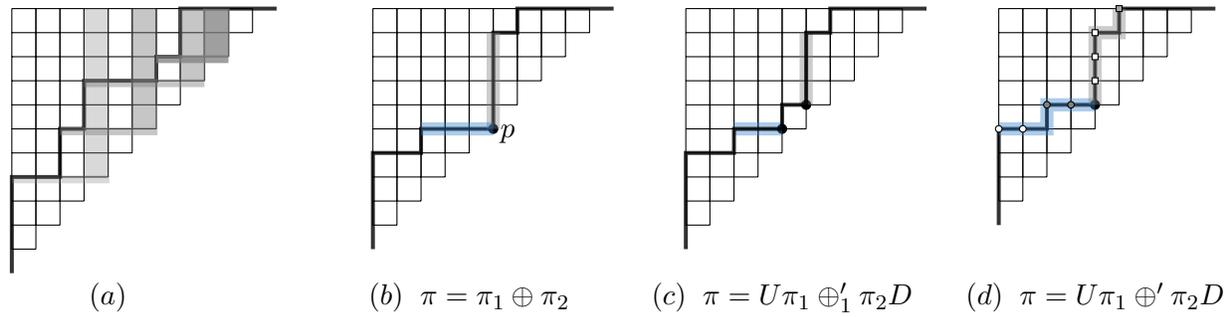
\medspace

The computation of the number of upper covers  of an arbitrary Dyck path $\pa$
is recursive. Unlike the case with the lower covers, we cannot have an explicit formula of  
$|\ucov{}{\pa}|$  in terms of 
the upper covers of the irreducible components of $\pa$, since
there exist  upper covers that occur 
by inserting a $\U$-step in a component of $\pa$ and a $\D$-step in the subsequent one. 
In the following proposition we
use the notation $\oplus,\oplus'_1$ without requiring that the components are 
irreducible. In other words, $\pa_1\oplus \pa_2$ is the concatenation of any two  Dyck paths 
$\pa_1,\pa_2$.
If $\U\pa_1\D,\U\pa_2\D$ are Dyck paths
then  $\U\pa_1\oplus' \pa_2\D$ is the  concatenation  of $\U\pa_1,\pa_2\D$,       
whereas 
$\U\pa_1\oplus'_1 \pa_2\D $ is the concatenation of $\U\pa_1\U$ and $\D\pa_2\D$.

\begin{theorem}
	\begin{enumerate}[(i)]
		\item[]
		\item   If $\pa_1,\pa_2$ are Dyck paths and $\pa=\pa_1\oplus\pa_2$ then 
		\begin{align}
		|\ucov{}{\pa}|=|\ucov{}{\pa_1}|+|\ucov{}{\pa_2}|-1+b\cdot{a}
		\label{equ:ub1}
		\end{align}
		where $\D^{b}$ is the last descent of $\pa_1$
		and $\U^a$ is the first ascent of $\pa_2$.       
		\item   If  $\pa=\U\pa_1\oplus'_{1}\pa_2\D$ then 
		\begin{align}
		|\ucov{}{\pa}|=|\ucov{}{\U\pa_1\D}|+|\ucov{}{\U\pa_2\D}|+(b+1)(a+1)
		\label{equ:ub2}
		\end{align}
		where $\D^{b+1}$ is the last descent of $\U\pa_1\D$
		and $\U^{a+1}$ is the first ascent of $\U\pa_2\D$.  
		\item[(iii)]
		If  $\pa=\U\pa_1\oplus'\pa_2\D$ with $\U\pa_1\D,\U\pa_2\D$ strongly irreducible Dyck paths
		then 
		\begin{align}  
		|\ucov{}{\pa}|=|\ucov{}{\U\pa_1\D}|+|\ucov{}{\U\pa_2\D}|-2+a'b+ab+ab'
		\label{equ:ub3}
		\end{align}
		where $\D^{b'}\U\D^b$ is the subpath in the last row of 
		$\U\pa_1\D$ and
		$\U^{a}\D\D^{a'}$ is the subpath in the first column of $\U\pa_2\D$. 
	\end{enumerate}  
	\label{prop:UCdecomp}
\end{theorem}
\begin{proof}

	For (i), first notice that all Dyck paths  of the form  
	$\pa'\oplus\pa_2$ with $\pa'\in\ucov{}{\pa_1}$ or  the form 
	$\pa_1\oplus\pa'$ with
	$\pa'\in\ucov{}{\pa_2}$ are upper covers of $\pa$. 
	Since  the upper cover $\pa_1\U\D\pa_2$ occurs in both cases, 
	we have   $|\ucov{}{\pa_1}|+|\ucov{}{\pa_2}|-1$  such paths. 
	We next need to count the upper covers that occur by inserting a $\U$-step 
	in $\pa_1$ and a $\D$-step in $\pa_2$ and are not encountered in the previous cases. 
	Assume that  $p=(k,k)$ is the intersection  point of $\pa_1,\pa_2$ (see 
	Figure \ref{fig:UCformula}(b)).
	Viewing these upper covers in terms of bounce insertions, one can 
	see that these are  the paths we get by  inserting a $\U$-step
	at any integer point, other than $p$, of the last descent of $\pa_1$ 
	(light blue, Figure \ref{fig:UCformula}(b))
	and  a $\D$-step at any integer point, other than $p$, of the first ascent of 
	$\pa_2$ (light gray, Figure \ref{fig:UCformula}(b)).
	This contributes the term $b\cdot{a}$  in \eqref{equ:ub1}.

	For (ii), notice that  all Dyck paths  of the form   
	$\U\pa'\oplus'_1\pa_2\D$ where
	$\U\pa'\D\in\ucov{}{\U\pa_1\D}$ or of the form $\pa_1\oplus'_1 \pa'$ where
	$\U\pa'\D\in\ucov{}{\U\pa_2\D}$ are upper covers of $\pa$. 
	This contributes $|\ucov{}{\pa_1}|+|\ucov{}{\pa_2}|$ in \eqref{equ:ub2}.   
	As before, we also have to count the upper covers  
	that occur by inserting a $\U$-step in $\pa_1$ and a $\D$-step in $\pa_2$. 
	Viewing these upper covers in terms of bounce insertions,  these are all 
	paths we get by  inserting a $\U$-step
	at any integer point  of the last descent of $\pa_1$ 
	(light blue, Figure \ref{fig:UCformula}(c))
	and  a $\D$-step at any integer point  of the first ascent of $\pa_2$ (light 
	gray, Figure \ref{fig:UCformula}(c)).
	This contributes  $(b+1)(a+1)$  in \eqref{equ:ub2}.    
	
	For (iii), 
	in order to compute $|\ucov{}{\pa}|$,   notice that 
	all Dyck paths  of the form  
	$\U\pa'\oplus'\pa_2\D$ with $\U\pa'\D\in\ucov{}{\U\pa_1\D}$ or  of the form $\U\pa_1\oplus'\pa'\D$ with 
	$\U\pa'\D\in\ucov{}{\U\pa_2\D}$ are upper covers of $\pa$. 
	Since  the upper covers $\U\pa_1\U\D\pa_2\D$ 
	and  $\U\pa_1\D\U\pa_2\D$ 
	occur in both instances, we have a total of $|\ucov{}{\pa_1}|+|\ucov{}{\pa_2}|-2$ 
	contributed in \eqref{equ:ub3}.
	We next count all upper covers  
	that occur by inserting a $\U$-step in $\pa_1$ and a $\D$-step in $\pa_2$.
	Viewing these upper covers in terms of bounce insertions, we have two categories: \\
	a) we insert a 
	$\U$-step in one of the first $b'$ integer points  of the last but one descent  $\D^{b'}$  of $\pa_1$ (i.e., points 
	\tikz{\draw[line width=0.7pt] (0,0) circle (2.5pt);}
	in Figure \ref{fig:UCformula}(c)) and insert a $\D$-step in one of the last $a$ integer points of 
	the first ascent $\U^a$  of $\pa_2$
	(i.e.,  points 
	\tikz{\draw[line width=0.7pt,black] (0,0) rectangle (.15,.15);} in Figure \ref{fig:UCformula}(c)). This contributes $a\cdot b'$ upper covers in \eqref{equ:ub3},  or \\
	b) we insert a  $\U$-step in one of the first $b$ integer points  of 
	the last descent $\D^{b}$ of $\pa_1$ (i.e., points 
	\tikz{\draw[line width=0.7pt,fill=gray] (0,0) circle (2.5pt);}
	in Figure \ref{fig:UCformula}(c)) and  insert a $\D$-step in one of 
	the last  $a$ integer points of $\U^a$  or the last  $a'$ integer points of $\U^{a'}$ 
	(i.e.,  points 
	\tikz{\draw[line width=0.7pt,black] (0,0) rectangle (.15,.15);} 
	or  \tikz{\draw[line width=0.7pt,black,fill=gray] (0,0) rectangle (.15,.15);} 
	in Figure \ref{fig:UCformula}(c)). This contributes $b\cdot (a+a')$ upper covers in \eqref{equ:ub3}.
\end{proof}

\begin{example}\rm  In  Figure  \ref{fig:UCformula}
	\begin{itemize}
		\item[(b)]  
		we depict the path $\pa=\pa_1\oplus\pa_2$ for   $\pa_1=\U^4\D^2\U\D^3$ and $\pa_2=\U^4\D\U\D^4$. 
		Using Theorem \ref{prop:uc} we compute 
		$|\ucov{}{\pa_1}|=11$ and  $|\ucov{}{\pa_2}|=10$. 
		In view of Theorem \ref{prop:UCdecomp}(i), we  have $|\ucov{}{\pa}|=11+10+3\cdot4-1=32$.
		\item[(c)]
		we depict the path $\pa=\U\pa_1\oplus'_1\pa_2\D$ for   $\U\pa_1\D=\U^4\D^2\U\D^3$ and $\U\pa_2\D=\U^4\D\U\D^4$. In view of Theorem \ref{prop:UCdecomp}(ii) we have  $|\ucov{}{\pa}|=11+10+3\cdot4=33$
		\item[(d)] 
		we depict the path $\pa=\U\pa_1\oplus'\pa_2\D$ for   $\U\pa_1\D=\U^4\D^2\U\D^3$ and $\U\pa_2\D=\U^4\D\U\D^4$. In view of Theorem \ref{prop:UCdecomp}(iii) we have   $|\ucov{}{\pa}|=11+10-2+ 
		2\cdot{3}+2\cdot{3}+2\cdot{1}=33$.
	\end{itemize}
	
	
\end{example}

\section{Pattern avoidance in Shi tableaux}
\label{sec:avoid}

Recall from  our definition in  Section \ref{sec:shi} that $T'$ occurs as a pattern in $T$  if $T'$ can be obtained   from $T$ after an iteration of  bounce deletions.
Otherwise, we say that $T$ {\em avoids} $T'$. 
%
%
We denote by $\av{}{T'}$ the subset of Shi tableaux   that avoid $T'$,
and by $\av{n}{T'}$ the  subset of Shi tableaux in  $\shiT{n}$ that avoid $T'$. If  
$|\av{n}{T'}|=|\av{n}{T}|$ for all $n\in\mathbb N$,
we say that $T$ and $T'$ are \emph{Wilf-equivalent}.
A fundamental problem on patterns is related to Wilf-equivalence and more precisely to finding 
classes of Wilf equivalent objects. 
\medskip

We begin this  section with an analysis of small patterns. 
More precisely, we present closed formulas for   $|\av{n}{T}|$, where 
$T$ is any Shi tableau of size 2.  We continue with  analogous results for   
certain tableau of size $k$, which can be considered as natural
generalizations of the ones of size 2. 
Finally, we compare 
the results we obtain with known results in the theory of permutation-patterns.
\subsection{Shi tableaux of size 2}
There are five Shi tableaux of size two 
\begin{align}
\label{size2}
\te{}=\U^3\D^3, \tg{}=\U^2\D\U\D^2, \tor{}=\U^2\D^2\U\D, 
\tv{}=\U\D\U^2\D^2 \text{ and } \tf{}=(\U\D)^3, 
\end{align}
which are divided into two Wilf-equivalence classes, as shown in the following proposition. 
\begin{proposition}
	\label{prop:avtableaux2}
	For every $n\geq 2$ we have
	\begin{enumerate}[(i)]
		\item $|\av{n}{\te{}}|=|\av{n}{\tf{}}|=2^n$, and
		\item $|\av{n}{\tg{}}|=|\av{n}{\tor{}}|=|\av{n}{\tv{}}|=\binom{n+1}{2}+1$.
	\end{enumerate}
\end{proposition}
To prove Proposition \ref{prop:avtableaux2} we need the following lemma which 
provides us with precise characterizations of pattern avoidance  
for each of the five tableau of size 2. 
\begin{lemma}
	\begin{enumerate}[(i)]
		\item  $T\in\shiT{n}$ avoids
		$\tf{}$ if and only if its bounce path $b(T)$
		has at most two return points.
		\item  $T\in\shiT{n}$ avoids  $\te{}$ if and only if its height is at most 2. 
		\item  $T\in\shiT{n}$ avoids $\tg{}$
		if and only if it has at most one peak at height $\geq{2}$. 
		\item   $T\in\shiT{n}$ avoids $\tv{}$ if and only if
		is  $T=\U^{n+1}\D^{n+1}$ or   $T=\U^{r}\D\cdots\D(\U\D)^{n-r}$ with  $1\leq{r}\leq{n}$. 
		Since $T\in \shiT{n}$, the subword between $\D\cdots\D$ is a permutation  of $r-1$ $\D$'s and a 
		single $\U$.
		\item[(iv\,$'$)] $T\in\shiT{n}$ avoids $\tor{}$ if and only if
		is  $T=\U^{n+1}\D^{n+1}$ or   $T=(\U\D)^{n-r}\U\cdots \U\D^r$ with  $1\leq{r}\leq{n}$. 
		Since $T\in \shiT{n}$, 
		the subword between $\U\cdots\U$ is a permutation  of $r-1$ $\U$'s and a 
		single $\D$.
	\end{enumerate}
	\label{lem:avtableaux2}
\end{lemma}
\begin{proof}
	For (i), notice that if the bounce path of $T$ 
	has at most two return points, say $(k,k)$ with $1\leq{k}\leq{n+1}$ and   
	$(n+1,n+1)$, then the columns $\col{k+1},\ldots,\col{n+1}$,
	as well as the rows $\row{1},\ldots,\row{k}$,
	are empty (see Figure \ref{fig:avoid_full2}).  Thus,
	if we want to arrive   at the   pattern $\tf{}$  we have to delete $k-1$ among the $k$ first 
	rows  of   $T$. This  eliminates $k-1$ among the first $k$ columns of 
	$T$, resulting in a tableau with at most  one non-empty column which, 
	clearly,  cannot 
	contain  the pattern $\tf{}$. 
	For the reverse, we leave  the reader to check that if $T$ has more than 
	two return points,    there exists an iteration of  bounce  deletions which leads 
	to   $\tf{}$. 
	
	For (ii), it is immediate to see that $T$ has height  at most 
	2  if  all boxes, except possibly the last in each row, are full. 
	To prove the claim in (ii), notice that if  $T$ has height more than 2, then
	it has a row ending with at least two empty boxes. 
	This forces the box below these two empty boxes to be empty as well, 
	therefore $T$ contains a subtableau of type  $\te{}$.
		Reversely, if no  row of $T$ has more than one empty 
	boxes (height 2) then no iteration of bounce deletions can
	increase the number of empty boxes in a row. 
	 Hence, $T$  contain no tableau 	with height greater than 2 and, as a result,
	 $T$ does not contain the pattern  $\te{}=\U^3\D^3$.
	
	For (iii) notice that if 
	$T$ has at most one peak at height $\geq{2}$, then 
	it can be written as   $T=(\U\D)^{a_1}\U^{a_2}\D^{a_2}(\U\D)^{a_3}$
	with $a_1+a_2+a_3=n+1$ (see Figure \ref{fig:avoidtg}(a)). 
	Arguing as in (ii), iteration of any bounce deletions  leads to 
	a   tableau of a similar   form, i.e., 
	$(\U\D)^{a'_1}\U^{a'_2}\D^{a'_2}(\U\D)^{a'_3}$
	with $a'_i\leq{a_i}$. 
	This implies that we can never obtain the pattern $\tg{}=\U^2\D\U\D^2$. 
		For the reverse, we prove  that if $T$ has more 
	than  one peaks at height $\geq{2}$, then $T$ contains the pattern $\tg{}$.
Indeed, consider two peaks of $T$ at height 2 and delete all columns 
preceding the first and all rows above the second. Clearly, we will be left 
with a Dyck path of type $\U^2\D\,w\,\U\D^2$ with $\#D(w)=\#U(w)=k$ for some  $k$. 
Applying the bounce deletions $\de{i}{i-1}$  for $i=3,\ldots,k+2$
we arrive at $\U^2\D\U\D^2$, i.e., at the pattern $\tg{}$.

	The claim in (iv) is obvious for $T=\U^{n+1}\D^{n+1}$. 
	We next show that if 
	$T=\U^{r}\D\cdots\D(\U\D)^{n-r}$ as described in the statement
	then it avoids the pattern $\tv{}=\U\D\U^2\D^2$ (see Figure~\ref{sfig:avoidtvk}(a)).
	Indeed, in order to obtain the  pattern $\U\D\U^2\D^2$
	we have to delete at least $r-1$ among the first $r$ $\U$-steps.    This  
	deletes  at least $r-1$ among the first $r$ columns of $T$, 
	resulting to one of the following tableaux: 
	$\U^2\D^2(\U\D)^{n-r}$, 
	$\U\D\U\D(\U\D)^{n-r}$ or $\U\D(\U\D)^{n-r}$, none of which  contains the pattern $\tv{}=\U\D\U^2\D^2$.
		For the reverse,  we need  to prove that if $T$ begins with $\U^r$
		and there exists  $r+2\leq{i}\leq n+1$ such that 
	the $i$-th row of $T$ has an empty box, then  there is a sequence  of bounce deletions leading to $\U\D\U^2\D^2$. Indeed, it is not difficult to see that 
	 there is a way to delete all rows indexed with $r+1\leq j\leq n+1$ except from 
	 the $i$-th row. 
	 Then, we will be left with a Dyck path of type $\U^r\D w \U\D^2$ 
	 where $w$ is a permutation of $r-1$ $\D$'s  and a singe $\U$. 
	 We leave the reader to verify that the latter contains the pattern $\U\D\U^2\D^2$.

	The statement in  (iv\,$'$) is completely analogous by symmetry.
\end{proof}
\begin{proof}[Proof of Proposition \ref{prop:avtableaux2}]
	For (i) we need to enumerate the tableaux  described in 
	Lemma \ref{lem:avtableaux2}(i),(ii). 
	We first  count the tableaux having at most two return points. 
	Since the second return point is always the endpoint $(n+1,n+1)$, 
	let us assume that the first return point of the bounce path is $(k,k)$
	with $1\leq{k\leq{n+1}}$. 
	To count the  tableau whose bounce path has first return point $(k,k)$ 
	it suffices to count the  paths within the rectangle defined by 
	$(1,k),(k,k),(k,n),(1,n)$ (see Figure~\ref{fig:avoid_full2}a). 
	The latter are   $\binom{n}{k-1}$, from which we deduce that 
	$|\av{n}{\tf{}}|=\sum_{k=1}^{n+1}\binom{n}{k-1}=2^n$.
	Next, notice that the tableaux with height at most   2 are just 
	those all whose boxes are full,  except possibly the last box of each row. 
	This immediately implies that $|\av{n}{\tf{}}|=2^n$.
	
\begin{figure}[h]
	\begin{subfigure}{0.42\textwidth}
		\label{sfig:avtf}
		\begin{center}
			\begin{tikzpicture}[scale=0.45]
				\fill[rounded corners=1, color=black, line width=1,opacity=0.1]
				(0, 3)--(0,8)--(3,8)--(3,3);
				(0,0)--(0,3)--(3,3)--(3,8)--(8,8);
				\draw[dotted] (0, 0) grid (0, 1);
				\draw[dotted] (0, 1) grid (1, 2);
				\draw[dotted] (0, 2) grid (2, 3);
				\draw[dotted] (0, 3) grid (3, 4);
				\draw[dotted] (0, 4) grid (4, 5);
				\draw[dotted] (0, 5) grid (5, 6);
				\draw[dotted] (0, 6) grid (6,7);
				\draw[dotted] (0, 7) grid (7,8);
				\draw[dotted] (0, 7) grid (7,8);
				\draw[dotted] (7,8) -- (8,8);
				\draw[black] (3,3) circle (3pt);
				\draw[black] (8,8) circle (3pt);
				\fill (0,0) circle (1.5pt);
				\node at (0,-0.3) {\scalebox{0.6}{ $(0,0)$}};
				\node at (3.8,3) {\scalebox{0.6}{$(k,k)$}};
				\node at (9.8,8) {\scalebox{0.6}{ $(n+1,n+1)$}};
				\node at (-1,1.5) {\scalebox{0.9}{$\U^k\D$}};
				\node at (6,8.5)  {\scalebox{0.9}{$\D^k$}};
				\draw[rounded corners=1, color=black, line width=1,opacity=0.8]
				(0,0)--(0,3)--(1,3);
				\draw[rounded corners=1, color=black, line width=1,opacity=0.8]
				(3,8)--(8,8);
				\draw[color=black,rounded corners=1.3,opacity=0.2,line width=0.7pt]
				(1,3)--(2,3)--(2,5)--(3,5)--(3,5)--(3,8);
				
			\end{tikzpicture}
		\end{center}
		\subcaption*{(a) \, Size 2: a tableau   avoids $\tf{}$ iff its bounce path has 
			at most 
			two return points.}
	\end{subfigure}
	\hfill
	\begin{subfigure}{0.45\textwidth}
		\begin{center}
			\begin{tikzpicture}[scale=0.35]
				\foreach \i in {0,1,2,3,4,5,6,7,8,9,10,11}{
					\draw[dotted](0,\i)grid (\i,\i+1);}
				\draw[dotted] (11,12) grid (12,12);
				\fill[color=black,line width=1pt,opacity=0.1,rounded corners=0.5]
				(0,0)--(0,3)--(3,3)--(3,5)--(5,5)--(5,8)--(8,8)--
				(8,10)--(10,10)--(10,12)--(12,12)--(0,12);

				\draw[rounded corners=1,color=black,line width=1,opacity=0.8](0,0)--(0,3)--(1,3);
				\draw[rounded corners=1,color=black,line width=1,opacity=0.8] (3,5)--(4,5);
				\draw[rounded corners=1,color=black,line width=1,opacity=0.8] (5,8)--(6,8);
				\draw[rounded corners=1,color=black,line width=1,opacity=0.8] (8,10)--(9,10);
				\draw[rounded corners=1,color=black,line width=1,opacity=0.8] (10,12)--(12,12);
				
				\draw[black] (3,3) circle (3pt);
				\draw[black] (5,5) circle (3pt);
				\draw[black] (8,8) circle (3pt);
				\draw[black] (10,10) circle (3pt);
				\draw[black] (12,12) circle (3pt);
				\node at (6,13.5){\color{white}.};
			\end{tikzpicture}
		\end{center}
		\subcaption*{(b)\, Generalization: a tableau  avoids $\tf{5}$ iff its bounce 
			path has at most 5 return points.}
		\label{sfig:avoidtek}
	\end{subfigure}
	\caption{}
	\label{fig:avoid_full2}
\end{figure}


	For (ii), we need to enumerate the tableaux  described in 
	Lemma \ref{lem:avtableaux2}(iii),(iv). 
	Recall from the proof of Lemma \ref{lem:avtableaux2}(iii) that 
	if a tableau $T$ has at most one peak at height $\geq{2}$, then 
	it can be written as   $T=(\U\D)^{a_1}\U^{a_2}\D^{a_2}(\U\D)^{a_3}$
	with $a_1+a_2+a_3=n+1$ and $a_2\geq{2}$ (see Figure \ref{fig:avoidtg}(a)). 
	Thus, we have  to count the  non-negative integer solutions of 
	$a_1+a_2+a_3=n+1$ with $a_2\geq{2}$. These are $\binom{n+1}{2}$. Including also the tableau $(\U\D)^{n+1}$, 
		all   whose   peaks are at height 1, we  conclude that  
	$|\av{n}{\tg{}}|=\binom{n+1}{2}+1$. 
	Finally,  in view of the characterisation in Lemma \ref{lem:avtableaux2}(iv), 
	for each $1\leq{r}\leq n$   we  count the words 
	with $r-1$ $\D$'s and a single $\U$. 
	Including  $T=\U^{n+1}\D^{n+1}$,  we have  that $|\av{n}{\tv{}}|=1+\sum_{r=1}^n r=1+\binom{n+1}{2}$. 
\end{proof}
\medspace

\begin{figure}[h]
	\begin{subfigure}[t!]{0.45\textwidth}
		\begin{center}
			\begin{tikzpicture}[scale=0.44]
				\draw[rounded corners=1, color=black, line width=1]
				(0,0)--(0,1)--(1,1)--(1,2)--(2,2)--(2,6)--(6,6)--(6,7)--(7,7)--(7,8)--(8,8);
				\draw[dotted] (0, 0) grid (0, 1);
				\draw[dotted] (0, 1) grid (1, 2);
				\draw[dotted] (0, 2) grid (2, 3);
				\draw[dotted] (0, 3) grid (3, 4);
				\draw[dotted] (0, 4) grid (4, 5);
				\draw[dotted] (0, 5) grid (5, 6);
				\draw[dotted] (0, 6) grid (6,7);
				\draw[dotted] (0, 7) grid (7,8);
				\draw[dotted] (0, 7) grid (7,8);
				\draw[dotted] (7,8) -- (8,8);
				\fill (0,0) circle (1.5pt);  \fill (8,8) circle (1.5pt);
				\foreach \i in {1,2,3,4,5,6,7}  {\draw[color=black,fill=black, opacity=0.7]
					(0.5,\i+0.5) circle (4pt);}
				\foreach \i in {1,2,3,4,5,6}
				{\draw[color=black,fill=black, opacity=0.7] (1.5,\i+1.5)circle (4pt);}
				\foreach \i in {2,3,4,5,6}
				{\draw[color=black,fill=black, opacity=0.7] (\i+0.5,7.5)circle (4pt);}
				\foreach \i in {2,3,4,5}
				{\draw[color=black,fill=black, opacity=0.7] (\i+0.5,6.5)circle (4pt);}
				\draw[color=black, opacity=0.7] (2.5,3.5) circle (4pt);
				\draw[color=black, opacity=0.7] (2.5,4.5) circle (4pt);
				\draw[color=black, opacity=0.7] (3.5,4.5) circle (4pt);
				\draw[color=black, opacity=0.7] (2.5,5.5) circle (4pt);
				\draw[color=black, opacity=0.7] (3.5,5.5) circle (4pt);
				\draw[color=black, opacity=0.7] (4.5,5.5) circle (4pt);
				\node at (-1.6,-0.2){\rotatebox{45}{\color{black}\fs \em height 1
						$\longrightarrow$}};
				\node at (-1.7,2.4){\rotatebox{45}{\color{black}\fs \em height 4
						$\longrightarrow$}};
			\end{tikzpicture}
		\end{center}
		\subcaption*{(a)\, Size 2: a tableau  avoids $\tg{}$ iff it has at most one 
			peak at 
			height   $\geq{2}$.  }
		\label{sfig:avoidtg}
	\end{subfigure}\hfill
	\begin{subfigure}[t!]{0.5\textwidth}
		\begin{center}
			\begin{tikzpicture}[scale=0.47]
				\foreach \i in {0,1,2,3,4,5,6,7}{
					\draw[dotted](0,\i)grid (\i,\i+1);}
				\draw[dotted] (7,8) grid (8,8);
				\draw[rounded corners=1, color=black, line width=1]
				(0,0)--(0,5)--(3,5)--(3,6)--(5,6)--(5,8)--(8,8);
				\draw[color=black!40](0,3)--(5,8);
				\draw[rounded corners=1, color=black!50,dashed, line width=1,opacity=1]
				(0,3)--(0,4)--(1,4)--(1,5)--(2,5);
				\node at (-1.6,1.3){\rotatebox{45}{\color{black!50}\fs \em height 3
						$\longrightarrow$}};
				\node at (0,-1.2){\color{white}.};
				\draw[color=black, opacity=0.7] (0.5,4.5) circle (4pt);
				\draw[color=black, fill=black, opacity=0.7] (0.5,5.5) circle (4pt);
				\draw[color=black, fill=black, opacity=0.7] (1.5,5.5) circle (4pt);
				\draw[color=black, fill=black, opacity=0.7] (2.5,5.5) circle (4pt);
				\draw[color=black, fill=black, opacity=0.7] (0.5,6.5) circle (4pt);
				\draw[color=black, fill=black, opacity=0.7] (1.5,6.5) circle (4pt);
				\draw[color=black, fill=black, opacity=0.7] (2.5,6.5) circle (4pt);
				\draw[color=black, fill=black, opacity=0.7] (3.5,6.5) circle (4pt);
				\draw[color=black, fill=black, opacity=0.7] (4.5,6.5) circle (4pt);
				\draw[color=black, fill=black, opacity=0.7] (0.5,7.5) circle (4pt);
				\draw[color=black, fill=black, opacity=0.7] (1.5,7.5) circle (4pt);
				\draw[color=black, fill=black, opacity=0.7] (2.5,7.5) circle (4pt);
				\draw[color=black, fill=black, opacity=0.7] (3.5,7.5) circle (4pt);
				\draw[color=black, fill=black, opacity=0.7] (4.5,7.5) circle (4pt);
				
				\draw[color=black, opacity=0.7] (0.5,1.5) circle (4pt);
				\draw[color=black, opacity=0.7] (0.5,2.5) circle (4pt);
				\draw[color=black, opacity=0.7] (1.5,2.5) circle (4pt);
				\draw[color=black, opacity=0.7] (0.5,3.5) circle (4pt);
				\draw[color=black, opacity=0.7] (1.5,3.5) circle (4pt);
				\draw[color=black, opacity=0.7] (2.5,3.5) circle (4pt);
				\draw[color=black, opacity=0.7] (1.5,4.5) circle (4pt);
				\draw[color=black, opacity=0.7] (2.5,4.5) circle (4pt);
				\draw[color=black, opacity=0.7] (3.5,4.5) circle (4pt);
				\draw[color=black, opacity=0.7] (3.5,5.5) circle (4pt);
				\draw[color=black, opacity=0.7] (4.5,5.5) circle (4pt);
				\draw[color=black, opacity=0.7] (5.5,6.5) circle (4pt);
				\draw[color=black, opacity=0.7] (5.5,7.5) circle (4pt);
				\draw[color=black, opacity=0.7] (6.5,7.5) circle (4pt);
				
			\end{tikzpicture}
		\end{center}
		\subcaption*{(b)\, Generalization: a tableau avoids $\tg{4}$ iff it has no 
			valley 
			at height      $\geq{3}$. }
		\label{sfig:2peaks}
	\end{subfigure}
	\caption{}
	\label{fig:avoidtg}
\end{figure}
\begin{figure}[h]
	\begin{subfigure}[t]{0.45\textwidth}
		\begin{center}
			\begin{tikzpicture}[scale=0.5]
				\foreach \i in {0,1,2,3,4,5,6,7}{\draw[dotted](0,\i)grid (\i,\i+1);}
				\draw[dotted] (7,8) grid (8,8);
				\foreach \i in {0,1,2,3}{\draw[color=black,line width=0.7pt,opacity=0.8]
					(\i+0.5,4.5) circle (4pt);}
				\foreach \i in {0,1,2}{\draw[color=black,line width=0.7pt,opacity=0.8]
					(\i+0.5,3.5) circle (4pt);}
				\foreach \i in {0,1}{\draw[color=black,line width=0.7pt,opacity=0.8]
					(\i+0.5,2.5) circle (4pt);}
				\draw[color=black,line width=0.7pt,opacity=0.8]
				(0.5,1.5) circle (4pt);
				\foreach \i in {5,6,7}{\draw[color=black,fill=black,opacity=0.8]
					(0.5,\i+.5) circle (4pt);}
				\foreach \i in {1,2,3,4,5,6}{\draw[color=black,fill=black, opacity=0.8]
					(\i+.5,7.5) circle (4pt);}
				\foreach \i in {1,2,3,4,5}{\draw[color=black,fill=black, opacity=0.8]
					(\i+.5,6.5) circle (4pt);}
				\foreach \i in {1,2,3,4}{\draw[color=black,fill=black, opacity=0.3]
					(\i+.5,5.5) circle (4pt);}
				\node at (-1,0.5){\s{ \em row 1}};
				\node at (-1,1.5){\s{ \em row 2}};
				\node at (-1,3){ $\vdots$};
				\node at (-1,4.5){\s{ \em row r}};
				\draw[rounded corners=1, color=black, line width=1,opacity=0.8]
				(0,0)--(0,5)--(1,5);
				\draw[rounded corners=1, color=black, line width=1,opacity=0.8]
				(5,6)--(6,6)--(6,7)--(7,7)--(7,8)--(8,8);
			\end{tikzpicture}
		\end{center}
		\subcaption*{(a)\, Size 2: Tableaux avoiding $\tv{}$. Gray bullets
			can be either full  or empty boxes. }
		\label{sfig:avoidtv}
	\end{subfigure}\hfill
	\begin{subfigure}[t]{0.5\textwidth}
		\begin{center}
			\begin{tikzpicture}[scale=0.45]
				\draw [color=black,decorate,decoration={brace,amplitude=6pt}](0,4) -- (0,9);
				\node at (-3.4,6.5)[text width=2cm]{\scalebox{0.75}{ \color{black} 
						$non$-$empty$
						$rows$}};
				\foreach \i in {0,1,2,3,4,5,6,7,8}{\draw[dotted](0,\i)grid (\i,\i+1);}
				\draw[dotted] (8,9) grid (9,9);
				\draw[color=black,fill=black,line width=.45cm,opacity=0.12](4,8.5)--(8,8.5);
				\draw[color=black,fill=black,line width=.45cm,opacity=0.12](4,7.5)--(7,7.5);
				\draw[color=black,fill=black,line width=.45cm,opacity=0.12](4,6.5)--(6,6.5);
				\draw[color=black,fill=black,line width=.45cm,opacity=0.12](4,5.5)--(5,5.5);
				
				
				\foreach \i in {4,5,6,7,8}{\draw[color=black,fill=black,opacity=0.8]
					(0.5,\i+.5) circle (4pt);}
				\foreach \i in {1,2,3,4,5,6,7}{\draw[color=black,fill=black, opacity=0.3]
					(\i+.5,8.5) circle (4pt);}
				\foreach \i in {1,2,3,4,5,6}{\draw[color=black,fill=black, opacity=0.3]
					(\i+.5,7.5) circle (4pt);}
				\foreach \i in {1,2,3,4,5}{\draw[color=black,fill=black, opacity=0.3]
					(\i+.5,6.5) circle (4pt);}
				\foreach \i in {1,2,3,4}{\draw[color=black,fill=black, opacity=0.3]
					(\i+.5,5.5) circle (4pt);}
				\foreach \i in {1,2,3}{\draw[color=black,fill=black, opacity=0.3]
					(\i+.5,4.5) circle (4pt);}
				\foreach \i in {0,1,2}{\draw[color=black,opacity=1](\i+.5,3.5)circle(4pt);}
				\foreach \i in {0,1}{\draw[color=black,opacity=1](\i+.5,2.5)circle(4pt);}
				\foreach \i in {0}{\draw[color=black,opacity=1](\i+.5,1.5)circle(4pt);}
				\draw[rounded corners=1, color=black, line width=1,opacity=0.8]
				(0,0)--(0,4)--(1,4);
				\node at (4.8,10.1){\rotatebox{65}{{\fs{\color{gray} $\col{n+2-\ell}$}}}};
				\node at (5.8,10.1){\rotatebox{65}{{\fs{\color{gray} $\col{n+3-\ell}$}}}};
				\node at (7,10.2){\color{gray}$\cdots$};
			\end{tikzpicture}
		\end{center}
		\subcaption*{(b)\, Generalization: 
			If $T$ has $\ell$ non-empty rows then,  $T$ avoids $\tv{k}$ iff the 
			tableau in  colums $\col{n+2-\ell},\ldots,\col{n+1}$ (shaded)
			avoids  $\te{k-1}.$  Gray bullets
			can be either full  or empty boxes.}
	\end{subfigure}
	\caption{}
	\label{sfig:avoidtvk}
\end{figure}

\subsection{Generalizations}
In this subsection  we generalize  the  five cases in \eqref{size2} to  the following Shi 
tableaux of size $k\geq{3}$:
\begin{align}
\label{sizek}
\te{k}=\U^{k+1}\D^{k+1},\, \tg{k}=\U^k\D\U\D^k,\, \tor{k}=\U^k\D^k\U\D, \,
\tv{k}=\U\D\U^k\D^k \text{ and }\, \tf{k}=(\U\D)^{k+1}. 
\end{align}
It turns out  that for  fixed $k\geq 3$  the tableaux in \eqref{sizek}
fall into two Wilf-equivalence classes.
\begin{proposition}
	\label{prop:avtableauxk}
	For every $3\leq k\leq n$ we have
	\begin{enumerate}[(i)]
		\item $|\av{n}{\te{k}}|=|\av{n}{\tf{k}}|=|\av{n}{\tg{k}}|=|\tnk{n+1}{k}|$, 
		\item $|\av{n}{\tv{k}}|=|\av{n}{\tor{k}}|=
		\sum\limits_{\ell=0}^{k-1}\frac{n-\ell+1}{n+1}\binom{n+\ell}{\ell}+|\mathcal{F}(n,n,k-1)|+
		\sum\limits_{\ell=k}^{n-1}\sum\limits_{h=0}^{k-1}\binom{n-\ell+h-1}{h}|\mathcal{F}(\ell-h,\ell,k-1)|,$
	\end{enumerate} 
	\medskip 
	where $|\tnk{n}{k}|$ is the number  of Dyck paths in $\dP{n}$ with 
	height at most $k$ and  $|\mathcal{F}(m,n,k)|$  is the number of lattice 
	paths  { from $(0,0)$ to 
		$(n,m)$ with steps $(1,0)$ and $(0,1)$ such that all points $(x,y)$ 
		visited by the path satisfy $x\leq y\leq x+ k$.}
\end{proposition}

\begin{remark}
	The numbers $|\mathcal{F}(m,n,k)|$ can be computed efficiently using the 
	following formula,
	which is a consequence of the iterated reflection principle \cite[Chap.~1.3, Thm.~2]{Mohanty}.
	$$ |{\mathcal F}(m,n,k)|
	=\sum_{i=0}^{\infty}\tfrac{n-m+2i(k+2)+1}{n+i(k+2)+1}\tbinom{m+n}{m-i(k+2)}
	+\sum_{i=1}^{\infty}\tfrac{n-m-2i(k+2)+1}{m+i(k+2)}\tbinom{m+n}{m+i(k+2)}.$$
\end{remark}
To prove  Proposition~\ref{prop:avtableauxk}  we need the following lemma.  
\begin{lemma}
	\label{lem:avk}
	Let $k\geq 3$ and $T\in \shiT{n}$. Then
	\begin{enumerate}[(i)]
		\item $T$  avoids   $\tf{k}$ if and only if  its bounce path 
		has at most $k$ return points,
		\item $T$  avoids $\te{k}$ if and only if its height is at most  $k$,
		\item $T$ avoids $\tg{k}$ if and only if it has no valley at 
		height     $\geq{k-1}$.
		\item   Let  $\ell$ denote the number of the 
		non-empty rows of $T$. 
		Then $T\in\av{n}{\tv{k}}$ if and only if the tableau $T'\in\shiT{\ell-1}$ 
		obtained    by deleting the first $n+1-\ell$ columns of $T$ avoids $\te{k-1}.$
	\end{enumerate} 
\end{lemma} 
\newcommand{\wud}{w_{\scalebox{0.6}{\U\D}}}
\newcommand{\wdu}{w_{\scalebox{0.6}{\D\U}}}
\begin{proof}
	To prove (i) we use induction on $k$, the base case $k=2$ being Lemma 
	\ref{lem:avtableaux2}(i).
	%
	Thus assume that the claim holds for $k$ and consider the case $k+1$.
	Suppose on the contrary that there exists some $T\in\shiT{n}$
	containing $\tf{k+1}$ and  whose bounce path  has  $k+1$  return points.
	Then $n>k+1$ and there there exist a non-empty sequence $\de{i_1}{*},\ldots,\de{i_\ell}{*}$ of bounce deletions such that $\de{i_1}{*}\cdots \de{i_\ell}{*}(T)=\tf{k+1}$.
	Applying one more bounce deletion we have $\de{i_0}{*}\, \de{i_1}{*}\cdots
	\de{i_\ell}{*}(T)=\tf{k}$.
	In particular the tableau $\de{i_\ell}{*}(T)$ contains $\tf{k}$.
	However, the deletion $\de{i_\ell}{*}$ can only possibly reduce the number of return points of the bounce path by one.
	Thus the bounce path of $\de{i_\ell}{*}(T)$ has at least $k$
	return points and, by the induction hypothesis, $\de{i_\ell}{*}(T)$
	avoids $\tf{k}$.
	This is a contradiction.
	
	For the other direction, we show that any $T$ whose bounce path has at least $k+1$
	return points must contain $\tf{k}$.
	Indeed, this can be done by  applying $\de{i}{i}$
	on each row $i>0$ that does not contain a bounce step (see Figure \ref{fig:avoid_full2}b).

	For (ii), we use the fact that  $T$ has height at most $k$ if and only if 
	no row has more than $k-1$ empty boxes. In view of this, it is immediate to see that if a row of $T$ has more than $k-1$ empty boxes, then 
	$T$ contains $\te{k}$.
	{ Reversely,  if all rows of $T$ have at most $k-2$ empty boxes,  
		then   no iteration of bounce  deletions can  lead  to a tableau having a row with 
		$k$ or more  empty boxes, and thus to $\te{k}$. }

	To prove (iii), we use the fact that  $T$ has no valley at 
	height $\geq{k-1}$
	if and only if    each initial subword  of the form $w\D\U$
	satisfies    $\#\U(w\D\U)-\#\D(w\D\U)\leq{k}$. Next, notice that each bounce deletion 
	either deletes a $\U,\D$  or a single $\U$(the corresponding $\D$ being on a subsequent position) from each such initial subword $w\D\U$. Thus, the above inequality   still holds after any bounce deletion, which futher implies that 
	$T$ cannot contain the pattern $\tg{k}$. Reversely, if $T$ has a valley at height $\geq{k-1}$
	then, deleting all rows and columns above and to the left of this valley we get a tableau  $\tg{\kappa}$ with $\kappa\geq{k}$ which, clearly,  contains $\tg{k}$.

	Finally,  notice that (iv) is trivially true for $\ell<k$
	since  both assumptions that  $T$ avoids ${\tv{k}}$ and $T'$ avoids ${\te{k-1}}$ are satisfied (by size restrictions). 
	If $\ell\geq{k}$,  assume first that $T'$ contains $\te{k-1}$. 
	Then, it is easy to see that there exist bounce deletions on $T$ leading to a Shi tableau $T''$ which is $T'$ with a full column adjoint on its left. 
	Since $T'$ contains $\te{k-1}$ we conclude that $T''$ contains $\tv{k}$ and consequently that $T$ contains $\tv{k}$. 
	Next, assume that  $\ell\geq{k}$  and $T'$ avoids $\te{k-1}$. 
	In order to answer the existence (or not) of the pattern  $\tv{k}=\U\D\U^k\D^k$
	in $T$ and since $T$ begins with $\U^{n-\ell+1}\D$, we have to delete $n-\ell$ among the first $n-\ell+1$  $\U$-steps of $T$
	in a way that the resulting tableaux $T''$ has a full first column,\,i.e., begins with
	$\U\D$. 
	This will delete $n-\ell$ among the first $n-\ell+1$ columns of $T$ so that 
	$T''$   is $T'$ with a full column adjoint on its left.
	Then, it is clear that since the pattern $\te{k-1}$ does not occur in $T'$, 
	the pattern $\tv{k}=\U\D\U^k\D^k$ does not occur in 
	$T''$ and hence in $T$.
	
	%
	%
	%
\end{proof}

Now we are ready to prove Proposition~\ref{prop:avtableauxk}
\begin{proof}[Proof of Proposition~\ref{prop:avtableauxk}]
	(i) The fact that $\tf{k},\te{k}$ are Wilf-equivalent
	is a consequence of the well-known  {\em zeta map}
	\cite{akop-adnil-02,hag-dh-08}.
	Consider a Dyck path $\pa\in\Dp{n+1}$ with area vector 
	$a(\pa)=(a_1,a_2,\ldots,a_{n+1})$.
	That is, $a_i$ denotes the number of empty boxes in row $i$ of the corresponding Shi tableau.
	For each $j=0,\ldots,n$ we define $w_j$ as a word in the
	alphabet $\{\U,\D\}$ obtained as follows: reading $a(\pa)$ from left to right,
	we draw a down-step whenever we encounter an entry $a_i=j-1$
	and a up-step whenever we encounter an entry $a_i=j$.
	Finally, we set $\zeta(\pa)=w_0w_1\cdots{w}_n$.
	It is easy to see that $\zeta(\pa)\in\dP{n+1}$ since every entry of
	$a(\pa)$ contributes twice: an up-step first and a down-step later.
	%
	The map $\zeta:\dP{n+1}\rightarrow\dP{n+1}$ is a bijection with surprising 
	properties.
	The following are equivalent:
	\begin{enumerate}[(a)]
		\item The area vector $a(\pa)$ satisfies $a_i\leq{k-1}$ for all $i$.
		\item The word $w_j$ is empty for all $j>k$.
		\item The bounce path of $\zeta(\pa)$ has at most $k$ return points.
	\end{enumerate}
	Notice that a Dyck path has height at most $k$
	if and only if all entries of its area vector are less than $k$.
	The equality $|\av{n}{\te{k}}|=|\av{n}{\tf{k}}|=|\tnk{n+1}{k}|$ therefore
	follows from the characterization of the Shi tableaux in $\av{n}{\te{k}}$
	and $\av{n}{\tf{k}}$ in Lemma~\ref{lem:avk}~(i) and~(ii) and
	the equivalence of (a) and (c) above.
	
	It remains to prove that $|\av{n}{\tg{k}}|=|\tnk{n+1}{k}|$.
	In view of Lemma~\ref{lem:avk}~(iii) suppose that $\pa$ is a Dyck path 
	with no valleys at height greater or equal to $k-1$. Then the steps of $\pa$ above the line 
	$y=x+k-1$   form a (possibly empty) set of disconnected peaks.
	Replacing each such peak of the form $\U^\ell \D^\ell$ by a sequence
	$(\U\D)^\ell$ we obtain a Dyck path of height at most $k$.
	This yields a bijection between $\tg{k}$-avoiding and $\te{k}$-avoiding
	Shi tableaux.
	

	To prove (ii), let $T\in\av{n}{\tv{k}}$ with $\ell$ non-empty rows.
	By Lemma~\ref{lem:avk}~(iv) the tableau $T$ falls in one of three categories:
	(a) $\ell<k$, (b) $\ell=n$ and the tableau $T'\in\shiT{n-1}$ obtained by deleting the 
	first column of $T$ avoids $\te{k-1}$, or (c) $k\leq \ell<n$ and the tableau 
	$T'\in\shiT{\ell-1}$ obtained by deleting the first $n+1-\ell$ columns of $T$ avoids 
	$\te{k-1}$.
	The set of Shi tableaux that satisfy $\ell<k$ corresponds naturally
	to ballot paths with $n$ $\U$-steps and $\ell$ $\D$-steps, that is, 
	lattice paths  from $(0,0)$ to 
	$(\ell,n)$ with steps $(1,0)$ and $(0,1)$ that never go below the line $x=y$. 
	Such ballot paths are known to be counted by $\frac{n-\ell+1}{n+\ell+1}\binom{n+\ell+1}{\ell}=\frac{n-\ell+1}{n+1}\binom{n+\ell}{\ell}$ (see \cite[Relation (10.10)]{Krat-LPE}).
	The second category contains $|\tnk{n}{k-1}|=|{\mathcal F}(n,n,k-1)|$ elements by 
	part (i) of the proposition.
	It remains to count the tableaux that fall into the third category.
	To this end write $T=\U^{n+1-\ell}\D \pa \D \pa'$, where $\pa$ contains $n-\ell-1$ 
	$\D$-steps, and let $h$ denote the number of $\U$-steps of $\pa$.
	Thus the tableau $T'$ corresponds to the path $\U^h\pa'$.
	In order for $T'$ to avoid $\te{k-1}$ we must have $h\leq k-1$ and $\pa'$ must correspond to a ballot path from $(0,0)$ to $(\ell-h,\ell)$ of height at most $k-1$.
	There are $|\mathcal F(\ell-h,\ell,k-1)|$ such ballot paths.
	The fact that $T'$ avoids $\te{k-1}$ imposes no restriction on the path $\pa$.
	Thus there are $\binom{n-\ell+h-1}{h}$ possible choices for $\pa$.
\end{proof}

\begin{remark}
	The zeta map and the bounce path of a Dyck path were first defined by Andrews 
	et al.~in \cite{akop-adnil-02} in order to enumerate $ad$-nilpotent ideals in 
	a borel subalgebra of the complex simple Lie algebra of type $A$ with class 
	of nilpotence less than $k$.
	These ideals are in bijection with certain Dyck paths 
	(\cite[Thm.~4.1]{akop-adnil-02}), which correspond precisely to the Shi 
	tableaux that avoid $\tf{k}$.
	The zeta map provides a bijection to Dyck paths of height at most $k$ (see 
	\cite[\textsection5]{akop-adnil-02}), which of course coincide with 
	$\te{k}$-avoiding tableaux.
	Thus our use of the zeta map in the proof of \ref{prop:avtableauxk}~(i) 
	really agrees with its original purpose.
	The combinatorial description of $ad$-nilpotent ideals with bounded class of nilpotence in terms of pattern avoidance is, in the opinion of the authors, quite intriguing.
	
	We further remark that the zeta map also ties to different problems in 
	algebraic combinatorics such as the Hilbert series of diagonal 
	harmonics~\cite{hag-ha-13,hag-dh-08} 	 and  the $q,t$-Catalan Positivity Conjecture   \cite{ha-ga02}.
	The zeta map and the bounce path have also previously been considered in 
	conjuction with the Shi arrangement, for example, in \cite{arm-hadh-02}.
\end{remark}

  \section{Open Problems}
  \label{sec:problems}
  
  Once a notion of pattern avoidance for a class of combinatorial objects has been defined there are plenty of interesting questions which have been studied for permutation patterns and can be transferred to the new setting.
  We mention here only {three} such (potentially difficult) problems and conclude the paper with two more possible further directions that are more particular to the case of Shi tableaux.
  
  \begin{problem}
  	Does the poset $\cT$ of Shi tableaux contain infinite antichains?
  \end{problem}
  
  The analogous question can be answered affirmatively for permutations ordered by pattern containment~\cite{SpielmanBona}.
  On the other hand, the poset of finite words over a finite alphabet is an example of a poset with no infinite antichains~\cite{Higman}.
  \medskip

  Another interesting direction is to study the following questions. 
  \begin{problem}
  	Let $T\in\cT$ be a Shi tableau (or a collection of Shi tableaux).
  	What can be said about the generating function $\sum_{n\in\NN}\av{n}{T} x^n$?
  	Are all formal power series obtained in this way rational, algebraic, D-finite,$\dots$?
  \end{problem}
  
  We have no strong intuition as to what the correct answer should be.
  Pattern avoidance in permutations gives rise to (conjecturally) very complicated generating functions~\cite{Albert-etal}.
  However, there is at least a chance that the poset of Shi tableaux is sufficiently simpler such some results can be obtained.  
  \medskip

  	Finally, since $\cT$ is a (graded) infinite poset, in order to further  study  its structure one should focus on the enumerative combinatorics of its intervals
  	(or, at least, certain well-behaved intervals).
  	This is the content of the 	work in \cite{bernini2019enumerative}
  	where the authors  study  {\em their}  Dyck pattern poset. 
  	Motivated by this, we pose  analogous questions for our Dyck  pattern poset.
  	    \begin{problem}
 Is it possible to say something  about the enumerative combinatorics or the M\"{o}bius 
 function of some particular intervals in the poset of Shi tableaux? 
  \end{problem}

  \subsection{Connections to pattern-avoiding permutations}
  Considering our first enumerative results on pattern avoidance in Shi tableaux, we observe that the sequences in Proposition~\ref{prop:avtableaux2} and Proposition~\ref{prop:avtableauxk}~(i) have already appeared in the literature in the context of permutation-patterns. The following table summarizes these results.
  
  \bigskip
  \begin{center}
  	\begin{tabular}{|l|c|l|l|}
  		\hline
  		Shi tableaux of type $A$ & Sequence & OEIS~\cite{oeis} & Pairs of permutations \\
  		\hline
  		$|\av{n}{\te{}}|,|\av{n}{\tf{}}|$ & $2^n$ & \oeis{A000079} & $|\av{n+1}{132,123}|$ \\
  		\hline
  		$|\av{n}{\tg{}}|,|\av{n}{\tv{}}|,|\av{n}{\tor{}}|$ & $\binom{n}{2}+n+1$ & \oeis{A000124} & $|\av{n+1}{132,321}|$\\
  		\hline
  		$|\av{n}{\te{k}}|,|\av{n}{\tf{k}}|,|\av{n}{\tg{k}}|$ &$|\tnk{n+1}{k}|$ &
  		\oeis{A080934} &
  		$|\av{n+1}{132, 12\dots k}|$\\
  		\hline
  	\end{tabular}
  \end{center}
  
  \bigskip
  Using Proposition~\ref{prop:avtableauxk}~(ii) we compute the first few values of $|\av{n}{\tv{k}}|$.
  For $k=2$ we find the sequence of $\{132,321\}$-avoiding permutations \cite[\oeis{A000124}]{oeis}.
  For $k=3$ the numbers seem to match the sequence of $\{123,3241\}$-avoiding permutations \cite[\oeis{A116702}]{oeis}.
  For $k=4$ we apparently obtain the sequence of $\{123,51432\}$-avoiding permutations \cite[\oeis{A116847}]{oeis}.
  The first few values for $k=5$ and $k=6$ are as follows:
  \begin{align*}
  k&=5: &&
  1, 2, 5, 14, 42, 131, 413, 1294, 4007, 12272, 37277, 112622, 339152, 1019457,\dots\\
  k&=6: &&
  1, 2, 5, 14, 42, 132, 428, 1411, 4675, 15463, 50928, 166999, 545682, 1778631,\dots,
  \end{align*}
  and are not part of the oeis yet.
  The above data might lead to the guess that there should exist also patterns $\sigma$ and $\tau$ of lengths three and $k+1$ respectively such that $|\av{n}{\tv{k}}|$ equals the number of $\{\sigma,\tau\}$-avoiding permutations.
  
  \begin{problem}
  	Find an explanation for (or quantify) the above phenomenon linking pattern avoidance in Shi tableaux of type $A$ to permutations avoiding a pair of patterns.
  \end{problem}
  
  Considering the fact that Shi tableaux  are counted by Catalan numbers and are hence in bijection with $\sigma$-avoiding permutations for any pattern $\sigma$ of length three, one might hope to explain this enumerative parallel by finding a bijection that translates the bounce deletions into pattern containment on the permutation side.
  Unfortunately such an attempt is bound to fail:
  There is no choice of pattern $\sigma$ of length three for which the poset of Shi tableaux is isomorphic to the poset of $\sigma$-avoiding permutations.
  This can be seen quickly for example from the number of cover relations in Figure~\ref{fig:all_the_poset}.
  
  %
  
  \subsection{\texorpdfstring{$ad$}{ad}-nilpotent ideals}
  
  A Shi arrangement can be attached to any crystallographic root system.
  In each case the dominant regions are indexed by Shi tableaux -- binary fillings whose shape is determined by the root poset.
  In types $B$ and $C$ this shape is a doubled staircase as opposed to the staircase shape in type $A$.
  In type $D$ the root poset is no longer planar but with some adjustments the combinatorics can still be made to work.
  
  \begin{problem}
  	Find and investigate the cover relations that define patterns in Shi tableaux for other classical types.
  \end{problem}
  
  There are two facts that should serve as guidelines as well as encouragement in this endeavor.
  First, the Shi arrangement is an exponential sequence of arrangements for the classical types.
  Hence there is geometric motivation for (at least some of) the bounce deletions.
  
  Secondly, there are enumerative results on $ad$-nilpotent ideals in a Borel subalgebra of the complex simple Lie algebra of classical type with bounded class of nilpotence~\cite{KrattenthalerOrsinaPapi}.
  In type $C$ the ideals of class of nilpotence $k$ are counted by formulas depending on the parity of $k$.
  At least in one of the two cases there seems to be an immediate connection between such ideals and the Shi tableaux avoiding a full Shi tableau.
  \medskip

\subsection*{Acknowledments}
We would  like to thank the anonymous referees for their
careful reading and  valuable comments. 
\subsection*{Data Availability Statement} No datasets were generated or analysed during the current study.

\bibliographystyle{plain}
\bibliography{bibliography}

\end{document}